\documentclass[reqno, 11pt]{amsart}

\usepackage[utf8]{inputenx} 
\usepackage[T1]{fontenc}    
\usepackage[a4paper,hmargin={2.7cm,2.7cm},vmargin={3.3cm,3.3cm}]{geometry}
\usepackage{amssymb}   
\usepackage{amsthm}    
\usepackage{thmtools}  
\usepackage{mathtools} 
\usepackage{mathrsfs}  
\usepackage{commath}
\usepackage{bbm}
\usepackage{enumerate}

\usepackage[all]{xy}
\usepackage{todonotes}

\usepackage{xcolor}
\newcommand\numberthis{\addtocounter{equation}{1}\tag{\theequation}}

\usepackage[noadjust]{cite}

\usepackage{varioref}
\usepackage{hyperref}
\usepackage[nameinlink, capitalize, noabbrev]{cleveref}

\declaretheorem[style = plain, numberwithin = section]{theorem}
\declaretheorem[style = plain,      sibling = theorem]{corollary}
\declaretheorem[style = plain,      sibling = theorem]{lemma}
\declaretheorem[style = plain,      sibling = theorem]{proposition}
\declaretheorem[style = definition, sibling = theorem]{definition}

\declaretheorem[style = definition, sibling = theorem]{example}

\declaretheorem[style = remark,     sibling = theorem]{remark}

\numberwithin{equation}{section}

\DeclareMathOperator{\vol}{vol}
\DeclareMathOperator{\spn}{span}

\newcommand{\B}{\mathcal{B}}

\newcommand{\N}{\mathbb{N}}   
\newcommand{\Z}{\mathbb{Z}}   
\newcommand{\R}{\mathbb{R}}   
\newcommand{\C}{\mathbb{C}}   
\newcommand{\T}{\mathbb{T}}   

\def\lhs#1#2{{_\bullet\!}\langle #1,#2\rangle}

\DeclareMathOperator{\Img}{Im}
\newcommand{\E}{\mathcal{E}}
\newcommand{\F}{\mathcal{F}}

\newcommand{\lfb}{C_1}
\newcommand{\ufb}{C_2}

\newcommand{\veco}{f}
\newcommand{\vect}{g}
\DeclareMathOperator{\vN}{L}

\newcommand{\Hi}{\mathcal{H}}
\newcommand{\Hip}{\mathcal{H}_{\pi}}
\newcommand{\Hpi}{\mathcal{H}_{\pi}}
\newcommand{\Hik}{\Hpi^{1,{\alpha}}}

\newcommand{\pker}{P_{\pi}}
\newcommand{\Hs}{\Hpi^{\infty}}

\newcommand{\SIP}{SI/\pker}
\def\module#1#2{\mathcal{E}_{#1,#2}}

\DeclareMathOperator{\id}{id}

\newcommand{\gramian}{\mathscr{G}}
\newcommand{\synthesis}{\mathscr{D}}
\newcommand{\frameop}{\mathscr{S}}
\newcommand{\analysis}{\mathscr{C}}

\title[Lattice orbits of nilpotent groups and strict comparison]{Smooth lattice orbits of nilpotent groups and strict comparison of projections}

\author{Erik Bédos}

\address{Department of Mathematics, University of Oslo, PB1053 Blindern, 0316 Oslo, Norway.}
\email{bedos@math.uio.no}

\author{Ulrik Enstad}
\address{Department of Mathematics,
Stockholm University,
SE-106 91 Stockholm, Sweden.}
\email{ulrik.enstad@math.su.se}

\author{Jordy Timo van Velthoven}
\address{Delft University of Technology,
Faculty EECMS/DIAM,
Mekelweg 4, Building 36,
2628 CD Delft, The Netherlands.}
\email{j.t.vanvelthoven@tudelft.nl}

\makeatletter
\@namedef{subjclassname@2020}{\textup{2020} Mathematics Subject Classification}
\makeatother



\begin{document}

\maketitle

\begin{abstract}
This paper provides sufficient density conditions for the existence of smooth vectors generating a frame or Riesz sequence in the lattice orbit of a square-integrable projective representation of a nilpotent Lie group. The conditions involve the product of lattice co-volume and formal dimension, and complement Balian--Low type theorems for the non-existence of smooth frames and Riesz sequences at the critical density.
The proof hinges on a connection between smooth lattice orbits and generators for an explicitly constructed finitely generated Hilbert $C^*$-module. 
An important ingredient in the approach is that twisted group $C^*$-algebras associated to finitely generated nilpotent groups have finite decomposition rank, hence finite nuclear dimension, which allows us to deduce that any matrix algebra over such a simple $C^*$-algebra has strict comparison of projections.
\end{abstract}

\section{Introduction}

Let $G$ be a nilpotent Lie group and let $(\pi, \Hpi)$ be an irreducible, square-integrable projective representation of $G$. 
For a lattice $\Gamma \leq G$, consider the orbit of $\pi$ under a vector $\vect \in \Hpi$,
\begin{align} \label{eq:coherent}
\pi (\Gamma) \vect = (\pi(\gamma) \vect )_{\gamma \in \Gamma}. 
\end{align}
The aim of this paper is to study the existence of a vector $\vect \in \Hpi$ such that $\pi(\Gamma) \vect$ forms a frame or Riesz sequence (Riesz basis for its span) in $\Hpi$, that is, $\pi(\Gamma) g$ satisfies the \emph{frame inequalities}
\begin{align} \label{eq:frame_ineq_intro}
A \|f \|_{\Hpi}^2 \leq \sum_{\gamma \in \Gamma} |\langle f, \pi(\gamma) g \rangle |^2 \leq B \| f \|_{\Hpi}^2, \quad f \in \Hpi,
\end{align}
for constants $0< A \leq B < \infty$, or the \emph{Riesz inequalities}
\begin{align} \label{eq:riesz_ineq_intro}
A \| c \|_{\ell^2}^2 \leq \bigg\| \sum_{\gamma \in \Gamma} c_{\gamma} \pi(\gamma) g \bigg\|_{\Hpi}^2 \leq B \| c \|_{\ell^2}^2, \quad c \in \ell^2 (\Gamma).
\end{align}
A particular focus will be on the existence of frames and Riesz sequences $\pi (\Gamma) \vect$ for which the associated diagonal matrix coefficient function $C_g g : G \to \mathbb{C}$, defined by
\begin{align} \label{eq:matrix_coeff_intro}
    C_g g (x) = \langle g, \pi(x) g \rangle, \quad x \in G,
\end{align}
possesses an additional form of localization, e.g., smoothness or $L^1$-integrability. 

Frames and Riesz sequences are classical notions in various areas of complex and harmonic analysis, 
and play an important role in the applications of these areas as they provide stable and unconditionally convergent Hilbert space expansions. More modern variants of these notions have also been studied in the setting of operator theory and operator algebras, most notably in Hilbert $C^*$-modules, where they give rise to projections in associated $C^*$-algebras.

In this paper the existence of localized frames and Riesz sequences of the form \eqref{eq:coherent} will be studied via a correspondence to projections in an associated twisted group $C^*$-algebra. 
It turns out that recent results on $C^*$-algebras (in particular, group $C^*$-algebras) provide adequate tools that are capable of treating localization properties in the existence problem. 

Before formulating the main results and describing the methods used in their proof, the requisite background and context will be sketched. 

\subsection{Background and context}
A first fundamental obstruction to the existence of frames and Riesz sequences of lattice orbits $\pi(\Gamma)g$ 
is provided by the density theorem, relating the lattice co-volume $\vol(G/\Gamma)$ or its reciprocal (the so-called \textquotedblleft density\textquotedblright) and the formal dimension $d_{\pi} > 0$ of $\pi$. 
Different versions of this theorem can be found in, e.g., \cite{rieffel1981von, bekka2004square, romero2020density}.

\begin{theorem} \label{thm:density_intro}
Let $G$ be a nilpotent Lie group with a lattice $\Gamma \leq G$. Let $(\pi, \Hpi)$ be an irreducible, square-integrable projective representation of $G$ of formal dimension $d_{\pi} > 0$. 
\begin{enumerate}[(i)]
\item If there exists $\vect \in \Hpi$ such that $\pi (\Gamma) \vect$ forms a frame, 
 then $\vol(G/\Gamma) d_{\pi} \leq 1$.
\item If there exists $\vect \in \Hpi$ such that $\pi (\Gamma) \vect$ forms a Riesz sequence, then $\vol(G/\Gamma) d_{\pi} \geq 1$.
\end{enumerate}
$($The value $\vol(G/\Gamma) d_{\pi}$ is independent of the choice of Haar measure on $G$.$)$
\end{theorem}

Theorem \ref{thm:density_intro} provides a critical density for a lattice to admit a frame or Riesz sequence in its orbit. In particular, a lattice admitting an orthonormal basis must have the critical density $\vol(G/\Gamma) = d_{\pi}^{-1}$. Necessary conditions of this type are commonly referred to as \emph{density conditions} and can also be obtained for discrete index sets that do not necessarily form a group, see, e.g., \cite{fuehr2017density, mitkovski2020density}.

There is a converse to Theorem \ref{thm:density_intro} for irreducible representations of a nilpotent Lie group $N$ that are square-integrable modulo the center $Z = Z(N)$.
Representations of this type can be treated as projective representations of $G = N / Z$, so-called \emph{projective relative discrete series representations} (see Section \ref{sec:relativeDS}). The following result can be derived from \cite{bekka2004square, En21} by combining \cite[Theorem 1.3]{En21} and the arguments underlying \cite[Theorem 3]{bekka2004square}.

\begin{theorem} \label{thm:density_intro2}
Let $G$ be a connected, simply connected nilpotent Lie group with a lattice $\Gamma \leq G$. Let $(\pi, \Hpi)$ be a projective relative discrete series representation of $G$ of formal dimension $d_{\pi} > 0$.
\begin{enumerate}[(i)]
    \item If $\vol(G/\Gamma) d_{\pi} \leq 1 $, then there exists $\vect \in \Hpi$ such that $\pi(\Gamma) \vect$ forms a frame. 
    \item If $\vol(G/\Gamma) d_{\pi} \geq 1$, then there exists $\vect \in \Hpi$ such that $\pi(\Gamma) \vect$ forms a Riesz sequence. 
\end{enumerate}
\end{theorem}

Together, Theorem \ref{thm:density_intro} and Theorem \ref{thm:density_intro2} provide a dichotomy that completely describes the reproducing properties (frame and Riesz sequence) of lattice orbits of square-integrable representations in terms of the lattice co-volume or density.
The existence claims in Theorem \ref{thm:density_intro2} rely on techniques for von Neumann algebras and are not accompanied by explicit constructions.
For more specific representations and lattices acting via two group actions, special cases of the existence claims in Theorem \ref{thm:density_intro2} can also be obtained via tiling arguments \cite{han2001lattice, dutkay2013on}, in which case the generating vector can be chosen to be an indicator function of a common fundamental domain. For historical expositions on the density theorem in time-frequency analysis, see \cite{heil2007history, folland2006abstruse}.

For vectors $g \in \Hpi$ possessing certain localization properties (i.e., a smooth or integrable vector), a second obstruction to the existence of frames and Riesz sequences of the form $\pi(\Gamma) g$ is given by the strictness of the density conditions in Theorem \ref{thm:density_intro}. For the Euclidean plane $G = \mathbb{R}^2$ and its projective Schr\"odinger representation $(\pi, L^2 (\mathbb{R}))$, the fundamental Balian--Low theorem in time-frequency analysis asserts that there exists no orthonormal basis (or Riesz basis) of the form $\pi (\Gamma) g$ for a Schwartz function $g \in \mathcal{S} (\mathbb{R})$,  \cite{daubechies1990wavelet, benedetto1998gabor}. Alternatively, for a Schwartz function, the associated density inequalities in Theorem
\ref{thm:density_intro} are strict \cite{grochenig2015deformation, feichtinger2004varying, ascensi2014dilation}. Balian--Low type theorems for (classes of) nilpotent groups have been obtained in \cite{jakobsen2020deformation, grochenig2020balian} and show that the inequalities in Theorem \ref{thm:density_intro} are strict for integrable vectors. It should be mentioned that (non-localized) orthonormal bases  in the orbit of a nilpotent Lie group could still exist by Theorem \ref{thm:density_intro2}, and even for nilpotent Lie groups not admitting a lattice, cf.~\cite{grochenig2018orthonormal, oussa2019compactly}.

A key problem in time-frequency or phase-space analysis \cite{groechenig2001foundations, folland1989harmonic} is the existence of smooth frames (resp.\ Riesz sequences) $\pi(\Gamma) g$ for a given lattice $\Gamma \leq G$ with super-critical (resp.\ sub-critical) density. While the mere existence of such frames and Riesz sequences are well-known for lattices possessing a qualitative \textquotedblleft covering density\textquotedblright     \,\cite{groechenig1991describing}, there are currently no quantitative results that match the necessary conditions provided by Theorem \ref{thm:density_intro}, except for the specific setting of the Heisenberg group. 
Indeed, for $G = \mathbb{R}^2$ and its projective 
Schr\"odinger representation $(\pi, L^2 (\mathbb{R}))$ (for which $d_{\pi} = 1)$, the density theorems for sampling and interpolation in Bargmann-Fock spaces  \cite{seip1992density, seip1992density2, lyubarskij1992frames} can be recast as the Gaussian Gabor system $\pi(\Gamma) \vect$ with $\vect(t) = e^{-\pi t^2}$ forming a frame (resp.\ Riesz sequence) for $L^2 (\mathbb{R})$ if and only if $\vol(\mathbb{R}^2 / \Gamma) < 1$ (resp.\ $\vol(\mathbb{R}^2 / \Gamma) > 1$), see also  \cite{grochenig2018sampling, grochenig2013gabor, janssen1994signal}.
Although the frame and Riesz property of a multivariate Gaussian Gabor system cannot be simply described in terms of a density condition
\cite{pfander2013remarks, grochenig2011multivariate}, it is still expected \cite[Remark 2] {grochenig2011multivariate} that 
Gabor frames (resp.\ Riesz sequences) $\pi(\Gamma) g$ with arbitrary smooth window $g \in L^2 (\mathbb{R}^d)$ exist for any lattice $\Gamma \leq \mathbb{R}^{2d}$ satisfying $\vol(\mathbb{R}^{2d} / \Gamma) < 1$ (resp.\ $\vol(\mathbb{R}^{2d} / \Gamma) > 1)$, see also \cite{heil2007history, pfander2012geometric}. Only recently has there been a first contribution \cite{jakobsen2021duality} to this existence problem for Gabor frames in higher dimensions, namely for so-called non-rational lattices  $\Gamma \leq \mathbb{R}^{2d}$, by exploiting the structural results on (irrational) non-commutative tori \cite{Ri88} and its link with Gabor frames \cite{Lu09}; 
see Section \ref{sec:heisenberg} for a more detailed discussion.

\subsection{Main results}
Our main result concerns the existence of frames and Riesz sequences generated by smooth vectors, i.e., vectors $g \in \Hpi$ for which the orbit maps $x \mapsto \pi(x) g$ are smooth; in notation, $g \in \Hs$. 
The result relies on a compatibility condition between the $2$-cocycle $\sigma$ of the projective representation $\pi$ and the lattice $\Gamma$, known as \textquotedblleft Kleppner's condition\textquotedblright; see \cite{Kle62,Om15,P,Om14,BO}. A pair $(\Gamma, \sigma)$ satisfies \emph{Kleppner's condition} if, for any non-trivial $\gamma \in \Gamma$ satisfying $\sigma(\gamma, \gamma') = \sigma(\gamma', \gamma)$ for all $\gamma' \in \Gamma$ such that $\gamma' \gamma = \gamma \gamma'$, the associated conjugacy class $\{ (\gamma')^{-1} \gamma \gamma' : \gamma' \in \Gamma \}$ is infinite.

The following theorem is 
a special case of our main theorem (Theorem \ref{thm:existence}).

\begin{theorem} \label{thm:main_intro}
Let $(\pi, \Hpi)$ be a $\sigma$-projective relative discrete series representation of a connected, simply connected nilpotent Lie group $G$ of formal dimension $d_{\pi} > 0$. Suppose that $\Gamma \leq G$ is a lattice 
such that $(\Gamma, \sigma)$ satisfies Kleppner's condition.
\begin{enumerate}[(i)]
    \item If $\vol(G/\Gamma) d_{\pi} < 1 $, then there exists $\vect \in \Hs$ such that $\pi(\Gamma) \vect$ forms a frame. 
    \item If $\vol(G/\Gamma) d_{\pi} > 1$, then there exists $\vect \in \Hs$ such that $\pi(\Gamma) \vect$ forms a Riesz sequence.
\end{enumerate}
\end{theorem}

Under Kleppner's condition, Theorem \ref{thm:main_intro} provides a full converse to the Balian--Low type theorems 
\cite{grochenig2020balian, grochenig2015deformation, feichtinger2004varying, ascensi2014dilation} on the strictness of the density conditions (Theorem \ref{thm:density_intro}) for smooth vectors. In fact, as a direct consequence of Theorem \ref{thm:main_intro}, the smooth vectors in Theorem \ref{thm:main_intro} could even be chosen to be analytic (cf.~Corollary \ref{cor:analytic}).
A more general version of Theorem \ref{thm:main_intro}, valid for projective representations arising from genuine representations that are merely square-integrable modulo their projective kernel, is given in Theorem \ref{thm:existence}.

It is currently not known whether 
the existence claims (i) and (ii) in Theorem \ref{thm:main_intro} also hold without the assumption of Kleppner's condition. The fact that Kleppner's condition is not needed in Theorem \ref{thm:density_intro2} and in a version of Theorem \ref{thm:main_intro} for the $3$-dimensional Heisenberg group indicates that it might be superfluous for Theorem \ref{thm:main_intro} in general, too.  

For applications to time-frequency analysis, we mention that the representations $(\pi, \Hpi)$ appearing in Theorem \ref{thm:main_intro} can, by Kirillov's orbit method, be realized to act on some $L^2 (\mathbb{R}^d)$, with the action of $\pi$ in a coordinate parametrization given by
\[
\pi(x) f(t) = e^{i P(x,t)} f(Q(x,t)), \quad t \in \mathbb{R}^d, \; x \in \mathbb{R}^n,
\]
for polynomials $P$ and $Q$. In such a realization, the space of smooth vectors $\Hs$ is precisely the space $\mathcal{S} (\mathbb{R}^d)$ of Schwartz functions, and the corresponding matrix coefficients define functions in $\mathcal{S} (\mathbb{R}^n)$.
Theorem \ref{thm:main_intro} provides therefore new classes of localized frames and Riesz sequences in $L^2 (\mathbb{R}^d)$. A key feature of such localized systems is that, via techniques underlying the theory of localized frames \cite{aldroubi2008slanted, balan2006density, romero2021dual, fornasier2005intrinsic, grochenig2004localization}, the reproducing properties of frames and Riesz sequences, namely
\[
f \in L^2 (\mathbb{R}^d) \quad  \text{if and only if} \quad
f = \sum_{\gamma \in \Gamma} c_{\gamma}  \pi(\gamma) g \quad \text{for some }(c_{\gamma})_{\gamma \in \Gamma} \in \ell^2 (\Gamma),
\]
respectively
\[
c \in \ell^2 (\Gamma) \quad \text{if and only if} \quad c_{\gamma} = \langle f, \pi(\gamma) g \rangle \quad \text{for some } f \in L^2 (\mathbb{R}^d),
\]
automatically extend to families of associated Banach spaces; in particular, so-called \emph{coorbit spaces} \cite{feichtinger1989banach}. Therefore, such localized systems provide a description and characterization of these Banach spaces and can be used for the purpose of (generalized) time-frequency analysis on $\mathbb{R}^d$; see \cite{grochenig2021new} for 
a concrete exposition associated to lower-dimensional nilpotent groups.

\subsection{Methods}
With notation as in Theorem \ref{thm:main_intro}, 
our proof method is based on the interpretation of a vector $\vect \in \Hpi$
defining a lattice orbit $\pi(\Gamma) g$ as an element of a module over an associated operator algebra. 
The relevant operator algebras are generated by the $\sigma$-twisted left regular representation $(\lambda_{\Gamma}^{\sigma}, \ell^2 (\Gamma))$ of $\Gamma$, determined by 
\[ \lambda_{\Gamma}^{\sigma}(\gamma) \delta_{\gamma'} = \sigma(\gamma,\gamma') \delta_{\gamma \gamma'} \;\;\; \text{for $\gamma,\gamma' \in \Gamma$,} \]
where $\{ \delta_{\gamma} : \gamma \in \Gamma \}$ is the canonical basis for $\ell^2(\Gamma)$. The completion of the span of the collection
\[ \lambda_{\Gamma}^{\sigma}(\Gamma) = \{ \lambda_{\Gamma}^{\sigma}(\gamma) : \gamma \in \Gamma \} \subseteq \B(\ell^2(\Gamma)) \]
in the strong operator topology gives the $\sigma$-twisted group von Neumann algebra $\vN(\Gamma,\sigma)$, while completion in the norm topology gives the (reduced) $\sigma$-twisted group $C^*$-algebra $C_r^*(\Gamma,\sigma)$. Since a lattice $\Gamma \leq G$ in a nilpotent Lie group $G$ is finitely generated and nilpotent, Kleppner's condition on $(\Gamma,\sigma)$ is  equivalent to the algebra $\vN(\Gamma,\sigma)$ (resp.\ $C_r^*(\Gamma,\sigma)$) being a factor (resp.\ simple), cf.~\cite{Kle62,P}. In addition, since $\Gamma$ is amenable, the reduced algebra $C_r^*(\Gamma,\sigma)$ is isomorphic to the \emph{full} twisted group $C^*$-algebra $C^*(\Gamma,\sigma)$. Our approach makes a fundamental use of the algebras $C_r^*(\Gamma, \sigma)$ and $C^*(\Gamma, \sigma)$ being simple. It should be mentioned that the non-twisted group $C^*$-algebra $C^* (\Gamma)$ (i.e., $\sigma$ being trivial) is simple if and only if $\Gamma$ is trivial, so that the use of cocycles is essential for our approach.

The question of the existence of a general vector $g \in \Hpi$ generating a frame $\pi(\Gamma)g$ in $\Hpi$ (see Theorem \ref{thm:density_intro2}) can be naturally approached using techniques for von Neumann algebras and their Hilbert modules, as shown by Bekka \cite{bekka2004square} (cf.~\cite{En21} for Riesz sequences).
The fundamental observation here is that there exists a frame (resp.\ Riesz sequence) of the form $\pi(\Gamma) \vect$ for some $\vect \in \Hip$ if and only if $\pi|_{\Gamma}$ is a subrepresentation of $\lambda_{\Gamma}^{\sigma}$ (resp.\ $\lambda_{\Gamma}^{\sigma}$ is a subrepresentation of $\pi|_{\Gamma}$), cf.~\cite[Corollary 3 and 4]{bekka2004square}. By the square-integrability of $\pi$, the restriction $\pi|_{\Gamma}$ can be extended to give $\Hip$ the structure of a Hilbert $\vN(\Gamma,\sigma)$-module, so that the existence of frames (resp.\ Riesz sequences) in $\Hip$ is equivalent to $\Hpi$ being a submodule of $\ell^2(\Gamma)$ (resp.\ $\ell^2(\Gamma)$ is a submodule of $\Hip$), cf.~\cite[Proposition 1]{bekka2004square} and  \cite[Theorem 5.1]{En21}. When $\vN(\Gamma,\sigma)$ is a factor, such submodule inclusions are in turn equivalent to inequalities involving the associated von Neumann dimensions, which gives rise to the inequalities $ \vol(G/\Gamma) d_{\pi} \leq 1$ (resp.\ $ \vol(G/\Gamma) d_{\pi} \geq 1$). The basis for these results is that projections in a II$_1$ factor (such as $\vN(\Gamma, \sigma))$ are completely classified by their value under the canonical tracial state by the Murray-von Neumann comparison theory.

For providing density conditions for the existence of a localized vector $\vect$ yielding a frame or Riesz sequence of the form $\pi(\Gamma) \vect$, 
the above mentioned von Neumann algebra techniques do not seem to be sufficient. In contrast, we show in the present paper that the theory of $C^*$-algebras and associated Hilbert $C^*$-modules do provide powerful techniques for approaching the localization problem. 
The following explicitly constructed Hilbert $C^*$-module plays a central role in the proof of Theorem \ref{thm:main_intro}.

\begin{theorem}\label{thm:module_intro}
Let $(\pi, \Hpi)$ be a $\sigma$-projective relative discrete series representation of a connected, simply connected, nilpotent Lie group $G$. Suppose $\Gamma \leq G$ is a lattice. Then the space $\Hip^\infty$ of smooth vectors can be completed into a finitely generated left Hilbert $C_r^*(\Gamma,\sigma)$-module $\module{\pi}{\Gamma}$, where the left action and the $C_r^*(\Gamma,\sigma)$-valued inner product $\lhs{\cdot}{\cdot}$ are determined by
\begin{align*}
    a \cdot \veco &= \sum_{\gamma \in \Gamma} a(\gamma) \pi(\gamma) \veco  \;\;\; \text{for $a \in \mathcal{S}(\Gamma)$ and $\veco \in \Hip^\infty$,} \\
    \lhs{\veco}{\vect}(\gamma) &= \langle \veco, \pi(\gamma) \vect \rangle \;\;\;\;\;\;\; \text{for $\veco, \vect \in \Hip^\infty$ and $\gamma \in \Gamma$,}
\end{align*}
where $\mathcal{S}(\Gamma) \subset \ell^1(\Gamma)$ denotes the Schwartz space on $\Gamma$, cf.\ \Cref{sec:lattice_orbits}.
\end{theorem}

A general method for the construction of a Hilbert $C_r^*(\Gamma,\sigma)$-module from an integrable $\sigma$-projective representation of a discrete group $\Gamma$ was outlined by Rieffel \cite{Ri85}. However, the explicit construction of such modules was only accomplished in \cite{Ri88} for the projective Heisenberg representation of a locally compact abelian group of the form $G \times \widehat{G}$. The modules constructed in \cite{Ri88} are usually called \emph{Heisenberg modules} and have found numerous applications in operator algebras and noncommutative geometry, see, e.g., \cite{EcLuPh10,Wa01,BoChHeLi18,Lu11,LeMo16,Co80}. The explicit module provided by Theorem \ref{thm:module_intro} forms a natural generalization of the Heisenberg modules to all nilpotent Lie groups and is established here via the representation theory of nilpotent Lie groups and associated coorbit space theory.
 It should be mentioned that the Heisenberg modules of Rieffel \cite{Ri88} are, in addition, also equipped with a natural right action which gives them the structure of imprimitivity bimodules. No such extra structure is present for the modules provided by Theorem \ref{thm:module_intro}. 
 
The link between the Hilbert $C^*$-module $\module{\pi}{\Gamma}$ and lattice orbits $\pi(\Gamma) \vect$ with $\vect \in \Hip^\infty$ is given by the following theorem, see \Cref{prop:generating_multiwindow} for a more general version and \Cref{subsec:spanning_independence} for definitions of the terms used below.
 
\begin{theorem}\label{thm:intro_module_frame}
With notation as in Theorem \ref{thm:module_intro}, 
let $A := C_r^*(\Gamma,\sigma)$ and let $\vect_1, \ldots, \vect_n \in \Hip^\infty$. Then the following assertions hold:
\begin{enumerate}[(i)]
    \item The set $\{ \vect_1, \ldots, \vect_n \}$ is an algebraic generating set for $\module{\pi}{\Gamma}$ if and only if $( \pi(\Gamma) \vect_j)_{1 \leq j \leq n}$ is a frame for $\Hip$.
    \item The set $\{ \vect_1, \ldots, \vect_n \}$ is $A$-linearly independent and has closed $A$-span in $\module{\pi}{\Gamma}$ if and only if $( \pi(\Gamma) \vect_j)_{ 1 \leq j \leq n}$ is a Riesz sequence in $\Hip$.
\end{enumerate}
\end{theorem}

Theorem \ref{thm:intro_module_frame} provides a correspondence between spanning (resp.\ linear independent) sets in $\module{\pi}{\Gamma}$ and frames (resp.\ Riesz sequences) in $\Hpi$.
For the particular setting of the projective Heisenberg representation $(\pi,L^2(\R^d))$ of $\R^{2d}$, the correspondence for frames (part (i)) was first proved by Luef in \cite{Lu09}. Combined with the general fact that imprimitivity bimodules between unital $C^*$-algebras must be finitely generated, this was used in \cite{Lu09} to prove the existence of a Gabor frame $( \pi(\gamma) \vect_j)_{\gamma \in \Gamma, 1 \leq j \leq n}$ for $L^2 (\mathbb{R}^d)$ with finitely many localized windows over any given lattice $\Gamma$ in $\R^{2d}$. In the present paper we prove directly the existence of such frames in the orbit of nilpotent Lie groups by exploiting classical sampling techniques \cite{feichtinger1989banach, groechenig1991describing}.\footnote{This should be compared with the discussion  \cite[p.1942]{Lu09}, where it is asserted that the techniques \cite{feichtinger1989banach, groechenig1991describing} would not provide an explanation for the fact that only finitely many generators are needed.}  Via Theorem \ref{thm:intro_module_frame}, the existence of a multiwindow frame with finitely many windows in $\Hs$ implies that the module $\module{\pi}{\Gamma}$ constructed in Theorem \ref{thm:module_intro} is finitely generated, which is essential for our approach towards proving \Cref{thm:main_intro}. 

A finitely generated Hilbert $C^*$-module $\module{\pi}{\Gamma}$ as constructed in Theorem \ref{thm:module_intro} corresponds naturally to a projection in a matrix algebra over $C_r^*(\Gamma,\sigma)$, 
which can be explicitly constructed using module frames, see, e.g., \cite{FrLa02, rieffel2010vector}. Using this interpretation, the existence of frames and Riesz sequences $\pi(\Gamma)\vect$ generated by a single $\vect \in \Hs$ can be approached via the comparison theory for projections in (matrix algebras over) the $C^*$-algebra $C_r^*(\Gamma,\sigma)$. This approach is reminiscent of the method \cite{bekka2004square} towards Theorem \ref{thm:density_intro2} using von Neumann algebras.
However, in contrast to the setting of von Neumann algebras, the comparison theory for projections in $C^*$-algebras is remarkably subtle, see, e.g.,\ \cite{Bla,Str} and the references therein. Among others, this is caused by the fact that $C^*$-algebras, in contrast to von Neumann algebras, might not possess sufficiently many projections for a satisfactory comparison theory. This has lead, among others, to the more general notion of Cuntz subequivalence of positive elements in  (matrix algebras over) a $C^*$-algebra \cite{Cun}, and the corresponding comparison then concerns whether Cuntz subequivalence of positive elements can be described via tracial states. A $C^*$-algebra satisfying such a property is said to have \emph{strict comparison of positive elements} and it is this notion, which is stronger than strict comparison of projections, that forms a central ingredient in the present paper (see \Cref{subsec:strict_comp_pos}).
We mention that for irrational non-commutative tori, the presence of strict comparison of projections was proven in \cite{Ri88} (see also \cite{Bla}) and (implicitly) exploited in \cite{jakobsen2021duality} for the study of Gabor frames; see Section \ref{sec:heisenberg}. 

A part of the (revised) Toms--Winter conjecture predicts that for any unital separable, simple, nuclear, infinite-dimensional $C^*$-algebra, strict comparison of positive elements is equivalent to a regularity property known as \emph{finite nuclear dimension} (see e.g.\ \cite[p.\ 302]{Str}). 
The implication from finite nuclear dimension to strict comparison of positive elements is known\footnote{In fact, under various additional assumptions (which all hold in our setting), the full equivalence is known, see, e.g., \cite{Str} and our discussion after Theorem \ref{WR}.},  and follows by combining results of Rørdam \cite{Ror1} and Winter \cite{W}. Therefore, for 
establishing strict comparison of positive elements in our setting, it would suffice to prove that the twisted group $C^*$-algebra $C_r^*(\Gamma,\sigma)$ has finite nuclear dimension. We prove the even stronger property that $C_r^*(\Gamma,\sigma)$ has finite decomposition rank:

\begin{theorem}\label{thm:intro_fnd}
Let $\Gamma$ be a finitely generated, nilpotent group and let $\sigma$ be a 2-cocycle  on $\Gamma$. Then the twisted group $C^*$-algebra $C_r^*(\Gamma,\sigma)$ has finite decomposition rank, in particular finite nuclear dimension. Hence $C_r^*(\Gamma,\sigma)$ has strict comparison of positive elements, and therefore of projections, whenever $(\Gamma,\sigma)$ satisfies Kleppner's condition.
\end{theorem}

Finite decomposition rank of 
group $C^*$-algebras associated to finitely generated, nilpotent groups is due to Eckhardt, Gillaspy and McKenney \cite{EGMK,EMK}. Our proof of Theorem \ref{thm:intro_fnd} relies on \cite{EGMK} and extends their result to the twisted case via the theory of representation groups; in particular, we use a recent result of Hatui, Narayanan and Singla \cite[Theorem 3.5]{HNS}, cf.\ \Cref{sec:nilpotent} for precise details.

Theorems \ref{thm:module_intro}, \ref{thm:intro_module_frame} and \ref{thm:intro_fnd} are the essential ingredients in our proof of Theorem \ref{thm:main_intro}. 

An additional result of independent interest, at least to operator algebraists, is the following theorem, which provides new examples of classifiable $C^*$-algebras in the sense of the classification program for simple separable nuclear $C^*$-algebras (see e.g. \cite[Chapter 18]{Str} and references therein).

\begin{theorem}\label{thm:classifiable}
Let $\Gamma$ be a finitely generated, nilpotent group and let $\sigma$ be a 2-cocycle on $\Gamma$ such that $(\Gamma,\sigma)$ satisfies Kleppner's condition. Then $C_r^*(\Gamma,\sigma)$ is a unital, separable, simple $C^*$-algebra with finite nuclear dimension that satisfies the UCT, hence is classifiable by the Elliott invariant. Moreover, $C_r^*(\Gamma,\sigma)$ has stable rank one.
\end{theorem}

We note that there are remarkably few examples in the literature of pairs $(\Gamma, \sigma)$ (with $\Gamma$ countable) such that $C_r^*(\Gamma, \sigma)$ has stable rank one (cf.~the discussion in \cite[p.~293 and Appendix A]{BO}). Within the class of countable amenable groups, our result about this property in 
Theorem \ref{thm:classifiable}
has previously only been known for finitely generated free abelian groups (this may be deduced from \cite[Theorem 1.5]{BKR}, which deals with simple noncommutative tori), and for some examples of finitely generated nilpotent groups (see the comment after Theorem 4.7 in \cite{OsPhi2006}). 
For a countable nilpotent group and  a 2-cocycle $\sigma$ on $\Gamma$ such that $(\Gamma, \sigma)$ satisfies Kleppner's condition,  Osaka and Phillips raise in \cite[Problem 4.8]{OsPhi2006} the question whether $C^*(\Gamma, \sigma)\simeq C_r^*(\Gamma, \sigma)$ has real rank zero and stable rank one, and whether the order on projections over $C^*(\Gamma, \sigma)$ is determined by traces (this is equivalent to the one we call strict comparison of projections). 
Our Theorems \ref{thm:intro_fnd} and \ref{thm:classifiable} provide a partial answer to their question. 

\subsection{Extensions} 
The proof of Theorem \ref{thm:main_intro} makes a fundamental use of the presence of strict comparison of projections in twisted group $C^*$-algebras $C^*(\Gamma, \sigma)$ of finitely generated nilpotent groups $\Gamma$, which we show by
relying on the paper \cite{EGMK}.
The results in \cite{EGMK} are, however, also valid for virtually nilpotent groups, 
and it might be that a version of Theorem \ref{thm:main_intro} is also valid for more general (classes of) groups of polynomial growth. 
For such an extension, several parts in our approach, among others, the explicitly constructed Hilbert $C^*$-modules (Theorem \ref{thm:module_intro}) would
require different arguments, while there are also several ingredients that do currently not have analogues for virtual nilpotent groups, among others, the existence of suitable representation groups \cite{HNS}. With an eye on such possible extensions, we prove several auxiliary results in a slightly more general setting than strictly needed for our main result Theorem \ref{thm:main_intro}, provided they do not require additional arguments.

As another extension, we mention that our current approach also allows a version of Theorem \ref{thm:main_intro} for Gabor systems on general (compactly generated) locally compact abelian groups, which was left open in \cite{jakobsen2021duality}. For this extension, the module constructed in Theorem \ref{thm:module_intro} could be replaced by the Heisenberg modules 
of \cite{Ri88}.

\subsection{Outline} 
Section \ref{sec:prelim} provides preliminary results on frames, square-integrable representations and group operator algebras. Hilbert $C^*$-modules and their generating sets are discussed in Section \ref{sec:hilbertcmodule}. A general construction of Hilbert $C^*$-modules associated to projective representations of discrete groups is outlined in Section \ref{sec:modules_projective}. Section \ref{sec:strict_comparison} is devoted to strict comparison of positive elements in $C^*$-algebras. In particular, the presence of strict comparison of projections in simple twisted group $C^*$-algebras (\Cref{thm:intro_fnd}) associated to finitely generated nilpotent groups is proven in Section \ref{sec:nilpotent}, along with \Cref{thm:classifiable}. In Section \ref{sec:lattice_orbits} the results obtained in previous sections are applied 
to the setting of nilpotent Lie groups to prove \Cref{thm:main_intro} (cf.\ Theorem \ref{thm:existence}), along with \Cref{thm:module_intro} and \Cref{thm:intro_module_frame}. Lastly, we discuss some examples in \Cref{sec:examples}.

\subsection*{Notation} The notation $\mathbb{N}_0$ will be used for the natural numbers including zero $0$. The complex numbers without zero will be denoted by $\mathbb{C}^{\times} = \mathbb{C} \setminus \{0\}$.
The cardinality of a set $X$ is denoted by $|X| \in [0,\infty]$. For functions $f_1, f_2 : X \to [0,\infty)$, we write $f_1 \asymp f_2$ if there exist constants $\lfb, \ufb > 0$ such that $f_1 (x) \leq \lfb f_2 (x)$ and $f_2(x) \leq \ufb f_1 (x)$ for all $x \in X$.

\section{Preliminaries} \label{sec:prelim}

\subsection{Frames and Riesz sequences} \label{sec:frames}
Let $J$ be a countable index set. A family $(\vect_j)_{j \in J}$ of vectors in a Hilbert space $\Hi$ is called a \emph{frame} for $\Hi$ if 
\begin{align} \label{eq:frame_ineq}
     \| \veco \|_{\Hi}^2 \asymp \sum_{j \in J}|\langle \veco, \vect_j \rangle|^2  \quad  \text{for all} \quad \veco \in \Hi. 
\end{align}
A frame $(\vect_j)_{j \in J}$ is called a \emph{Parseval frame} if \eqref{eq:frame_ineq} can be chosen to be equality.
A system $(\vect_j)_{j \in J}$ is called a \emph{Bessel sequence} if the associated \emph{analysis operator} $\analysis \colon \Hi \to \ell^2(J)$ given by
\[ \analysis \veco = ( \langle \veco, \vect_j \rangle )_{j \in J} \;\;\; \text{for $\veco \in \Hi$} \]
is a bounded, linear operator. Its adjoint, the \emph{synthesis operator} $\synthesis \colon \ell^2(J) \to \Hi$, is determined by $\synthesis e_j = \vect_j$, where $e_j$ denotes the standard basis vector of $\ell^2(J)$ corresponding to an index $j \in J$. The \emph{frame operator} associated to $(\vect_j)_{j \in J}$ is given by $\frameop = \analysis^* \analysis \colon \Hi \to \Hi$. The system $(\vect_j)_{j\in J}$ is a frame for $\Hi$ if and only if $\frameop$ is invertible.

The family $(\vect_j)_{j\in J}$ is called a \emph{Riesz sequence} in $\Hi$ if 
\[  \| c \|^2_{\ell^2} \asymp \Big\| \sum_{j \in J} c_j \vect_j \Big\|^2  \;\;\; \text{for all $c = (c_j)_{j \in J} \in \ell^2(\Gamma)$.} \]
Alternatively, $(\vect_j)_{j \in J}$ is a Riesz sequence if and only if the associated \emph{Gramian operator} $\gramian := \synthesis^*\synthesis \colon \ell^2(J) \to \ell^2(J)$ is invertible, where $\synthesis$ is the synthesis operator of the sequence $(\vect_j)_j$.

For background and further results on frames and Riesz sequences, see, e.g., \cite{christensen2016introduction, young2001an}.

\subsection{Cocycles and projective representations} \label{sub:projective}

Throughout, $G$ denotes a second countable, locally compact, unimodular group with identity element $e$. We assume that a Haar measure $\mu_G$ on $G$ is fixed and let $L^p(G)$ be the associated Lebesgue space for $p \in [1,\infty]$.

By a \emph{cocycle} on $G$ we will mean a Borel measurable map $\sigma \colon G \times G \to \T$ that satisfies the identities
\begin{enumerate}
    \item $\sigma(x,y)\sigma(xy,z) = \sigma(x,yz)\sigma(y,z)$ \;\;\; \text{for all $x,y,z \in G$,}
    \item $\sigma(e,e) = 1$.
\end{enumerate}
Such maps are frequently called normalized $2$-cocycles, or multipliers, in the literature. We denote by $Z^2(G, \T)$ the set of all such cocycles.

Given a cocycle $\sigma$ on $G$, a \emph{$\sigma$-projective unitary representation} $\pi$ of $G$ on a Hilbert space $\Hip$ is a map $\pi \colon G \to \mathcal{U}(\Hip)$ (where $\mathcal{U}(\Hip)$ denotes the unitary operators on $\Hip$) such that
\[ \pi(x)\pi(y) = \sigma(x,y)\,\pi(xy) \;\;\; \text{for all $x,y \in G$.} \]
We will always assume that representations are \emph{measurable}, i.e.,\ $x \mapsto \pi(x) \veco$ is a Borel measurable function on $G$ for every $\veco \in \Hip$. 
A subspace of $\Hip$ is said to be invariant under $\pi$ if it invariant under $\pi(x)$ for every $x \in G$. We say that $\pi$ is \emph{irreducible} if $\{0\}$ and $\Hip$ are the only closed subspaces of $\Hip$ which are invariant under $\pi$. 

Given $\veco,\vect \in \Hip$, we can form the function $C_{\vect}\veco \colon G \to \C$ given by
\[ C_{\vect} \veco(x) = \langle \veco, \pi(x) \vect \rangle \;\;\; \text{for $x \in G$.} \]
Such functions on $G$ are called \emph{matrix coefficients} associated to $\pi$. If $\veco = \vect$, then $C_{\veco} \veco$ is called a \emph{diagonal} matrix coefficient. Matrix coefficients satisfy the relation
\begin{equation}
    C_{\vect} (\pi(x)\veco)(y) = \sigma(x,x^{-1}y) C_{\vect} \veco (x^{-1}y) \;\;\; \text{for $\veco,\vect \in \Hip$ and $x,y \in G$.} \label{eq:matrix_coefficient_intertwining}
\end{equation}
The $\sigma$-twisted left regular representation of $G$ is the $\sigma$-projective unitary representation of $G$ on $L^2(G)$ given by
\[ \big(\lambda_G^\sigma(x)\veco\big)(y) = \sigma(x,x^{-1}y)\,\veco(x^{-1}y) \;\;\; \text{for $x,y \in G$, $\veco \in L^2(G)$.} \]
In terms of $\lambda^\sigma_G$, \eqref{eq:matrix_coefficient_intertwining} can be stated as the intertwining relation $C_{\vect} (\pi(x)\veco) = \lambda^{\sigma}_G(x) C_{\vect} \veco$, provided that $C_{\vect}\veco \in L^2(G)$.

An irreducible, projective unitary representation $\pi$ is called \emph{square-integrable} if there exist nonzero $\veco, \vect \in \Hip$ such that $\int_G | \langle \veco, \pi(x) \vect \rangle |^2 d\mu_G (x) < \infty$. 
 In that case, there exists a unique $d_{\pi} > 0$ called the \emph{formal dimension} of $\pi$ (depending on the Haar measure on $G$) such that
\begin{align} \label{eq:ortho} \int_G \langle \veco, \pi(x) \vect \rangle \overline{ \langle \veco', \pi(x)\vect' \rangle } \; d\mu_G (x) = \frac{1}{d_{\pi}} \langle \veco, \veco' \rangle \overline{ \langle \vect, \vect' \rangle } \;\;\; \text{for all $\veco,\veco',\vect,\vect' \in \Hip$.} 
\end{align}
Square-integrability of $\pi$ implies that 
for each $\vect \in \Hip$,
the \emph{coefficient operator} $C_{\vect} \colon \Hip \to L^2(G)$, mapping each $\veco \in \Hip$ to 
$C_\vect \veco$, is a well-defined, bounded, linear operator. 
Moreover, $d_\pi^{\,1/2} C_{\vect}$ is an isometry which 
intertwines $\pi$ and $\lambda_G^\sigma$, cf.~\eqref{eq:matrix_coefficient_intertwining}, and thus realizes $\pi$ as a subrepresentation of $\lambda_G^\sigma$. 

For more details on projective and square-integrable representations, cf.~\cite{varadarajan1986geometry, Ani06, Pac2, Ra98}.

\subsection{Lattices}
Let $\Gamma \leq G$ be a discrete subgroup of a second countable unimodular locally compact group $G$. 
A left (resp.\ right) fundamental domain of $\Gamma$ in $G$ is a Borel set $\Sigma \subset G$ such that $G = \Gamma \cdot \Sigma$ and $\gamma \Sigma \cap \gamma' \Sigma = \emptyset$ (resp.\ $G= \Sigma \cdot \Gamma$ and $\Sigma \gamma \cap \Sigma \gamma' = \emptyset$) for all $\gamma, \gamma' \in \Gamma$ with $\gamma \neq \gamma'$. The discrete $\Gamma \leq G$ is called a \emph{lattice} if it admits a left or right fundamental domain of finite measure. Alternatively, a discrete subgroup $\Gamma \leq G$ is a lattice if, and only if, the quotient $G/\Gamma$ carries a finite $G$-invariant Radon measure. By Weil's integral formula, $\vol(G/\Gamma) = \mu_G (\Sigma)$ for any choice of fundamental domain $\Sigma \subset G$ for $\Gamma$. See, e.g.,~\cite{raghunathan1972discrete} for more details and properties.

\subsection{Twisted group operator algebras}\label{subsec:group_operator_algebras} 
Let $\Gamma$ be a countable discrete group, and let $\sigma$ be a cocycle on $\Gamma$. The \emph{$\sigma$-twisted convolution} of two functions $a,b \colon \Gamma \to \C$ in $\ell^1(\Gamma)$ is defined by
\[ (a \ast_{\sigma} b)(\gamma') = \sum_{\gamma \in \Gamma} \sigma(\gamma,\gamma^{-1}\gamma')a(\gamma)b(\gamma^{-1}\gamma') \;\;\; \text{for $\gamma' \in \Gamma$.} \]
We often simply write $\ast = \ast_{\sigma}$. The \emph{$\sigma$-twisted involution} of $a$ is given by
\[ a^*(\gamma) = \overline{ \sigma(\gamma,\gamma^{-1}) a(\gamma^{-1}) } \;\;\; \text{for $\gamma \in \Gamma$.} \]
The Banach space $\ell^1(\Gamma)$ becomes a unital Banach $*$-algebra with respect to $\sigma$-twisted convolution and $\sigma$-twisted involution, which we denote by $\ell^1(\Gamma,\sigma)$.

The \emph{full $\sigma$-twisted group $C^*$-algebra of $\Gamma$} is the completion $C^*(\Gamma,\sigma)$ of $\ell^1(\Gamma,\sigma)$ with respect to the universal $C^*$-norm
\[ \| a \|_{u} = \sup_\pi \| \pi(a) \| \;\;\; \text{for $a \in \ell^1(\Gamma,\sigma)$,} \]
where the supremum is taken over all nondegenerate $*$-representations of $\ell^1(\Gamma,\sigma)$ on a Hilbert space. 
We will frequently consider $\ell^1(\Gamma, \sigma)$ as embedded in $C^*(\Gamma, \sigma)$.
There is a bijective correspondence $\pi \mapsto \widehat \pi$ between $\sigma$-projective unitary representations of $\Gamma$ and nondegenerate $*$-representations of $C^*(\Gamma, \sigma)$, 
where $\widehat\pi$ is determined from $\pi $ by \[\widehat\pi(a) = \sum_{\gamma \in \Gamma} a(\gamma) \pi(\gamma) \]
for all $a\in \ell^1(\Gamma, \sigma)$. As $\pi$ is irreducible if and only if 
$\widehat \pi$ is irreducible, one deduces from the Gelfand-Naimark theory  for $C^*$-algebras that there always exists enough $\sigma$-projective irreducible unitary representations of $\Gamma$ to separate its elements.   

Let $\lambda_\Gamma^\sigma$ be the $\sigma$-twisted left regular representation of $\Gamma$ on $\ell^2(\Gamma)$. The $C^*$-algebra generated by $\lambda_\Gamma^\sigma(\Gamma) \subseteq \B(\ell^2(\Gamma))$ is called the \emph{reduced $\sigma$-twisted group $C^*$-algebra of $\Gamma$} and is denoted by $C_r^*(\Gamma,\sigma)$. Equivalently, we have $C_r^*(\Gamma,\sigma) = \widehat{\lambda_\Gamma^\sigma} \big(C^*(\Gamma, \sigma)\big)$. Similarly, the von Neumann algebra generated by $\lambda_\Gamma^\sigma(\Gamma) \subseteq \B(\ell^2(\Gamma))$ is called the \emph{$\sigma$-twisted group von Neumann algebra of $\Gamma$} and is denoted by $\vN(\Gamma,\sigma)$.
It is equipped with a faithful, normal tracial state $\tau$ given by
\[ \tau(a) = \langle a \delta_e, \delta_e \rangle \;\;\; \text{for $a \in L(\Gamma,\sigma)$},\]
where $\delta_e \in \ell^2(\Gamma)$ denotes the characteristic function of $\{e\}$ in $\Gamma$.  
We refer to $\tau$ as the \emph{canonical tracial state}, and denote its restriction to $C_r^*(\Gamma,\sigma)$ also by $\tau$.

 The canonical map $\widehat{\lambda_\Gamma^\sigma}: C^*(\Gamma,\sigma) \to C_r^*(\Gamma,\sigma)$ is always faithful on $\ell^1(\Gamma, \sigma)$. Hence we will often consider $\ell^1(\Gamma, \sigma)$ as embedded in $C_r^*(\Gamma, \sigma)$ via this map.   If the group $\Gamma$ is amenable, then 
 $\widehat{\lambda_\Gamma^\sigma}$ is a $*$-isomorphism. In this case, e.g.,~when $\Gamma$ is nilpotent, we will frequently identify $C^*(\Gamma,\sigma)$ with $C_r^*(\Gamma,\sigma)$.
For additional information about the operator algebras associated to $(\Gamma, \sigma)$, the reader may consult, e.g., \cite{BC, Pac2, Zel} and references therein. 

An element $\gamma \in \Gamma$ is called \emph{$\sigma$-regular} if we have $\sigma(\gamma,\gamma') = \sigma(\gamma',\gamma)$ whenever $\gamma' \in \Gamma$ and $\gamma \gamma' = \gamma' \gamma$. If $\gamma$ is $\sigma$-regular, then every element in the conjugacy class of $\gamma$ is $\sigma$-regular, hence it makes sense to talk about $\sigma$-regular conjugacy classes. The pair $(\Gamma,\sigma)$ is said to satisfy \emph{Kleppner's condition} if every nontrivial, $\sigma$-regular conjugacy class is infinite. The twisted group von Neumann algebra $\vN(\Gamma,\sigma)$ is a factor (i.e., has a trivial center) if and only if $(\Gamma,\sigma)$ satisfies Kleppner's condition, cf.~\cite[Theorem 2]{Kle62}.
Kleppner's argument shows that  $C_r^*(G, \sigma)$ has a nontrivial center whenever  $(G, \sigma)$ does not satisfy Kleppner's condition. Hence this condition is necessary  
for $C_r^*(G, \sigma)$ to be simple (i.e., to have no non-trivial ideals), but it is not always sufficient, cf.~\cite{BO}. 
See also \cite{P, Om14} for other results relying on this condition.

A function $a \colon \Gamma \to \C$ is said to be \emph{$\sigma$-positive definite} if
\begin{equation}
\sum_{i,j=1}^n c_i \overline{c_j} a(\gamma_j \gamma_i^{-1}) \sigma( \gamma_j \gamma_i^{-1}, \gamma_i) \geq 0 \;\;\; \text{for all $\gamma_1, \ldots, \gamma_n \in \Gamma$, $c_1, \ldots, c_n \in \C$} ; \label{eq:pos_def}
\end{equation}
see \cite{BC, LeXi19} for a slightly different definition (where $a$ would be called $\bar\sigma$-positive definite).

The following characterization of $\sigma$-positive functions will be convenient for our purposes.   
The result is part of the folklore for trivial cocycles $\sigma \equiv 1$. 

\begin{proposition}\label{prop:pos_def}
Assume $a \in \ell^1(\Gamma,\sigma)$. Then $a$ is $\sigma$-positive definite as a function on $\Gamma$ if and only if $\widehat{\lambda_\Gamma^\sigma}(a)$ is positive as an element of $C_r^*(\Gamma,\sigma)$. 
It follows that if $a \in \ell^1(\Gamma, \sigma)$ is a diagonal matrix coefficient associated to a $\sigma$-projective unitary representation of $\Gamma$, then $\widehat{\lambda_\Gamma^\sigma}(a)$ is positive in $C_r^*(\Gamma,\sigma)$.
\end{proposition}

\begin{proof}
Given a finite subset $F = \{ \gamma_1, \ldots, \gamma_n \}$ of $\Gamma$, denote by $P_F$ the orthogonal projection of $\ell^2(\Gamma)$ onto $\spn \{ \delta_\gamma : \gamma \in F \}$. Then it is well-known that an operator $T \in \B(\ell^2(\Gamma))$ is positive if and only if $P_F T P_F$ is positive for any such $F$. 
Let $\eta \in \ell^2(\Gamma)$. Write $P_F \eta = \sum_{i=1}^n c_i \delta_{\gamma_i}$ for some scalars $c_1, \ldots, c_n \in \C$ and note that
\begin{align*}
    \langle \widehat{\lambda_\Gamma^\sigma}(a) \delta_{\gamma_i}, \delta_{\gamma_j} \rangle &= \sum_{\gamma \in \Gamma} a(\gamma) \langle \lambda_{\Gamma}^{\sigma}(\gamma) \delta_{\gamma_i}, \delta_{\gamma_j} \rangle 
    = \sum_{\gamma \in \Gamma} a(\gamma) \sigma(\gamma,\gamma_i) \langle \delta_{\gamma \gamma_i}, \delta_{\gamma_j} \rangle \\
    &= a(\gamma_j \gamma_i^{-1}) \sigma( \gamma_j \gamma_i^{-1}, \gamma_i) .
\end{align*}
Hence
\begin{align*}
\langle P_F \widehat{\lambda_\Gamma^\sigma}(a) P_F \eta, \eta \rangle &= \langle \widehat{\lambda_\Gamma^\sigma}(a) P_F \eta, P_F \eta \rangle 
= \sum_{i,j=1}^n c_i \overline{c_j}\, \langle \widehat{\lambda_\Gamma^\sigma}(a) \delta_{\gamma_i}, \delta_{\gamma_j} \rangle \\
&= \sum_{i,j=1}^n c_i \overline{c_j}\, a(\gamma_j \gamma_i^{-1}) \sigma( \gamma_j \gamma_i^{-1}, \gamma_i).
\end{align*}
Thus the condition $\widehat{\lambda_\Gamma^\sigma}(a) \geq 0$ is equivalent to the condition that the above expression is non-negative for all $\gamma_1, \ldots \gamma_n \in G$ and $c_1, \ldots, c_n \in \C$, i.e., to the $\sigma$-positive definiteness of $a$, as desired.

Assume now that the function $a \in \ell^1(\Gamma, \sigma)$ may be written as $a(\gamma)= \langle f, \pi(\gamma) f\rangle $ for some  
$\sigma$-projective unitary representation $\pi$ of $\Gamma$ on $\Hip$ and some $\veco \in \Hip$. Then for $\gamma_1, \ldots, \gamma_n \in \Gamma$ and $c_1, \ldots, c_n \in \Gamma$ we have that
\begin{align*}
    \sum_{i,j=1}^n c_i \overline{c_j}& \,\langle \veco, \pi(\gamma_j \gamma_i^{-1}) \veco \rangle \sigma(\gamma_j \gamma_i^{-1}, \gamma_i) = \sum_{i,j=1}^n c_i \overline{c_j} \langle \veco,  \pi(\gamma_j) \pi(\gamma_i^{-1})\veco \rangle \sigma(\gamma_j \gamma_i^{-1}, \gamma_i) \sigma(\gamma_j, \gamma_i^{-1}) \\
    &= \sum_{i,j=1}^n c_i \overline{c_j} \langle \pi(\gamma_i)^* \veco, \pi(\gamma_i^{-1})\veco \rangle \sigma(\gamma_i^{-1},\gamma_i) 
    = \Big\langle \sum_{i=1}^n c_i \pi(\gamma_i)^*\veco, \sum_{j=1}^n c_j \pi(\gamma_j)^* \veco \Big\rangle \geq 0 .
\end{align*}
This shows that 
$a$ is $\sigma$-positive definite, hence that 
$\widehat{\lambda_\Gamma^\sigma}(a)$ is positive in $C_r^*(\Gamma,\sigma)$.
\end{proof}

\section{Hilbert C*-modules, generating sets and localization} \label{sec:hilbertcmodule}

Throughout this section, $A$ denotes a unital $C^*$-algebra with unit $1_A$.

\subsection{Hilbert C*-modules}\label{subsec:hilb_mod}
We follow the conventions in \cite{La95, RaWi98}, except that we prefer to work with \emph{left} Hilbert $C^*$-modules. Thus, by an \emph{inner product $A$-module} we mean a complex vector space $\E$ together with a left $A$-module structure and a map $\lhs{\cdot}{\cdot} \colon \E \times \E \to A$ such that the following axioms are satisfied: 
\begin{enumerate} 
    \item[(a1)] $\lhs{a \veco + b \vect}{h} = a\, \lhs{\veco}{h} + b\, \lhs{\vect}{h}$ for all $a, b \in A$ and $\veco,\vect,h \in \E$.
    \item[(a2)] $\lhs{\veco}{\vect}^* = \lhs{\vect}{\veco}$ for all $\veco,\vect \in \E$.
    \item[(a3)] $\lhs{\veco}{\veco} \geq 0$ (as a positive element of $A$) and $\lhs{\veco}{\veco} = 0$ if and only if $\veco = 0$.
\end{enumerate}
An inner product $A$-module becomes a normed space with respect to
\[ \| \veco \|_\E = \| \lhs{\veco}{\veco} \|^{1/2} \;\;\; \text{for $\veco \in \E$.} \]
If $\E$ is complete with respect to this norm, $\E$ is called a \emph{Hilbert $A$-module}. We will often consider $A$ itself as a Hilbert $A$-module with respect to the inner product $\lhs{a}{b} = ab^*, a,b \in A$.

If $A_0$ is a dense $*$-subalgebra of $A$, then a \emph{pre-inner product $A_0$-module} is a complex vector space $\E_0$ together with a left $A_0$-module structure and a map $\lhs{\cdot}{\cdot} \colon \E_0 \times \E_0 \to A_0$ such that the above three axioms are satisfied for $a \in A_0$ and $\veco,\vect \in \E_0$, where the positivity is interpreted in the completion $A$ of $A_0$. A pre-inner product $A_0$-module $\E_0$ can always be completed into a Hilbert $A$-module $\E$, see \cite[Lemma 2.16]{RaWi98}.

Given a closed $A$-submodule $\E_0$ of a Hilbert $A$-module $\E$, the \emph{orthogonal complement} of $\E$ is the set $\E_0^{\perp} = \{ \veco \in \E : \text{$\lhs{\veco}{\vect} = 0$ for all $\vect \in \E_0$} \}$. One always has $\E_0 \cap \E_0^{\perp} = \{ 0 \}$ but not necessarily $\E_0 + \E_0^{\perp} = \E$. If the latter is the case, $\E_0$ is called \emph{orthogonally complementable} in $\E$.

One may form the direct sum of finitely many Hilbert $A$-modules in the obvious way. The direct sum of $n$ copies of a Hilbert $A$-module $\E$ is denoted by $\E^n$.

\subsection{Adjointable operators}

A map $T \colon \E \to \F$ between Hilbert $A$-modules is called \emph{adjointable} if there exists a (uniquely determined) map $T^* \colon \F \to \E$ such that
\[ \lhs{T\veco}{\vect} = \lhs{\veco}{T^*\vect} \;\;\; \text{for all $\veco \in \E$ and $\vect \in \F$.} \]
An adjointable map is automatically a bounded, $A$-linear operator. 
We say that $\E$ and $\F$ are isomorphic (as Hilbert $A$-modules) if there exists an adjointable map $T:\E\to \F$ which is a unitary, i.e., satisfies that $T^*T = I_\E$ and $TT^* = I_\F$.
We denote by $\mathcal{L}_A(\E,\F)$ the set of all adjointable maps from $\E$ into $\F$, which is a Banach space with respect to the operator norm. Furthermore, we set $\mathcal{L}_A(\E) \coloneqq \mathcal{L}_A(\E,\E)$. The map $T \mapsto T^*$ is an involution on $\mathcal{L}_A(\E)$, and $\mathcal{L}_A(\E)$ becomes a $C^*$-algebra with respect to the operator norm.

If a map $T \colon \E \to \F$ is $A$-linear and isometric (as a map between the underlying Banach spaces), it need not be adjointable. However, the following conditions are equivalent:
\begin{itemize}
    \item[1)] $T$ is $A$-linear, isometric, and $\Img(T)$ is orthogonally complementable in $\E$;
    \item[2)] $T$ is adjointable with $T^*T = I_\E$;
    \item[3)] $T$ is an adjointable isometry.
\end{itemize} 
The equivalence between 1) and 2) is \cite[Proposition 3.6]{La95}. The fact that 2) implies 3) is straightforward. Finally, if 3) holds, then $\Img(T)$ is closed, and orthogonally complementable in $\F$ by \cite[Theorem 3.2]{La95}, so 1) holds.
An immediate consequence is that there exists an adjointable isometry $\E \to \F$ if and only if $\E$ is isomorphic to a closed, orthogonally complementable $A$-submodule of $\F$.

Given $\vect, h \in \E$, the \emph{rank-one operator} $\Theta_{\vect,h} \in \mathcal{L}_A(\E)$ is given by
$ \Theta_{\vect,h} \veco = \lhs{\veco}{\vect} h$ for $\veco \in \E$.

\subsection{Spanning and independence}\label{subsec:spanning_independence}

Let $\E$ be a Hilbert $A$-module. The \emph{$A$-span} of a set $S \subseteq \E$ is the set $\spn_A S$ of all finite $A$-linear combinations $\sum_{j=1}^n a_j \vect_j$ where $\vect_j \in S$, $a_j \in A$ for $1 \leq j \leq n$. A finite set $S \subseteq \E$ is called a \emph{generating set} for $\E$ if $\spn_A S = \E$, and $\E$ is called \emph{finitely generated} if it admits a finite generating set. Note that this notion is often called \emph{algebraically finitely generated} in Hilbert $C^*$-module theory to distinguish it from the weaker notion of being topologically finitely generated.

Associated to a finite set $\{ \vect_1, \ldots, \vect_n \} \subseteq \E$ are the \emph{analysis operator} $\analysis \colon \E \to A^n$ and the \emph{synthesis operator} $\synthesis \colon A^n \to \E$ given by
\begin{align*}
    \analysis \veco &= (\lhs{\veco}{\vect_j})_{j=1}^n , \\
    \synthesis (a_j)_{j=1}^n &= \sum_{j=1}^n a_j \vect_j
\end{align*}
for $\veco \in \E$ and $(a_j)_j \in A^n$. 
Both these operators are
adjointable, with $\analysis^* = \synthesis$. The operator $\frameop \colon \analysis^*\analysis \colon \E \to \E$ is called the \emph{frame operator} and the operator $\gramian = \synthesis^*\synthesis \colon A^n \to A^n$ is called the \emph{Gramian operator}. 

The following characterization of generating sets will be convenient for our purposes, cf.~\cite[Theorem 5.9]{FrLa02}.

\begin{lemma}[\cite{FrLa02}] \label{lem:frame_gen}
A finite set $\{ \vect_1, \ldots, \vect_n \} \subseteq \E$ is a generating set for $\E$ if and only if it is a \emph{frame} for $\E$, that is, there exist $\lfb, \ufb > 0$ such that
\begin{align}\label{eq:module_frame}
      \lfb \lhs{\veco}{\veco} \leq \sum_{j=1}^n \lhs{\veco}{\vect_j} \lhs{\veco}{\vect_j}^* \leq C_2 \lhs{\veco}{\veco} \;\;\; \text{for all $\veco \in \E$,} 
      \end{align}
      if and only if the associated frame operator $\frameop$ is invertible in $\mathcal{L}_A(\E)$.
\end{lemma}

\begin{proof} Consider the set $\{ \vect_1, \ldots, \vect_n \} \subseteq \E$. If it is  generating for $\E$, then it is a frame for $\E$ by \cite[Theorem 5.9]{FrLa02}. 
If it is a frame for $\E$, satisfying (\ref{eq:module_frame}), then, using \cite[Lemma 4.1]{La95}, we get that the positive operator $\frameop = \sum_{j=1}^n \Theta_{\vect_j, \vect_j}$ satisfies that $C_1\, I_\E \leq \frameop \leq C_2\, I_\E$; since $C_1 >0$, it follows that $\frameop$ is invertible in $\mathcal{L}_A(\E)$.
Finally, if  $\frameop$ is invertible in $\mathcal{L}_A(\E)$, and $\veco \in \E$, then $ \veco = \frameop \frameop^{-1} \veco = 
\sum_{j=1}^n  \lhs{\frameop^{-1}\veco}{\vect_j}\, \vect_j,$
which shows that $\{ \vect_1, \ldots, \vect_n \}$ is generating for $\E$.
\end{proof}

As for Hilbert spaces, if one can choose $C_1 = C_2 = 1$ in \eqref{eq:module_frame}, the frame is called \emph{Parseval}.

We call a finite set $\{ \vect_1, \ldots, \vect_n \}$ \emph{$A$-linearly independent} if whenever $a_1, \ldots, a_n \in A$ are such that $\sum_{j=1}^n a_j \vect_j = 0$, then $a_j = 0$ for $1 \leq j \leq n$. Note that contrary to the Hilbert space case, the $A$-span of a finite set might not be topologically closed. We say that a finite set $S$ has closed $A$-span if $\spn_A S$ is topologically closed.

\begin{lemma}[\cite{Ar07}]\label{prop:gramian_invertible}
A finite set $\{ \vect_1, \ldots, \vect_n \} \subseteq \E$ is $A$-linearly independent with closed $A$-span if and only if the associated Gramian operator $\gramian$ is invertible in $\mathcal{L}_A(\E)$.
\end{lemma}

\begin{proof}
Consider a finite set $\{ \vect_1, \ldots, \vect_n \} \subseteq \E$. Applying \cite[Proposition 2.1]{Ar07} to the associated analysis operator $\analysis:\E\to A^n$, and setting $\synthesis= \analysis^*$, we get that the following conditions are equivalent:
\begin{itemize}
    \item[1)] $\analysis$ is surjective;
    
    \smallskip
    \item[2)] There exists $C>0$ such that $C\, \| (a_j)_j \|_{A^n} \leq \| \synthesis(a_j)_j \|_\E$ for all  $(a_j)_j \in A^n$;
    
    \smallskip 
    \item[3)] There exists $C'>0$ such that $C'\lhs{(a_j)_j}{(a_j)_j} \leq \lhs{\synthesis (a_j)_j}{\synthesis (a_j)_j}$ for all $(a_j)_j \in A^n$.
\end{itemize}
In fact, to show these equivalences, it is shown in \cite{Ar07} that 1) $\Rightarrow \,  \synthesis^*\synthesis \text{ is invertible } \Rightarrow$ 3) $\Rightarrow$ 2) $\Rightarrow \, \synthesis \text{ is injective with closed range } \Rightarrow$ 1).  This means that the associated Gramian operator $\gramian=\synthesis^*\synthesis$ is invertible in $\mathcal{L}_A(\E)$ if and only if 
$\synthesis$ is injective with closed range, which is equivalent to $\{ \vect_1, \ldots, \vect_n \}$ being $A$-linearly independent with closed $A$-span.
\end{proof}

In \cite{AuJaLu20}, sets satisfying the properties of \Cref{prop:gramian_invertible} are called \emph{module Riesz sequences}. 

 We recall that if $M_n(A)$ denotes the $C^*$-algebra consisting of all $n\times n$ matrices over $A$ and $p \in M_n(A)$ is a projection (i.e., $p$ is self-adjoint and idempotent), then $A^np$ is a Hilbert $A$-submodule  of $A^n$  (we consider here elements of $A^n$ as row vectors).

\begin{proposition}\label{prop:generating_isometry}
Let $n \in \N$. Then the following hold:
\begin{enumerate}
    \item[(i)] There exists a generating set with $n$ elements in $\E$ if and only if there exists an adjointable isometry $\E \to A^n$, if and only if there exists a projection $p$ in $M_n(A)$ such that $\E\cong A^n p$.
    \item[(ii)] There exists an $A$-linearly independent set with $n$ elements and closed $A$-span in $\E$ if and only if there exists an adjointable isometry $A^n \to \E$.
\end{enumerate}
Furthermore, if there exists a generating set (resp.\ $A$-linearly independent set with closed $A$-span) with $n$ elements, then one can find a generating set (resp.\ $A$-linearly independent set with closed $A$-span) with $n$ elements that belong to any dense subspace $\E_0$ of $\E$.
\end{proposition}

\begin{proof}
(i) A generating set with $n$ elements is a frame by \Cref{lem:frame_gen}. Hence, the corresponding frame operator $\frameop$ is invertible. As in the Hilbert space case, by applying $\frameop^{-1/2}$ to each element of the frame, one obtains a new frame with $n$ elements which is Parseval. But then the associated analysis operator $\analysis \colon \E \to A^n$ is an adjointable isometry.
  
 Next, assume that there exists an adjointable isometry $\analysis: \E \to A^n$ for some $n\in \N$. Then one checks readily that $\Img(\analysis) = A^n p$, where  $p$ is the projection  in $M_n(A)$ whose $i$-th row vector is $\analysis\analysis^*e_i$, where $e_i= (\delta_{i,j}1_A)_{j=1}^n \in A^n$. Thus, $\E\cong A^n p$. 
 
 Finally, if there exists a projection $p$ in $M_n(A)$ such that $\E\cong A^n p$, then, as $A^n p$ has clearly a generating set with $n$ elements, this is also true for $\E$. 

(ii) Assume first that $\vect_1, \ldots, \vect_n$ are $A$-linearly independent and that $\mathcal{F} = \spn_A \{ \vect_1, \ldots, \vect_n \}$ is closed in $\E$. Then $\mathcal{F}$ is a Hilbert $A$-module, and $\{ \vect_1, \ldots, \vect_n \}$ is a generating set for $\F$. Denoting by $\frameop \in \mathcal{L}_A(\F)$ the corresponding frame operator, $\frameop$ is positive and invertible, and $(\vect_j)_{j=1}^n$ is a frame.
Since $\frameop \veco = \sum_{j=1}^n \lhs{ \veco}{\vect_j} \vect_j$ for every $\veco \in \F$, we get that
$ \vect_i = \sum_{j=1}^n \lhs{\frameop^{-1} \vect_i}{\vect_j} \vect_j  $
for every $1\leq i \leq n$.
By $A$-linear independence, this forces $\lhs{\frameop^{-1}\vect_i}{\vect_j} = \delta_{i,j}\,1_A$. Hence, the set $\{ \tilde{\vect}_j : 1 \leq j \leq n \}$, where $\tilde{\vect}_j :=\frameop^{-1/2} \vect_j$, is orthonormal in the sense that
\[ \lhs{\tilde{\vect}_i}{\tilde{\vect}_j} = \lhs{\frameop^{-1}\vect_i}{\vect_j} = \delta_{i,j}\,1_A \]
for all $1 \leq i,j \leq n$. It follows that the associated synthesis operator $\synthesis \colon A^n \to \E$ given by $\synthesis(a_j)_j = \sum_{j=1}^n a_j \tilde{\vect}_j$ is an adjointable isometry.

Conversely, suppose $\synthesis \colon A^n \to \E$ is an adjointable isometry. Set $\vect_j = \synthesis(e_j)$ for each $1 \leq j \leq n$, where $e_j$ denotes the $j$th element of the standard basis of $A^n$. Then $\{ \vect_1, \ldots, \vect_j \}$ is an orthonormal set, hence $A$-linearly independent. This finishes the proof of (ii).

Suppose now that 
$\{ \vect_1, \ldots, \vect_n \}$ is a generating set 
for $\E$, which by the argument for (i) above can be assumed to be a Parseval frame. Let $\frameop$ be the corresponding frame operator. In terms of rank-one operators, we get
$ \frameop = \sum_{j=1}^n \Theta_{\vect_j,\vect_j} = I_\E .$
Now it is an easy exercise to check that 
\[ \|\Theta_{\vect, \vect} -\Theta_{\vect', \vect'}\| \leq (\|\vect\| + \|\vect'\|) \|\vect-\vect'\| \]
for all $\vect, \vect' \in \E$.
 By density of $\E_0$ in $\E$, we can find $\vect_j' \in \E_0$ such that $\|\vect_j-\vect_j'\| < \delta$ for every $j=1, \ldots, n$, where $M:= \max\{\|\vect_1\|, \ldots , \|\vect_n\|\}$ and $\delta:= \min\{1, ((2M+1)n)^{-1}\}$. This gives that
 \[\|I_\E - \sum_{j=1}^n \Theta_{\vect_j',\vect_j'}\| \leq \sum_{j=1}^n \|\Theta_{\vect_j,\vect_j}-\Theta_{\vect_j',\vect_j'}\| \leq (2M+1) \sum_{j=1}^n \|\vect_j-\vect_j'\| < (2M+1)n \delta \leq 1,\]
 hence that $\frameop' := \sum_{j=1}^n \Theta_{\vect_j',\vect_j'}$ is invertible in $\mathcal{L}_A(\E)$.
 Since $\frameop'$ is the frame operator associated to the family $\{\vect_1', \ldots, \vect_n'\}$, it follows from Lemma \ref{lem:frame_gen} that this family, which lies in $\E_0$, 
 is a generating set for $\E$.
By considering the Gramian operator instead of the frame operator, a similar argument shows the analogous property for $A$-linearly independent sets with closed $A$-span.
\end{proof}

The first part of \Cref{prop:generating_isometry} is essentially known, see, e.g., \cite[Section 7]{rieffel2010vector} and \cite[Section 5]{FrLa02} for somewhat similar statements.

\subsection{Localization of Hilbert C*-modules}\label{subsec:loc}
We will repeatedly use the following simple observation, which relies on the elementary fact that $cac^* \leq cbc^*$ for every $c\in A$ whenever $a, b \in A$ are self-adjoint and $a\leq b$.

 \begin{lemma}\label{observation:trace_positive} 
Assume $\tau$ is a tracial state on $A$ 
and let $a,b \in A$ be positive. Then 
 \[ 0 \leq \tau(ab) \leq \| a \| \tau(b) .\]
\end{lemma}
\begin{proof}
Indeed, since $a$ is positive, we have that $0 \leq a \leq \| a \| 1_A$. 
 Hence we get
 \[ 0 \leq   b^{1/2} a b^{1/2} \leq b^{1/2} \| a \| 1_A b^{1/2} = \| a \| b, \]
 which implies that
 \[ 0  \leq \tau(b^{1/2} a b^{1/2}) \leq \tau(\| a \| b) = \| a \| \tau(b) .\]
 As $\tau$ is tracial, 
$ \tau(b^{1/2}a b^{1/2}) = \tau(ab)$, and
the result follows.
\end{proof}
 
We assume from now on that $A$ has a faithful, tracial state $\tau$.
 We denote by $\Hi$ the Hilbert space obtained from the GNS construction applied to $(A,\tau)$. 
Thus $\Hi$ is the Hilbert space completion of $A$ with respect to the inner product given by $\langle a, b\rangle_\tau = \tau(a b^*)$ for $a, b \in A$. To avoid confusion, we write $\widehat a$ when we view $a\in A$ as an element of $\Hi$. 
Since $\tau$ is faithful, we can view $A$ as a $C^*$-subalgebra of $\B(\Hi)$, 
whose action on $\Hi$ is determined by $a\widehat{b} = \widehat{ab}$ for $a, b \in A$. The vector $\veco_0:=\widehat{1_A} \in \Hi$ is then cyclic and separating for $A$, and we have
$\tau(a) = \langle a \veco_0, \veco_0\rangle_\tau$ for every $a \in A$.

Let $M = A'' \subseteq \B(\Hi)$ be the von Neumann algebra on $\Hi$ generated by $A$. By \cite[Proposition V.3.19]{Tak}, 
 the functional on $M$ given by $a \mapsto \langle a \veco_0, \veco_0\rangle_\tau$ is a faithful tracial normal state on $M$,  
which we also denote by $\tau$. 
The GNS-space of $(M,\tau)$, which is usually denoted by $L^2(M, \tau)$, can then be identified with $\Hi$, and $M$ acts also on it from the right in the obvious way.

Throughout this subsection we also fix a Hilbert $A$-module $\E$, where we denote the $A$-valued inner product by $\lhs{\cdot}{\cdot}$. We define a scalar-valued inner product on $\E$ by setting
\[ \langle \veco, \vect \rangle_{\Hi_{\E}^{\tau}} = \tau(\lhs{\veco}{\vect}) \;\;\; \text{for $\veco,\vect \in \E$,} \]
and denote by $\Hi_{\E}^\tau$ the corresponding Hilbert space completion of $\E$. (This is a special case of a procedure known as \emph{localization} of Hilbert $C^*$-modules, see \cite[p.\ 7]{La95}.) Since $\tau$ is a tracial state, the left action of $A$ on $\E$ extends to a representation $\pi_\E^\tau$ of $A$ on $\Hi_{\E}^\tau$.
Indeed, using \Cref{observation:trace_positive}, we get that for all $a\in A$ and $\veco\in \E$,  
\[ \|a\veco\|_{\Hi_{\E}^\tau}^2 = \tau(\lhs{a\veco}{a\veco})
= \tau(a\lhs{\veco}{\veco}a^*)=  \tau(a^*a \,\lhs{\veco}{\veco})\leq \|a^*a\| \tau(\lhs{\veco}{\veco})=\|a\|^2 \|\veco\|_{\Hi_{\E}^\tau}^2.\]
It follows that the linear operator $\veco \mapsto a\veco$ extends to a bounded linear operator $\pi_\E^\tau(a)$ on $\Hi_{\E}^\tau$ for each $a\in A$, and one checks readily that the map $a \mapsto \pi_\E^\tau(a)$ is a $*$-homomorphism from $A$ into $\mathcal{B}(\Hi_{\E}^\tau)$.
We refer to the pair $(\Hi_{\E}^\tau,\pi_\E^\tau)$ as the \emph{localization} of $\E$ with respect to $(A,\tau)$.

We recall the notion of a Hilbert module over a von Neumann algebra, which is different from the notion of a Hilbert $C^*$-module over a $C^*$-algebra. If $N$ denotes a von Neumann algebra, a (left, normal) Hilbert $N$-module is a Hilbert space $\mathcal{K}$ together with a normal unital representation 
of $N$ on $\mathcal{K}$. 

The following proposition seems part of the folklore. 
As we could not find a suitable reference in the literature, we include a proof, for the ease of the reader. 

\begin{proposition}\label{proposition:extend_von_neumann}
The representation $\pi_\E^\tau$ of $A$ on $\Hi_{\E}^\tau$ extends uniquely 
to a normal representation of $M$ on $\Hi_{\E}^\tau$.
In other words, $\Hi_{\E}^\tau$ is a Hilbert $M$-module. 
\end{proposition}

\begin{proof}
Given $a \in M$, we define a map $\phi_a \colon \E \times \E \to \C$ by
\[ \phi_a(\veco,\vect) = \tau(a \lhs{\veco}{\vect}) \;\;\; \text{for $a \in M$ and $\veco,\vect \in \E$} .\]
Then $\phi_a$ is linear in the first variable and conjugate-linear in the second. We also have that
\[ \phi_a(\veco,\vect) = \tau(a \lhs{\veco}{\vect}) = \tau((a^*)^* \lhs{\vect}{\veco}^*) = \tau(( \lhs{\vect}{\veco}a^*)^*) = \overline{ \tau( \lhs{\vect}{\veco}a^*)} = \overline{ \phi_{a^*}( \vect, \veco) } .\]
In particular, if $a$ is self-adjoint, then $\phi_a(\veco,\vect) = \overline{\phi_a(\vect,\veco)}$. Moreover, if $a \geq 0$, then by \Cref{observation:trace_positive} we have that
$ \phi_a(\veco,\veco) = \tau(a \lhs{\veco}{\veco}) \geq 0 .$
Thus, for fixed $a \geq 0$ we have shown that $\phi_a$ is a semi-inner product on $\E$. Consequently it satisfies the Cauchy--Schwarz inequality:
\begin{equation}
    |\phi_a(\veco,\vect)| \leq \phi_a(\veco,\veco)^{1/2} \phi_a(\vect,\vect)^{1/2} . \label{eq:cauchy-schwarz}
\end{equation}
From \Cref{observation:trace_positive} it also follows for $a \geq 0$ that
\begin{equation}
    \phi_a(\veco,\veco) = \tau(a \lhs{\veco}{\veco}) \leq \| a \| \tau(\lhs{\veco}{\veco}) = \| a \| \| \veco \|_{\Hi_{\E}^\tau}^2. \label{eq:bounded}
\end{equation}
Combining \eqref{eq:cauchy-schwarz} and \eqref{eq:bounded}, we arrive at
\[ |\phi_a(\veco,\vect)| \leq \| a \| \| \veco \|_{\Hi_{\E}^\tau} \| \vect \|_{\Hi_{\E}^\tau} \;\;\; \text{for $a \geq 0$.} \]
By writing a given $a \in M$ as a linear combination of positive elements in $M$, one easily deduces that
$\phi_a$ is a bounded, sesquilinear form on $\E$ for every $a \in M$, so it extends uniquely to a bounded, sesquilinear form $\phi_a$ on $\Hi_{\E}^\tau$. By Riesz' representation theorem there exists a unique bounded linear operator $\pi_\E^\tau(a)$ on $\Hi_{\E}^\tau$ such that
\[ \langle \pi_\E^\tau(a) \veco, \vect \rangle_{\Hi_{\E}^\tau} = \phi_a(\veco,\vect) \;\;\; \text{for all $\veco,\vect \in \Hi_{\E}^\tau$}. \]
Thus we get a map $\pi_\E^\tau \colon M \to \B(\Hi_{\E}^\tau)$. Note that for $a \in A$ and $\veco,\vect \in \E$ we have that
\[ \langle \pi_\E^\tau(a) \veco, \vect \rangle_{\Hi_{\E}^\tau} = \tau( a \lhs{\veco}{\vect}) = \tau(\lhs{a\veco}{\vect}) = \langle a\veco, \vect \rangle_{\Hi_{\E}^\tau} .\]
It follows that $\pi_\E^\tau$ extends the representation of $A$ on $\Hi_{\E}^\tau$. Since $\phi_a(\veco,\vect)$ is linear in $a$ for fixed $\veco,\vect \in \E$, it also follows that $\pi_\E^\tau$ is linear on $M$. Further, from what we showed earlier, we have that
\[ \langle \pi_\E^\tau(a)^* \veco, \vect \rangle = \overline{ \langle \pi_\E^\tau(a) \vect, \veco \rangle } = \overline{ \phi_a(\vect,\veco)} = \phi_{a^*}(\veco,\vect) = \langle \pi_\E^\tau(a^*) \veco, \vect \rangle \]
for $\veco, \vect \in \E$. This implies that $\pi_\E^\tau$ preserves adjoints.

Next, we claim that $\pi_\E^\tau$ is a positive map: Indeed, let $a \in M$ be positive. Then for every $\veco \in \E$ we get from \Cref{observation:trace_positive} that
$ \langle \pi_\E^\tau(a) \veco, \veco \rangle_{\Hi_{\E}^\tau} = \tau( a \lhs{\veco}{\veco}) \geq 0 .$
Writing a general $\veco \in \Hi_{\E}^\tau$ as a limit of a sequence $(\veco_n)_{n \in \N}$ in $\E$, we get that
\[ \langle \pi_\E^\tau(a) \veco, \veco \rangle_{\Hi_{\E}^\tau} = \lim_{n \to \infty} \langle \pi_\E^\tau(a) \veco_n, \veco_n \rangle \geq 0. \]

We will now show that $\pi_\E^\tau$ is normal. Since 
$\pi_\E^\tau$ is a positive linear map, it suffices to show that for any given bounded, increasing net $(a_i)_{i \in I}$ of positive elements in $M$ with s.o.t.\ limit $a$, $\pi_\E^\tau(a)$ is the s.o.t.\ limit of $(\pi_\E^\tau(a_i))_{i \in I}$. (By s.o.t., we mean the strong operator topology). Note that $(\pi_\E^\tau(a_i))_{i \in I}$ is a bounded, increasing net of positive operators in $\B(\Hi_{\E}^\tau)$, so it has a s.o.t.\ limit which we denote by $T$. Thus, we want to show that $\pi_\E^\tau(a) = T$.
For this, let $\veco,\vect \in \E$. Since $a$ is the s.o.t.\ limit of $(a_i)_i$, the net $(a_i \lhs{\veco}{\vect})_i$  converges in $M$ to $a \lhs{\veco}{\vect}$ in the s.o.t. Using the s.o.t.\ continuity of $\tau$, we get
\[ \lim_i \langle \pi_\E^\tau(a_i) \veco, \vect \rangle_{\Hi_{\E}^\tau} = \lim_i \tau( a_i \lhs{\veco}{\vect}) = \tau(a \lhs{\veco}{\vect}) = \langle \pi_\E^\tau(a)\veco, \vect \rangle_{\Hi_{\E}^\tau} \]
However, since $T$ is the s.o.t.\ limit of $(\pi_\E^\tau(a_i))_i$, we have $\lim_i \pi_\E^\tau(a_i) \veco = T \veco$, so
\[ \lim_i \langle \pi_\E^\tau(a_i) \veco, \vect \rangle_{\Hi_{\E}^\tau} = \langle \lim_i \pi_\E^\tau(a_i) \veco, \vect \rangle_{\Hi_{\E}^\tau} = \langle T \veco, \vect \rangle_{\Hi_{\E}^\tau} \]
as well. We have thus shown that $\langle \pi_\E^\tau(a)\veco,\vect \rangle = \langle T \veco, \vect \rangle$ for all $\veco,\vect \in \E$.

Next, let $\veco \in \E$ and $\vect \in \Hi_{\E}^\tau$. Let $\vect = \lim_n \vect_n$ where $\vect_n \in \E$ for all $n$. Then it follows easily by the above that $\langle T \veco, \vect \rangle_{\Hi_{\E}^\tau} = \lim_n \langle \pi_\E^\tau(a) \veco, \vect_n \rangle_{\Hi_{\E}^\tau} = \langle \pi_\E^\tau(a) \veco, \vect \rangle_{\Hi_{\E}^\tau}$.
Therefore, $\langle \pi_\E^\tau(a) \veco, \vect \rangle_{\Hi_{\E}^\tau} = \langle T \veco, \vect \rangle_{\Hi_{\E}^\tau}$ for arbitrary $\vect \in \Hi_{\E}^\tau$, so we can conclude that $\pi_\E^\tau(a) \veco = T \veco$ for all $\veco \in \E$. By density, it follows that $\pi_\E^\tau(a) = T$. This completes the proof of the normality of $\pi_\E^\tau$.

To see that $\pi_\E^\tau$ is multiplicative on $M$, we first observe that $\pi_\E^\tau$, being normal, is ultraweakly continuous. Moreover, $\pi_\E^\tau$ is multiplicative on $A$, and $A$ is ultraweakly dense in $M$. Since multiplication in a von Neumann algebra is separately continuous in each variable for the ultraweak topology, it is then straightforward to check that $\pi_\E^\tau(ab) = \pi_\E^\tau(a)\pi_\E^\tau(b)$, first for $a \in A$ and $b\in M$, and next for $a, b\in M$.  

Finally, since any normal representation of $M$ is ultraweakly continuous,
 $\pi_\E^\tau$ is clearly the only normal representation of $M$ on $\Hi_{\E}^\tau$ extending the given representation of $A$ on $\Hi_{\E}^\tau$. 
\end{proof}

\subsection{Localization of adjointable operators}\label{subsec:loc_op}
As in \Cref{subsec:loc}, let $\tau$ be a faithful tracial state on $A$ and let $M$ be the von Neumann algebra coming from the GNS construction applied to $(A,\tau)$. Given two Hilbert $M$-modules $\Hi$ and $\mathcal{H}'$, we denote by $\B_M(\Hi,\mathcal{H}')$
the bounded, $M$-linear operators from $\Hi$ into $\Hi'$.

We will make use of a procedure called \emph{localization} of adjointable operators. Parts of the statements in the following result can be found in \cite[p.~13]{AuEn20} and \cite[p.~58]{La95}.

\begin{lemma}\label{lem:loc_map}
Let $\E$, $\F$ and $\mathcal{K}$ be Hilbert $A$-modules. Then the following hold:
\begin{enumerate}[(i)]
    \item Every adjointable operator $T \colon \E \to \F$ extends uniquely to a bounded, $M$-linear map $T^{\tau} \colon \Hi_{\E}^{\tau} \to \Hi_{\F}^{\tau}$. 
    The map $T \mapsto T^{\tau}$ defines an injective, bounded, linear operator from 
    $ \mathcal{L}_A(\E,\F)$ into $\B_M(\Hi_{\E}^{\tau}, \Hi_{\F}^{\tau})$
    \item Let $T \in \mathcal{L}_A(\E,\F)$ and $S \in \mathcal{L}_A(\F,\mathcal{K})$. Then $(ST)^{\tau} = S^{\tau}T^{\tau}$ and $(T^{\tau})^* = (T^*)^{\tau}$.
\end{enumerate}
\end{lemma}

\begin{proof}
(i) Let $T \colon \E \to \F$ be an adjointable map, and denote by $\| T \|$ its operator norm as a bounded linear map between the Banach spaces $\E$ and $\F$. Then by \cite[Proposition 1.2]{La95}, we have that
\[ \lhs{T\veco}{T\veco} \leq \| T \|^2 \lhs{\veco}{\veco} \;\;\; \text{for all $\veco \in \E$.} \]
Applying $\tau$ to the above inequality, it follows that
$ \| T \veco \|_{\Hi_{\F}^{\tau}}^2 \leq \| T \|^2 \| \veco \|_{\Hi_{\E}^{\tau}}^2 $ for all $\veco \in \E$. 
Hence, $T$ extends uniquely to a bounded linear operator $T^{\tau} \colon \Hi_{\E}^{\tau} \to \Hi_{\F}^{\tau}$,
which satisfies that $\| T^{\tau} \| \leq \| T \|$, where $\| T^{\tau} \|$ denotes the operator norm of $T^{\tau}$ as a bounded linear map from $\Hi_{\E}^{\tau}$ to $\Hi_{\F}^{\tau}$. Thus $T \mapsto T^{\tau}$ defines an injective, bounded linear map $\mathcal{L}_A(\E,\F) \to \B(\Hi_{\E}^{\tau},\Hi_{\F}^{\tau})$.
For the $M$-linearity, let $a \in M$ and let $(a_i)_{i \in I}$ be a net in $A$ that converges to $a$ in the strong operator topology. Then, for any $\veco \in \E$, 
we get that
$ T^{\tau}(a\veco) = T^{\tau}(\lim_i (a_i \veco)) = \lim_i (a_i T \veco) = a T^{\tau} \veco$.
Hence, if $\veco \in \Hi_{\E}^{\tau}$, say $\veco = \lim_n \veco_n$ for
some sequence 
$(\veco_n)_{n \in \N}$ in $\E$, then
$ T^{\tau}(a \veco) = T^{\tau}( \lim_n (a \veco_n)) = \lim_n T^{\tau}(a \veco_n)  = \lim_n a T^{\tau} \veco_n = a T^{\tau} \veco$, which shows that $T^{\tau}$ is $M$-linear.

(ii) For $\veco \in \E$, 
we get that $(S^{\tau}T^{\tau})\veco = S^{\tau} T\veco = ST\veco$ since $T^{\tau}$ (resp. $S^{\tau}$) extends $T$ (resp. $S$). Thus, by uniqueness of the extension $(ST)^{\tau}$, it follows that $(ST)^{\tau} = S^{\tau} T^{\tau}$.

For $\veco \in \E$ and $\vect \in \F$, applying $\tau$ to the equality $\lhs{T\veco}{\vect} = \lhs{\veco}{T^*\vect}$ immediately gives that $\langle T\veco, \vect \rangle_{\Hi_{\F}^{\tau}} = \langle \veco, T^*\vect \rangle_{\Hi_{\E}^{\tau}}$, 
from which it readily follows that
$ \langle T^{\tau}\veco, \vect \rangle_{\Hi_{\F}^{\tau}} = \langle \veco, (T^*)^{\tau} \vect \rangle_{\Hi_{\E}^{\tau}}$ for $\veco \in \Hi_{\E}^{\tau}$ and $\vect \in \Hi_{\F}^{\tau}$. 
Hence, $(T^\tau)^* = (T^*)^{\tau}$.
\end{proof}

Setting $\E = \F = \mathcal{K}$ in \Cref{lem:loc_map}, we get
that the map $T \mapsto T^{\tau}$ from $\mathcal{L}_A(\E)$ into $\B(\Hi_{\E}^{\tau})$ is an injective $*$-homomorphism, hence an isometry, see, e.g., \cite[Theorem 3.1.5]{Mu90}. Thus,  $\mathcal{L}_A(\E)$ can be viewed as a unital $C^*$-subalgebra of $\B(\Hi_{\E}^{\tau})$, and spectral invariance is known to hold (see, e.g., \cite[Theorem 2.1.11]{Mu90}), that is,\ if $T \in \mathcal{L}_A(\E)$ is invertible as an element of $\B(\Hi_{\E}^{\tau})$, then its inverse is necessarily in $\mathcal{L}_A(\E)$.

\section{Hilbert C*-modules from projective representations of discrete groups}\label{sec:modules_projective}

This section provides a general method for the construction of Hilbert $C^*$-modules from (integrable) $\sigma$-projective representations of discrete groups. Our method follows the approach outlined in \cite{Ri85} and \cite[Section 1]{Ri88} closely, complementing it with statements on frames/Riesz sequences. 

Henceforth, $\Gamma$ denotes a countable discrete group, and $\sigma$ denotes a cocycle on $\Gamma$. In this section we set $A$ equal to $C_r^*(\Gamma,\sigma)$, the reduced twisted group $C^*$-algebra of $(\Gamma,\sigma)$, and let $\tau$ denote the canonical tracial state on $A$. We let $C_c(\Gamma, \sigma)$ denote the $*$-subalgebra of $\ell^1(\Gamma,\sigma)$ consisting of all finitely supported functions on $\Gamma$. We fix a $\sigma$-projective unitary group representation $\pi$ of $\Gamma$ on a Hilbert space $\Hip$. In this context we make the following definition:

\begin{definition}\label{def:admissible}
Let $A_0$ be a $*$-subalgebra of $\ell^1(\Gamma,\sigma)$ containing $C_c(\Gamma, \sigma)$,
and let $\Hi_0$ be a dense subspace of $\Hip$. We call the pair $(A_0, \Hi_0)$ \emph{admissible for $\pi$} if the following hold:
\begin{enumerate}
    \item For every $a \in A_0$ and $\veco \in \Hi_0$, the vector
    \begin{equation}
        a \cdot \veco := \widehat\pi(a)f=
        \sum_{\gamma \in \Gamma} a(\gamma) \pi(\gamma) \veco \label{eq:admissible_action}
    \end{equation}
    is an element of $\Hi_0$.
    \item For every $\veco,\vect \in \Hi_0$, the function $\lhs{\veco}{\vect} \colon \Gamma \to \C$ given by
    \begin{equation}
        \lhs{\veco}{\vect}(\gamma) := \langle \veco, \pi(\gamma) \vect \rangle_{\Hip} \;\;\; \text{for $\gamma \in \Gamma$} \label{eq:admissible_inner_product}
    \end{equation}
    is an element of $A_0$.
\end{enumerate}
\end{definition}

\begin{proposition}\label{prop:admissible_module}
Let $(A_0,\Hi_0)$ be an admissible pair for $\pi$. Then, with the action of $A_0$ on $\Hi_0$ given by \eqref{eq:admissible_action} and the $A_0$-valued inner product on $\Hi_0$ given by \eqref{eq:admissible_inner_product}, $\Hi_0$ becomes a left pre-inner product $A_0$-module, which can be completed to a Hilbert $A$-module $\E$.
\end{proposition}

\begin{proof}
The assumption that $(A_0,\Hi_0)$ is an admissible pair ensures that the action of $A_0$ on $\Hi_0$ is well-defined into $\Hi_0$ and that the inner product on $\Hi_0$ takes values in $A_0$. Since $A_0$ contains $C_c(\Gamma, \sigma)$, it is dense in $\ell^1(\Gamma,\sigma)$. Moreover,  $\ell^1(\Gamma,\sigma)$ is dense in $C_r^*(\Gamma,\sigma)$, and the $\ell^1$-norm dominates the $C_r^*$-norm on $\ell^1(\Gamma,\sigma)$. Hence it follows that $A_0$ is dense in $C_r^*(\Gamma,\sigma)$.

For the $A_0$-linearity in the first argument of $\lhs{\cdot}{\cdot}$, let $\veco,\vect \in \Hi_0$. Consider first $a = \delta_{\gamma}$ for some $\gamma \in \Gamma$. Note that $\lhs{\veco}{\vect}$ is simply the matrix coefficient $C_{\vect} \veco$ associated to $\pi$ as in \Cref{sub:projective}. 
Using relation \eqref{eq:matrix_coefficient_intertwining}, we get
\begin{align*}
     \lhs{ a \cdot \veco}{ \vect}(\gamma') &= 
     \langle \pi(\gamma) f , \pi(\gamma')g\rangle_{\Hip}= [C_{\vect} (\pi(\gamma)\veco)](\gamma')
      = C_{\vect} \veco (\gamma^{-1}\gamma')\, \sigma(\gamma, \gamma^{-1}\gamma') \\
      &= \lhs{\veco}{\vect}(\gamma^{-1}\gamma')\, \sigma(\gamma, \gamma^{-1}\gamma') = \big(a *_\sigma \lhs{\veco}{\vect}\big)(\gamma')
 \end{align*}
for every $\gamma' \in \Gamma$, i.e., $\lhs{ a \cdot \veco}{ \vect} = a *_\sigma \lhs{\veco}{\vect}$.  
This formula clearly extends by linearity to every $a \in C_c(\Gamma, \sigma)$. Now assume that $a\in A_0$ and pick a sequence $(a_n)_{n=1}^\infty$ in $C_c(\Gamma, \sigma)$ such that $\|a-a_n\|_1 \to 0$ as $n \to \infty$. Using what we just have shown, we get that, for every $\gamma' \in \Gamma,$
\begin{align*}
 \big|\big(\lhs{ a \cdot \veco}{\vect}-a *_\sigma \lhs{\veco}{\vect}\big)(\gamma')\big| 
 &\leq \, \big|\lhs{ (a-a_n) \cdot \veco}{\vect}(\gamma')\big|\, +\, \big|\big((a_n-a) *_\sigma \lhs{\veco}{\vect}\big)(\gamma')\big| \\
 &\leq\, \big|\langle (a-a_n) \cdot \veco, \pi(\gamma') \veco\rangle_{\Hip}\big|\, +\, \|(a-a_n)*_\sigma \lhs{\veco}{\vect}\|_1 
 \\ &\leq \,\|(a-a_n) \cdot \veco\|_{\Hip}\, \|\pi(\gamma') \veco\|_{\Hip} \,+\, \|a-a_n\|_1 \, \|\lhs{\veco}{\vect}\|_1 \\&\leq \,\|a-a_n\|_1 \,\|\veco\|^2_{\Hip} \,+\, \|a-a_n\|_1 \, \|\lhs{\veco}{\vect}\|_1\\
 & \to 0
 \end{align*}
as $n \to \infty$. Hence, $\lhs{ a \cdot \veco}{ \vect} = a *_\sigma \lhs{\veco}{\vect}$ for every $a\in A_0$, as desired.

Moreover, for each $\veco \in \Hi_0$, the function $\lhs{\veco}{\veco} \in A_0 \subseteq \ell^1(\Gamma, \sigma)$ is a diagonal matrix coefficient on $\Gamma$ 
associated to $\pi$,
hence positive in $C_r^*(\Gamma,\sigma)$ by \Cref{prop:pos_def}.

Lastly, if $\lhs{\veco}{\veco} = 0$ for some $\veco \in \Hi_0$, then 
$\| \veco \|_{\Hip}^2 = \lhs{\veco}{\veco}(e) =0$, hence $\veco = 0$. This proves all the properties needed for $\Hi_0$ to be a pre-inner product $A_0$-module as defined in \Cref{subsec:hilb_mod}.
\end{proof}

Let $(A_0,\Hi_0)$ be an admissible pair for $\pi$, as in Proposition \ref{prop:admissible_module}.
Note that for $\veco, \vect \in \Hi_0$ we have
\[ \tau(\lhs{\veco}{\vect}) = \lhs{\veco}{\vect}(e) = \langle \veco, \pi(e) \vect \rangle_{\Hip} = \langle \veco, \vect \rangle_{\Hip} .\]
This implies that the localization space $\Hi_\E^\tau$ 
can be naturally identified with 
$\Hip$.
Thus, the representation $\pi_\E^\tau$ of $A$ on $\Hi_\E^\tau$ induces a representation $\pi_r$ of $A=C_r^*(\Gamma, \sigma)$ on $\Hip$, which satisfies that $\widehat\pi = \pi_r \circ \widehat{\lambda_\Gamma^\sigma}$. (This shows that $\pi$ is weakly contained in $\lambda_\Gamma^\sigma$ whenever there exists an admissible pair for $\pi$.) 

Similarly, for $a,b \in \ell^1(\Gamma, \sigma)$, the localized inner product on $A$ as a left module over itself is given by: 
\begin{align*} \tau( \lhs{a}{b}) &= (a*b^*)(e) = \sum_{\gamma \in \Gamma} \sigma(\gamma,\gamma^{-1} e)a(\gamma)b^*(\gamma^{-1}e) = \sum_{\gamma \in \Gamma}  \sigma(\gamma,\gamma^{-1}) a(\gamma) \overline{\sigma(\gamma^{-1},\gamma)} \overline{b(\gamma)} \\
&= \langle a, b \rangle_{\ell^2(\Gamma)}. 
\end{align*}
Consequently, the localization space $\Hi_A^\tau$ of $A$, considered as a left module over itself, can be naturally identified with $\ell^2(\Gamma)$. 
It is readily verified that the representation $\pi_A^\tau$ of $A$
on $\Hi_A^\tau$ corresponds then to the identity representation of $A$ on  $\ell^2(\Gamma)$.

\begin{proposition}\label{prop:generating_multiwindow}
Let $(A_0,\Hi_0)$ be an admissible pair for $\pi$, and $\E$ be  the associated Hilbert $A$-module. Let $\vect_1, \ldots, \vect_n \in \Hi_0$. Then the following hold:
\begin{enumerate}[(i)]
    \item The finite set $\{ \vect_1, \ldots, \vect_n \}$ is an algebraic generating set for $\E$ if and only if
    \[ (\pi(\gamma) \vect_j)_{\gamma \in \Gamma, 1 \leq j \leq n} \]
    is a frame for $\Hip$.
    \item The finite set $\{ \vect_1, \ldots, \vect_n \}$ is an $A$-linearly independent set in $\E$ with closed $A$-span if and only if
    \[ (\pi(\gamma) \vect_j)_{\gamma \in \Gamma, 1 \leq j \leq n} \]
    is a Riesz sequence in $\Hip$.
\end{enumerate}
\end{proposition}

\begin{proof}
Denote by $\analysis \colon \E \to A^n$ the analysis operator associated to a finite set $\{ \vect_1, \ldots, \vect_n \} \subseteq \E$, so $\analysis \veco 
= (\lhs{\veco}{\vect_j})_{j=1}^n $ for $f \in \E$. It maps $\Hi_0$ into $\ell^1(\Gamma)^n \cong \ell^1(\Gamma \times \{1, \ldots, n \})$, and after this identification, it acts on $\Hi_0$ by
\[ \analysis \veco 
= (\langle \veco, \pi(\gamma) \vect_j \rangle )_{\gamma \in \Gamma, 1 \leq j \leq n} .\]
Thus, the action of $\analysis$ on $\Hi_0$ coincides with the action of the analysis operator $\overline{\analysis} \colon \Hip \to \ell^2(\Gamma \times \{ 1, \ldots, n \})$ associated to the system $(\pi(\gamma) \vect_j)_{\gamma \in \Gamma, 1 \leq j \leq n}$. By density, it follows that the
localized operator
$\analysis^\tau: \Hi_\E^\tau \to \Hi_{A^n}^\tau$ 
can be identified with
$\overline{\analysis}$. Similarly, the localization of the synthesis operator $\synthesis \colon A^n \to \E$ 
can be identified with the synthesis operator $\overline{\synthesis} \colon \ell^2(\Gamma \times \{1, \ldots, n \}) \to \Hip$ of $(\pi(\gamma) \vect_j)_{\gamma \in \Gamma, 1 \leq j \leq n}$. Consequently, by \Cref{lem:loc_map},  the same identifications hold for the frame operator $\frameop \in \mathcal{L}_A(\E)$ and the Gramian operator $\gramian \in \mathcal{L}_A(A^n)$.

(i) By the discussion below \Cref{lem:loc_map}, the frame operator $\frameop \in \mathcal{L}_A(\E)$ is invertible if and only if its localization, the frame operator $\overline{\frameop}$ associated to $(\pi(\gamma) \vect_j)_{\gamma \in \Gamma, 1 \leq j \leq n}$, is invertible. Invertibility of the former is equivalent to $\{\vect_1, \ldots, \vect_n \}$ being a generating set by \Cref{lem:frame_gen}, while invertibility of the latter is equivalent to $(\pi(\gamma)\vect_j)_{\gamma \in \Gamma, 1 \leq j \leq n}$ being a frame for $\Hip$. 

(ii) Similarly, the Gramian associated to $\{ \vect_1, \ldots, \vect_n \}$ is invertible if and only if the Gramian $\overline{\gramian}$ associated to $(\pi(\gamma)\vect_j)_{\gamma \in \Gamma, 1 \leq j \leq n}$ is invertible. The former invertibility is equivalent to $\{ \vect_1, \ldots, \vect_n \}$ being $A$-linearly independent with closed $A$-span by \Cref{prop:gramian_invertible}, while the latter invertibility is equivalent to $(\pi(\gamma)\vect_j)_{\gamma \in \Gamma, 1 \leq j \leq n}$ being a Riesz sequence for $\Hip$.
\end{proof}

\section{Strict comparison in C*-algebras} \label{sec:strict_comparison}
Throughout this section, $A$ denotes a unital $C^*$-algebra.

\subsection{Strict comparison of positive elements and of projections}\label{subsec:strict_comp_pos} We first recall a few facts about the comparison of positive elements in $C^*$-algebras, originally introduced by Cuntz \cite{Cun}. We follow R\o rdam \cite{Ror2} (see e.g.~\cite{APT, Ror1, Str} for alternative presentations). 

Let 
$A^+$ denote the cone of positive elements of $A$. For $m, n \in \N$, we let $M_{m,n}(A)$ denote the space of $m\times n$ matrices over $A$.
We let $M_\infty(A)^+$ denote the (disjoint) union $\bigcup_{n\in \N} M_n(A)^+$. For $a, b \in M_\infty(A)^+$, say $a \in M_n(A)^+$ and $b\in M_m(A)^+$, we say that $a$ is \emph{Cuntz subequivalent} to $b$, and write $a\preceq b $, if there exists a sequence $(r_k)_{k=1}^\infty$ in $M_{m,n}(A)$ such that $\|r_k^*br_k - a\| \to 0$ in $M_n(A)$ as $k\to \infty$. 

Assume now that $A$ 
has at least one tracial state. If  $\tau$ is a tracial state on $A$, we define $\tau:M_n(A)\to \C$  for each $n\in \N$ by $\tau([a_{ij}]) := \sum_{i=1}^n \tau(a_{ii})$. Moreover, we define $d_\tau:M_\infty(A)^+\to [0, \infty)$ by $d_\tau(a)= \lim_{k\to\infty} \tau(a^{1/k})$ whenever $a \in  M_n(A)^+$. Then we say that $A$ has \emph{strict comparison of  positive elements} whenever the following implication holds for $a, b\in M_\infty(A)^+ $:
\[ \text{If } d_\tau(a) < d_\tau(b) \text{ for every tracial state } \tau \text{ on} A, \text{ then } a\preceq b.\]
We note that this property implies that $A$ has \emph{strict comparison of projections}, in the following extended sense (compared to \cite[FCQ2, p.~22]{Bla} and \cite[Definition 11.3.8]{Str}):  

\smallskip \quad \quad If $n\in \N$ and $p, q $ are projections in $M_n(A)$ such that $\tau(p) < \tau(q)$ for every tracial 

\quad \quad state $\tau$ on $A$, then $p$ is \emph{Murray-von Neumann subequivalent} to $q$ in $M_n(A)$, i.e., 

\quad \quad there exists some $v \in M_n(A)$ 
such that $p= v^*v$ and $vv^* \leq q$. 

\medskip 
Indeed, if $A$ has strict comparison of positive elements and the projections $p, q \in M_n(A)$ satisfy the assumption above, then we readily get that $d_\tau(p) = \tau(p) < \tau(q) = d_\tau(q)$ for every tracial state $\tau \text{ on} A$, hence that $p\preceq q$. This means that there exists a sequence $(r_k)_{k=1}^\infty$ in $M_n(A)$ such that $r_k^*qr_k \to p$ as $k\to \infty$, and it is well-known that this implies that  $p$ is Murray-von Neumann subequivalent to $q$ in $M_n(A)$, cf.~\cite[Proposition 2.1]{Ror1} or \cite[Proposition 14.2.1]{Str}.  

\medskip
We next mention some conditions ensuring that strict comparison of positive elements hold whenever $A$ belongs to a certain class of $C^*$-algebras. 
As we will not work explicitly with any of the properties involved, we do not recall the lengthy definitions and simply refer the reader to Strung's book \cite{Str} for  undefined terminology in the following theorem and in the comments related to it. 

\begin{theorem}[\cite{W, Ror1, Ror2}] \label{WR} 
Let $A$ be a unital, separable, simple, nuclear, infinite-dimensional $C^*$-algebra with at least one tracial state, and let $\mathcal{Z}$ denote the Jiang-Su algebra \cite{JS}.  Consider the following conditions:
\begin{itemize}
\item[1)] $A$ has finite decomposition rank;
\item[2)]  $A$ has finite nuclear dimension;
\item[3)] $A$ is $\mathcal{Z}$-stable;
\item[4)] $A$ has strict comparison of positive elements.
\item[5)] $A$ has stable rank one.
\end{itemize}
Then $1)$ $\Rightarrow$ $2)$ $\Rightarrow$ $3)$ $\Rightarrow$ $4)$. We also have $3)$ $\Rightarrow$ $5)$.  
\end{theorem}
The first implication is by definition, the second is due to Winter, cf.~\cite[Corollary 7.3] {W}, and the third is due to R\o rdam, cf.~\cite[Corollary 4.6]{Ror2} and \cite[Theorem 15.4.6]{Str}; to be a bit more precise, this implication can be deduced from \cite[Theorem 4.5]{Ror2} by arguing in the same way as in the proof of \cite[Theorem 5.2 (a)]{Ror1}, taking into account Blackadar and Handelman's characterization of lower semi-continuous dimension functions, cf.\ \cite{BlHa82}, and Haagerup's result that quasitraces on exact $C^*$-algebras, hence on nuclear $C^*$-algebras, are traces, cf.\ \cite{Ha14}. The implication $3)$ $\Rightarrow$ $5)$ follows from \cite[Theorem 6.7]{Ror2}.

Let $A$ be as in Theorem \ref{WR}. 
  For completeness, we add that if $A$ is also assumed to have a unique tracial state (or more generally, if the extreme boundary of the tracial state space of $A$ has a finite topological dimension), then Matui and Sato have shown in \cite{MS} that 4) implies 2), i.e., conditions 2), 3) and 4) are equivalent in this case. This means that the (revised) Toms-Winter conjecture (cf.~ \cite[p. 302]{Str}) holds in this case.
 
\subsection{Relation to finitely generated Hilbert C*-modules}\label{subsec:proj_trace}

Consider the Hilbert $A$-modules $\E = A^n p$ and $\F = A^n q$ associated to some projections $p,q \in M_n(A)$. Then $p$ is Murray--von Neumann subequivalent to $q$ if and only if $\E$ is isomorphic to an orthogonally complementable submodule of $\F$, i.e.,\ there exists an $A$-submodule $\E'$ of $\F$ such that $\E \cong \E'$ and $\E' \oplus \E'^{\perp} = \F$. This is again equivalent to the existence of an adjointable isometry $\E \to \F$.

Let $\E$ be a finitely generated Hilbert $A$-module. If $A$ has a tracial state $\tau$, then 
we define \[\tau(\E) = \tau(p) = \sum_i \tau(p_{ii})\] for any projection $p=[p_{ij}] \in M_n(A)$ such that $\E \cong A^n p$. (It is not difficult to check that $\tau(p) = \tau(q)$ if we also have $\E\cong A^k q$ for some projection $q \in M_k(A)$).

The correspondence between finitely generated Hilbert $A$-modules and projections in matrix algebras over $A$ allows us to prove the following result in the presence of strict comparison of projections (in the sense defined in the previous subsection):

\begin{proposition}\label{prop:trace_generators}
Suppose $A$ 
has at least one tracial state and strict comparison of projections. Let $\E$ be a finitely generated Hilbert $A$-module and $n \in \N$. Then the following hold:
\begin{enumerate}
    \item[(i)] If $\tau(\E) < n$ for all tracial states $\tau$ on $A$, then $\E$ admits a generating set with $n$ elements.
    \item[(ii)] If $\tau(\E) > n$ for all tracial states $\tau$ on $A$, then $\E$ admits an $A$-linearly independent set with $n$ elements such that its $A$-span is closed.
\end{enumerate}
\end{proposition}
\begin{proof}
(i) Assume that $\tau(\E) < n $ for all tracial states $\tau$ on $A$.  Since $\E$ is finitely generated, we can find $k \in \N$ and a projection $p \in M_k(A)$ such that $\E \cong A^k p$. We let $I_n$ denote the $n \times n$ identity matrix in $M_n(A)$, and $0_r$ denote the zero matrix in $M_r(A)$ for $r \in \N$.
If $b \in M_n(A)$ and $c\in M_r(A)$, we denote by $b\oplus c$ the matrix in $M_{n+r}(A)$ given by 
$b\oplus c := \left(\begin{matrix} b & 0 \\ 0 &c\end{matrix}\right)$.
Set \[\widetilde p 
= \begin{cases}p \oplus 0_{n-k}& \text{if } k<n,\\
 p & \text{if } n \leq k,
 \end{cases}, \quad  
 \widetilde{I_n} = \begin{cases}  I_n & \text{if } k\leq n,\\
 I_n\oplus 0_{k-n} & \text{if } n \leq k.
 \end{cases} 
 \] 
 Further, set $m =\max(n,k)$. Then $\widetilde{p}$ and $\widetilde{I_n}$ are projections in $M_m(A)$, satisfying that 
 \[\tau(\widetilde p) = \tau(p) < n = \tau(I_n) = \tau(\widetilde{I_n})\]
 for all tracial states $\tau$ on $A$. 
 Strict comparison of projections implies that $\widetilde p \preceq \widetilde{I_n}$. In terms of Hilbert $A$-modules, this means that there exists an adjointable isometry $A^m\widetilde{p}  \to  A^m \widetilde{I_n}$, so we get an adjointable isometry \[\E \cong A^k p \cong A^m\widetilde{p} \, \to \,  A^m \widetilde{I_n} \cong A^n,\] which by \Cref{prop:generating_isometry} means that $\E$ admits a generating set consisting of $n$ elements.

(ii) Arguing similarly as in (i), we now get that  there exists an adjointable isometry $A^n \to \E$, which by \Cref{prop:generating_isometry} means that $\E$ admits an $A$-linearly independent set with $n$ elements that has closed $A$-span.
\end{proof}

Assume now that $A$ has a faithful tracial state $\tau$ and $\E$ is a finitely generated Hilbert $A$-module. Then one may use the localization procedure of \Cref{subsec:loc} to express $\tau(\E)$ in terms of the dimension of the Hilbert $M$-module $\Hi_{\E}^\tau$, where $M$ is the von Neumann algebra associated to $(A, \tau)$.
Indeed, by \Cref{prop:generating_isometry} (and its proof), we can find an adjointable isometry $\analysis \colon \E \to A^n$ for some $n \in \N$, and we then have $\Img(\analysis) = A^n p$ for some projection $p \in M_n(A)$. Considering the localization of $\analysis$ (with respect to $\tau$) as in \Cref{lem:loc_map}, we obtain a bounded, $M$-linear map 
$\analysis^\tau \colon \Hi_{\E}^\tau \to \Hi_{A^n}^\tau \cong (\Hi_{A}^\tau)^n = L^2(M,\tau)^n $. In addition, since $\analysis^*\analysis = I$, we obtain by \Cref{lem:loc_map} that $(\analysis^{\tau})^*\analysis^{\tau} = (\analysis^*\analysis)^{\tau} = I$, hence $\analysis^{\tau}$ is an isometry. 
Moreover, we have $\Img(\analysis^\tau) = L^2(M, \tau)^n p$, where we consider $p$ as a projection in $M_n(M)$. 
It follows that 
$\Hi_{\E}^\tau$ is finitely generated as a Hilbert $M$-module (cf.~\cite[Proposition 8.5.3]{AnPo}) and that its \emph{dimension} (with respect to $\tau$) is given by 
\[ \dim_{(M,\tau)} \Hi_{\E}^\tau = \tau(p), 
\]
cf.~\cite[Definition 8.5.4]{AnPo}; see also \cite[Definition 1.6]{lueck1998dimension}, and \cite[Chapter 2]{JoSu97} in the case where $M$ is a $II_1$-factor.

The above discussion yields the following result.
\begin{proposition}\label{prop:trace_dimension}
Suppose $\E$ is  finitely generated Hilbert $A$-module and $\tau$ is a faithful tracial state on $A$. Then
\[ \tau(\E) = \dim_{(M,\tau)} \Hi_{\E}^\tau .\]
\end{proposition}

\subsection{Twisted group C*-algebras of finitely generated, nilpotent groups}\label{sec:nilpotent} 
In this section we will 
extend the results from \cite{EGMK, EMK} on the finite decomposition rank (nuclear dimension) of group $C^*$-algebras associated to
finitely generated nilpotent groups
to twisted group algebras 
of such groups.
By \Cref{WR}, this will imply the presence of strict comparison of projections, 
which will allow us to exploit Proposition \ref{prop:trace_generators} 
in the setting of Section \ref{sec:modules_projective}. 

The following result 
(cf.~\cite[Theorem 3.5]{HNS}) will be useful to us for extending the relevant results in \cite{EGMK, EMK}.

\begin{theorem}[\cite{HNS}]\label{fg-nil} Let $\Gamma$ be a finitely generated nilpotent group. Then $\Gamma$ has a representation group $\widetilde\Gamma$ which is finitely generated and nilpotent. 
\end{theorem}

In \cite[Definition 1.1]{HNS}, a group $\widetilde\Gamma$ is called a \emph{representation group}\footnote{There is a similar notion of representation group for certain locally compact groups which was introduced by C.C.~Moore \cite{Moore1964}, see also \cite[Section 3]{Pac2}.} 
of a discrete group $\Gamma$ if there is a central extension $1 \to B \to \widetilde\Gamma \to \Gamma \to 1$ 
such that the associated transgression map from ${\rm Hom}(B, \C^\times)$ into the second cohomology group $H^2(\Gamma, \C^\times)$
is an isomorphism.   As pointed out in \cite[Section 3]{HNS}, see~\cite{BT, S} for more information, $\Gamma$ has always a representation group $\widetilde\Gamma$, and one may alternatively say that $ \widetilde\Gamma$ is a \emph{Schur cover} of $\Gamma$ (sometimes called  a \emph{stem cover} of $\Gamma$), meaning that 
there is a central extension \[1 \to N \to \widetilde\Gamma \to \Gamma \to 1\] such that $N$ is contained in the commutator subgroup 
of  $ \widetilde\Gamma$ and is isomorphic to the second homology group $H_2(\Gamma, \Z)$.
We let $\phi: \widetilde \Gamma \to \Gamma$ denote the homomorphism appearing in the sequence above. 

The relevance of representation groups for our purposes lies in the fact that any projective representation of $\Gamma$ corresponds to a genuine representation of  $\widetilde\Gamma$. Indeed, let  $\sigma \in Z^2(\Gamma, \T)$ and let $\pi$ be a $\sigma$-projective \emph{unitary} representation $\Gamma$ on a Hilbert space $\Hip$. Identifying the circle group $\T$ with the center $\T\cdot I_{\Hip}$ of the group of unitary operators $\mathcal{U}(\Hip)$, and letting $q: \mathcal{U}(\Hip)\to \mathcal{U}(\Hip)/\T$ denote the quotient map, we get a homomorphism $\rho_\pi: \Gamma \to \mathcal{U}(\Hip)/\T$ given by $\rho_\pi = q\circ \pi$. Since ${\rm Ext}(\Gamma_{\rm ab}, \T)= 0$ (because $\T$ is divisible, cf.~\cite[Chapter III, Proposition 2.6]{HS}),
 we may then invoke \cite[Proposition V.5.5]{S}, and deduce that there exists a homomorphism $\widetilde \pi: \widetilde\Gamma \to \mathcal{U}(\Hip)$, i.e., a unitary representation $\widetilde \pi$ of $\widetilde\Gamma$ on $\Hip$, satisfying that
$q\circ \widetilde \pi = \rho_\pi\circ \phi $. 

From the relation between $\pi$ and $\widetilde{\pi}$, it follows readily that for every $\gamma \in \widetilde \Gamma$, there exists (a unique) $\mu_{\gamma} \in \T$ such that  
\[\widetilde \pi(\gamma) = \mu_{\gamma}\, \pi(\phi(\gamma))\,.\] 
Note that $\pi$ is irreducible if and only if $\widetilde \pi$ is irreducible: Indeed, as $\phi$ is surjective, we obviously have $\pi(\Gamma)' = \widetilde \pi(\widetilde\Gamma)'$, so this is a consequence of Schur's lemma. It is also immediate that $\pi(\Gamma)$ and $\widetilde \pi(\widetilde\Gamma)$ generate the same $C^*$-algebra of operators on $\Hip$.

If $\pi$ is a $\sigma$-projective unitary representation  of $\Gamma$ on a Hilbert space $\Hip$, we will denote by $C^*(\pi(\Gamma))$ the 
$C^*$-subalgebra of $\mathcal{B}(\Hip)$ generated by $\pi(\Gamma)$. In other words, $C^*(\pi(\Gamma))= \widehat\pi(C^*(\Gamma, \sigma))$. 
Since $C^*(\Gamma, \sigma)$ is nuclear whenever $\Gamma$ is amenable \cite[Corollary 3.9]{PR}, and as any 
$*$-homomorphic image of a nuclear $C^*$-algebra is also nuclear (see for example \cite[Corollary IV.1.13]{Bla2006}),
we get in particular that $ C^*(\pi(\Gamma))$ is nuclear whenever $\Gamma$ is nilpotent. 
Moreover,
using results due to Eckhardt, Gillaspy and McKenney \cite{EGMK} and Eckhardt and Gillaspy \cite{EG} in the case of ordinary unitary representations, we obtain the following:  

\begin{theorem} \label{fg-nil2}
Let $\Gamma$ be a finitely generated nilpotent group and $\sigma \in Z^2(\Gamma, \T)$. Let $\pi$ be a $\sigma$-projective unitary representation  of $\Gamma$ on a Hilbert space $\Hip$.
Then 
 \begin{itemize}
 \item[(i)] $C^*(\pi(\Gamma))$ and $C^*(\Gamma, \sigma)\simeq C_r^*(\Gamma, \sigma)$ are
 nuclear, quasidiagonal and have finite decomposition rank\,$;$ 
 \item[(ii)] If $\pi$ is irreducible, then $C^*(\pi(\Gamma))$ is also simple, it satisfies the universal coefficient theorem $($UCT$)$, and its decomposition rank is less or equal to $1$.
 \end{itemize}
\end{theorem}

\begin{proof} (i) The assertion about nuclearity  follows from our comment above. Next, according to Theorem \ref{fg-nil}, $\Gamma$ has a representation group $\widetilde\Gamma$ which is finitely generated and nilpotent. As explained previously, there exists a unitary representation $\widetilde \pi$ of $\widetilde\Gamma$ on $\Hip$ such that $C^*(\pi(\Gamma)) = C^*(\widetilde \pi(\widetilde \Gamma))$.  
Now,
\cite[Theorem 5.1]{EGMK} gives that $C^*(\widetilde\Gamma)$ has finite decomposition rank. Using \cite[(3.3)]{KW}, we deduce that 
any $*$-homomorphic image of $C^*(\widetilde\Gamma)$, in particular 
$C^*(\widetilde \pi(\widetilde\Gamma))$, 
has finite decomposition rank. 
Since $C_r^*(\Gamma, \sigma) = C^*(\lambda_\Gamma^\sigma(\Gamma)) = C^*(\widetilde{\lambda_\Gamma^\sigma}(\widetilde\Gamma))$, this implies that $C^*(\Gamma, \sigma)\simeq C_r^*(\Gamma, \sigma)$ has finite decomposition rank. Since any separable $C^*$-algebra with finite decomposition rank is quasidiagonal (cf.~\cite[Theorem 17.4.3]{Str}), we have shown (i).

(ii) Assume now that $\pi$ is irreducible. Then $\widetilde \pi$ is irreducible too. As any primitive ideal of $C^*(\widetilde\Gamma)$ is maximal \cite{Pog}, we get that $C^*(\widetilde \pi(\widetilde\Gamma))$ is simple. Further, \cite[Theorem 3.5]{EG} gives that $C^*(\widetilde \pi(\widetilde \Gamma))$ satisfies the UCT. Finally, \cite[Theorem 6.2]{EGMK} gives that the decomposition rank of $C^*(\widetilde \pi(\widetilde\Gamma))$ is less or equal to $1$. 
\end{proof}

\begin{remark}
By definition, the decomposition rank of a separable $C^*$-algebra bounds its nuclear dimension. Thus, in Theorem \ref{fg-nil2}
we could replace \emph{finite decomposition rank} with \emph{finite nuclear dimension}.
To show that this weaker property holds, we could then have invoked \cite[Theorem 4.4]{EMK} instead of quoting \cite[Theorem 5.1]{EGMK}. 
\end{remark}

\begin{theorem} \label{fg-nil3}
Let $\Gamma$ be a 
finitely generated nilpotent group and $\sigma \in Z^2(\Gamma, \mathbb{T})$.  
Assume that $(\Gamma, \sigma)$ satisfies Kleppner's condition.

Then $C^*(\Gamma, \sigma)\simeq C_r^*(\Gamma, \sigma)$ is a $($unital, separable, nuclear$)$ simple, quasidiagonal $C^*$-algebra with a unique tracial state, which satisfies the UCT and has decomposition rank less or equal to $1$.   
\end{theorem}

\begin{proof} 
The fact that $C^*(\Gamma, \sigma)\simeq C_r^*(\Gamma, \sigma)$ is simple with a unique tracial state whenever $\Gamma$ is nilpotent and $(\Gamma, \sigma)$ satisfies Kleppner's condition is due to Packer, cf.~\cite{P}. Let now $\pi$ be any $\sigma$-projective irreducible unitary representation of $\Gamma$. Since $C^*(\Gamma, \sigma)$ is simple, we have $C^*(\Gamma, \sigma)\simeq C^*(\pi(\Gamma))$, which we know is quasidiagonal and has decomposition rank less or equal to $1$ from 
Theorem \ref{fg-nil2}.  
\end{proof}
 
Combining Theorem \ref{WR} and Theorem \ref{fg-nil3}, we finally get:

\begin{corollary}\label{fg-nil4}
Let $\Gamma$ be a finitely generated nilpotent group, $\sigma \in Z^2(\Gamma, \mathbb{T})$, and assume  
 $(\Gamma, \sigma)$ satisfies Kleppner's condition. Then $C^*(\Gamma, \sigma)\simeq C_r^*(\Gamma, \sigma)$ has 
strict comparison of positive elements and  stable rank one. 
In particular, $C_r^*(\Gamma, \sigma)$
has strict comparison of projections (in our extended sense).
\end{corollary}

Note that a combination of \Cref{fg-nil3} and \Cref{fg-nil4} directly provides \Cref{thm:intro_fnd}. Lastly, we mention how to obtain \Cref{thm:classifiable} from these results.

\begin{proof}[Proof of \Cref{thm:classifiable}]
As shown by Tikuisis, White and Winter, cf.~\cite[Corollary D]{TWW}, the class $\mathcal{C}$ of all separable, unital, simple, infinite-dimensional $C^*$-algebras with finite nuclear dimension and which satisfy the UCT is classified by the Elliott invariant. Now Theorem \ref{fg-nil3} gives that $C_r^*(\Gamma, \sigma)$ belongs to $\mathcal{C}$  whenever $\Gamma$ and $\sigma$ satisfy the assumptions of this theorem (and $\Gamma$ is infinite).
\end{proof}

\section{Lattice orbits of nilpotent Lie groups}\label{sec:lattice_orbits}
Let $(\pi, \Hpi)$ be a projective representation of a nilpotent Lie group $G$ on $\Hip$. For a lattice $\Gamma \leq G$ and a vector $\vect \in \Hip$, we consider the system of vectors
\begin{align} \label{eq:lattice_orbit}
 \pi(\Gamma) \vect = (\pi(\gamma) \vect )_{ \gamma \in \Gamma }. 
\end{align}
A system \eqref{eq:lattice_orbit} will be treated as a $\Gamma$-indexed family, possibly with repetitions.

In this section the results obtained in the previous sections are applied to the restriction $\pi|_{\Gamma}$ of  $\pi$ to $\Gamma$ and an explicitly constructed associated Hilbert $C^*$-module.

\subsection{Relative discrete series  and projective representations} \label{sec:relativeDS}
Let $N$ be a connected, simply connected nilpotent Lie group 
and let $(\pi, \Hip)$ be an irreducible unitary representation of $N$. Denote by
$
P_{\pi} = \big\{ x \in N : \pi (x) \in \mathbb{C} \cdot I_{\Hip} \big\}
$
the projective kernel of $\pi$. Then $P_{\pi} \leq N$ forms a connected, simply connected normal subgroup. Assume that $(\pi, \Hip)$ is square-integrable modulo $P_{\pi}$, i.e.,
there exist $\veco, \vect \in \Hpi \setminus \{0\}$ such that
\begin{align} \label{eq:SIP}
\int_{N/\pker} | \langle \veco, \pi(x) \vect \rangle |^2 \; d\mu_{N/\pker} (x \pker) < \infty. 
\end{align}
Since $\pi (x) = \chi (x) I_{\Hpi}$ for a character $\chi \in \widehat{\pker}$, the integrand $N \ni x \mapsto | \langle \veco, \pi(x) \vect \rangle | \in [0,\infty)$ 
in \eqref{eq:SIP} defines a function on $N / \pker$. 

The orthogonality relations for $(\pi, \Hpi)$ yields that there exists a unique $d_{\pi} > 0$ such that
\begin{align} \label{eq:ortho_P}
\int_{N/\pker} | \langle \veco, \pi(x) \vect \rangle |^2 \; d\mu_{N/\pker} (x \pker) = d_{\pi}^{-1} \| \veco \|_{\Hpi}^2 \| \vect \|_{\Hpi}^2, \quad \veco, \vect \in \Hpi;
\end{align}
see, e.g.,~\cite{moore1974square, pedersen1994matrix, corwin1990representations}. An irreducible representation $\pi$
that is square-integrable modulo $\pker$ is called a \emph{relative discrete series representations}; this will be denoted by $\pi \in \SIP$. 

In particular, if $\pi$ is irreducible and square-integrable modulo the center $Z$ of $N$, then $\pi \in \SIP$, and $\pker = Z$ by \cite[Theorem 3.2.3]{corwin1990representations} and \cite[Corollary 4.5.4]{corwin1990representations}. 

A representation $\pi \in \SIP$ can be treated as a 
square-integrable projective representation of $N/\pker$: 
Given a smooth cross-section $s : N/\pker \to N$ for the quotient map $p : N \to N/\pker$, the mapping
\begin{align} \label{eq:projective}
\pi' : N/\pker \to \mathcal{U}(\Hpi), \quad x\pker \mapsto \pi(s(xP_{\pi}))
\end{align}
defines an irreducible projective unitary representation of $N/\pker$. The assumption on $\pi$ yields that $(\pi', \Hpi)$ is square-integrable on $N/\pker$ in the strict sense, i.e., $\langle \vect_1, \pi'(\cdot) \vect_2 \rangle \in L^2 (N/\pker)$ for all $\vect_1, \vect_2 \in \Hpi$. The constant $d_{\pi} > 0$ in \eqref{eq:ortho} coincides with the formal dimension $d_{\pi'}$ of $\pi'$ normalized according to Haar measure on $N/\pker$ (see Section \ref{sub:projective}). 
A different choice of cross-section yields equivalent projective unitary representations (cf.~\cite{Ani06} for details).

A square-integrable projective representation $\pi'$ obtained via a cross-section as in \eqref{eq:projective} will be referred to as a \emph{
projective relative discrete series representation} of the connected, simply connected nilpotent Lie group $G = N/\pker$. 
For simplicity, it will often also be written $\pi = \pi'$.

\subsection*{Notation} Throughout, unless stated otherwise, any nilpotent Lie group $G$ is assumed to be connected and simply connected. The Lie algebra of $G$ is denoted by $\mathfrak{g}$ and its dimension by $d$. The associated exponential map is denoted by $\exp_G : \mathfrak{g} \to G$ and forms a global diffeomorphism. The Schwartz space on $G$ consists of all $F : G \to \mathbb{C}$ such that $F \circ \exp_G \in \mathcal{S}(\mathfrak{g})$.

\subsection{Smooth vectors and matrix coefficients} 
Let $(\pi, \Hpi)$ be an irreducible unitary representation of a nilpotent Lie group $N$.
The space of smooth vectors $\Hs$ consists of all vectors $g \in \Hpi$ such that the orbit maps $N \ni x \mapsto \pi(x) g \in \Hpi$ are smooth. The Lie algebra $\mathfrak{n}$ acts on $\Hs$ via the derived representation 
\[
d\pi(X) g = \frac{d}{dt} \bigg|_{t = 0} \pi(\exp_N (tX)) g, \quad X \in \mathfrak{n}, \; g \in \Hs.
\]
For a basis $\{X_1, ..., X_d\}$ for $\mathfrak{n}$, a family of semi-norms in $\Hs$ is defined by
\[
\| g \|_{\beta} := \| d\pi(X^{\beta}) g \|_{\Hpi} = \| d\pi(X_1^{\beta_1}) \cdots d\pi(X_d^{\beta_d}) g \|_{\Hpi}, \quad \beta \in \mathbb{N}_0^d.
\]
The space $\Hs$ is $\pi$-invariant and is norm dense in $\Hpi$, see, e.g., \cite[Appendix A.1]{corwin1990representations}.

For smooth vectors, the associated matrix coefficients of a square-integrable representation define Schwartz functions, see, e.g., \cite{howe1977on, pedersen1994matrix, corwin1996lp, corwin1990representations} for different versions. For our purposes, the following version is most convenient, cf.~\cite[Theorem 2.6]{pedersen1994matrix} and \cite[Remark 2.2.8]{pedersen1994matrix}. 

\begin{lemma}[\cite{pedersen1994matrix}] \label{lem:smooth_matrix}
Let $\pi \in \SIP$.
If $\veco, \vect \in \Hs$, then $C_{\vect} \veco = \langle \veco, \pi' (\cdot) \vect \rangle \in \mathcal{S} (N /\pker)$. 
\end{lemma}
\begin{proof} The result follows from the general theorem \cite[Theorem 2.6]{pedersen1994matrix} in the following manner.

Let $\mathfrak{p} \subseteq \mathfrak{n}$ denote the Lie algebra of $\pker$. As in \cite[Remark 2.2.8]{pedersen1994matrix}, let $\mathfrak{n}_e \subset \mathfrak{n}$ be a subspace of even dimension, so that the orthogonal decomposition $\mathfrak{n} = \mathfrak{n}_e \oplus \mathfrak{p}$  gives rise to the diffeomorphism $\phi : \mathfrak{n}_e \to N/ P_{\pi}, \; X \mapsto \exp_N(X) \pker$ (cf.~\cite[Section 1.2]{corwin1990representations}), with inverse  $\phi^{-1} : N/\pker \to \mathfrak{n}_e$ given by $ \exp_N(X) \pker \mapsto X$. Hence, a (smooth) cross-section $s : N/\pker \to N$ for $p : N \to N/\pker$ is given by $s(\exp_N (X) \pker) = \exp_N(X)$, i.e., $p \circ s = \id_{N/\pker}$. For this cross-section, denote by $\pi'$ the projective representation of $N/\pker$ as defined in \eqref{eq:projective}.  

If $f, g \in \Hs$, then \cite[Theorem 2.6]{pedersen1994matrix} yields that 
$\mathfrak{n}_e \ni X \mapsto \langle f, \pi (\exp_N(X)) g \rangle \in \mathbb{C} $
is in $\mathcal{S} (\mathfrak{n}_e)$. 
A direct calculation using $\exp_{N/\pker} (X + \mathfrak{p}) = p (\exp_{N} (X))$ and the definition of $\pi'$ shows
\[ \langle \veco , \pi'(\exp_{N/\pker} (X + \mathfrak{p})) \vect \rangle 
= \langle \veco, \pi' (\exp_N (X) \pker) \vect \rangle 
= \langle \veco, \pi (\exp_N (X) ) \vect \rangle, \quad X \in \mathfrak{n}.
\]
Therefore, $X + \mathfrak{p} \mapsto \langle \veco , \pi'(\exp_{N/\pker} (X + \mathfrak{p})) \vect \rangle$ defines a Schwartz function on $\mathfrak{n} / \mathfrak{p} \cong \mathfrak{n}_e$, i.e., $\langle f, \pi'(\cdot) g \rangle \in \mathcal{S} (N/\pker)$.
\end{proof}

Lemma \ref{lem:smooth_matrix} allows to prove a convenient characterization of the space $\Hs$. For this and other purposes, a family of semi-norms on $\mathcal{S} (N / \pker)$ defined via polynomial weights and left-invariant differential operators will be used, cf.~\cite{schweitzer1993dense, hulanicki1984functional, ludwig1988minimal} for more details on what follows.

Let $U$ be fixed a symmetric, compact generating set for $G = N/\pker$ and define the length function $\tau : G \to [0,\infty)$ by
\[
\tau (x) = \min \{ n \in \mathbb{N}_0 \; | \; x \in U^n \},
\]
with $U^0 := \{e\}$.
Then $\tau(xy) \leq \tau(x) + \tau(y)$, $\tau(x^{-1}) = \tau(x)$ and $\tau(e) = 0$ for $x,y \in G$. Given $\alpha \in \mathbb{N}_0$, let $w_{\alpha} : G \to [1,\infty)$ be defined as $x \mapsto (1+\tau(x))^{\alpha}$ 
and define $L^p_{w_{\alpha}} (G)$ to be the collection of all $F \in L^p (G)$ such that $\| F \|_{L^p_{w_{\alpha}}} := \| w_{\alpha} \cdot F \|_{L^p} < \infty$. 

Let $\mathfrak{D}(G)$ be the (unital) algebra of all left-invariant differential operators on $G$, i.e., all linear operators $D : C^{\infty} (G) \to C^{\infty} (G)$ of the form $D = \sum_{\beta \in \mathbb{N}_0^d} c_{\beta} X^{\beta}$, where all but finitely many $c_{\beta} \in \mathbb{C}$ are zero and $X^{\beta} = X^{\beta_1}_1 \cdots X^{\beta_d}_d$ for a basis $\{X_1, ..., X_d\}$ for $\mathfrak{g}$.

Let $p \in [1,\infty]$. Then a function $F \in C^{\infty} (G)$ belongs to $\mathcal{S}(G)$ if, and only if, for all $D \in \mathfrak{D} (G)$ and $\alpha \in \mathbb{N}_0$,
\begin{align} \label{eq:schwartz_seminorm}
\| F \|_{D,\alpha, p} := \big\|  D F \big\|_{L^p_{w_{\alpha}}} < \infty.
\end{align}
The space $\mathcal{S}(G)$ is independent of the choice of the neighborhood $U$ and the exponent $p$, cf.
\cite{schweitzer1993dense, hulanicki1984functional, ludwig1988minimal}.

\begin{proposition} \label{prop:smooth_vectors}
Let $\pi \in \SIP$. For $\vect \in \Hs \setminus \{0\}$ and $\alpha \in \mathbb{N}_0$, let
\[\Hik := \bigg\{ \veco \in \Hpi \; : \; \| C_g f \|_{L^1_{w_{\alpha}}} = \int_{G} |C_g f (x)| w_{\alpha} (x) \; d\mu_{G} (x) < \infty \bigg\}. \]
Then
$
\Hs = \bigcap_{\alpha \in \mathbb{N}_0} \Hik.
$
\end{proposition}
\begin{proof}
As in the proof of Lemma \ref{lem:smooth_matrix}, consider the orthogonal decomposition $\mathfrak{n} = \mathfrak{n}_e \oplus \mathfrak{p}$ and the diffeomorphism $\phi : \mathfrak{n}_e \to N/\pker, \; X \mapsto \exp_N (X) \pker$. Let $\pi'$ denote the projective representation of $G = N/P_{\pi}$ 
defined via the section $s : N/\pker \to N, \; s(\exp_N (X) \pker) = \exp_N (X)$ as in \eqref{eq:projective}. 

If $\veco \in \Hs$, then $C_{\vect} \veco \in \mathcal{S} (G)$ by Lemma \ref{lem:smooth_matrix}. In particular, using the semi-norms \eqref{eq:schwartz_seminorm} with $p = 1$, yields directly that $C_{\vect} \veco \in L^1_{w_{\alpha}} (G)$ for all $\alpha \in \mathbb{N}_0$. 
For the converse, let $f \in \bigcap_{\alpha \in \mathbb{N}_0} \Hik$ and let $\vect \in \Hs \setminus \{0\}$ be normalized 
such that  
\begin{align} \label{eq:repro_form}
f = \int_{G} \langle f, \pi' (x) g \rangle \pi'(x) g \; d\mu_G (x) = \int_{N/\pker} \langle f, \pi (s(xP_{\pi})) g \rangle \pi(s(xP_{\pi})) g \; d\mu_{N/\pker} (x\pker);
\end{align}
 cf.~the orthogonality relations ~\eqref{eq:ortho} and \eqref{eq:ortho_P}. By 
\cite[Theorem 1.2.10]{corwin1990representations}, the map $\phi : \mathfrak{n}_e \to N/\pker$ transforms the Lebesgue measure $dY$ on $\mathfrak{n}_e$ to Haar measure $\mu_G$ on $G = N/\pker$. Therefore, the reproducing formula \eqref{eq:repro_form} and the change-of-variables formula yields
\[
f = \int_{\mathfrak{n}_e} \langle f, \pi(\exp_N (Y)) g \rangle \pi(\exp_N (Y)) g \; dY.
\]
Given a basis $\{X_1, ..., X_d \}$ of $\mathfrak{n}$ and a multi-index $\beta \in \mathbb{N}_0^d$, a direct calculation entails then that
$
d\pi(X^{\beta}) f = \int_{\mathfrak{n}_e} \langle f, \pi (\exp_N (Y)) g \rangle d\pi(X^{\beta}) \pi (\exp_N(Y)) g \; dY.  
$ Since $g \in \Hs$, it follows as in the proof of \cite[Lemma A.1.1]{corwin1990representations} that
\[
d\pi(X^{\beta}) \pi (\exp_N (Y)) g = \pi(\exp_N(Y)) \sum_{|\beta'| \leq |\beta|} p_{\beta'} (\exp_N (Y)) d\pi(X^{\beta'}) g,
\]
where $p_{\beta'}$ are polynomial functions on $N$, i.e., $p_{\beta'} \circ \exp_N$ is a polynomial on $\mathfrak{n}$. Combining these identities with norm estimates for vector-valued integrals (see, e.g., \cite[Theorem A.22]{folland2016course}) gives
\begin{align*} 
&\| d\pi(X^{\beta}) f \|_{\Hpi} \\
&\quad \quad \quad \quad \leq  \sum_{|\beta'| \leq |\beta|} \| d\pi (X^{\beta'}) g \|_{\Hpi} \int_{\mathfrak{n}_e} | \langle f, \pi (\exp_N (Y)) g \rangle | |p_{\beta'} (\exp_N (Y)) | \; dY \\
& \quad \quad \quad \quad =  \sum_{|\beta'| \leq |\beta|} \| d\pi (X^{\beta'}) g \|_{\Hpi}  \int_{N/\pker} | \langle f, \pi (s(xP_{\pi})) g \rangle | |p_{\beta'} (s(xP_{\pi})) | \; d\mu_{N/\pker} (x\pker). \numberthis \label{eq:smooth_seminorm_finite}
\end{align*}
By the assumption $g \in \Hs$, we have $C_1 := \max_{|\beta'| \leq |\beta|} \| d\pi(X^{\beta'}) g\|_{\Hpi}  < \infty$. 
Since $p_{\beta'} \circ \exp_N$ is a polynomial on $\mathfrak{n}_e$, it follows by the identity $p_{\beta'} (\exp_N (Y)) = p_{\beta'} (s(\exp_{N/\pker} (Y + \mathfrak{p})))$ that $ p_{\beta'} \circ s$ is a polynomial function on $G = N/\pker$. 
By \cite[Section 1.5]{ludwig1988minimal} or \cite[Section 3.5]{ludwig1995algebre}, 
any polynomial function $p$ on $N/\pker$ is comparable to the polynomial weight $w = 1+\tau(\cdot)$ in the sense that there exist $C_2, M > 0$ (depending on $p$) such that $p(x) \leq C_2 w(x)^M$ for $x \in G$.
Choosing $\alpha > 0$ sufficiently large, it follows therefore easily from this and  \eqref{eq:smooth_seminorm_finite} that there exists $C > 0$ such that
\begin{align} 
\| d\pi(X^{\beta}) f \|_{\Hpi} \leq C  \int_{G} | \langle f, \pi (s(x)) g \rangle | w_{\alpha} (x) \; d\mu_{G} (x) < \infty.
\end{align}
 Since $\beta \in \mathbb{N}_0^d$ was chosen arbitrary, it follows that $f \in \Hs$.
\end{proof}

The arguments used in the proof of \Cref{prop:smooth_vectors} are reminiscent of some arguments used in
\cite[Section 11.2]{groechenig2001foundations} and \cite[Section 2.2]{beltita2011modulation}, which require different assumptions than used here.

\subsection{Mapping properties for smooth vectors}
Henceforth, $\pi = \pi'$ will denote a 
projective relative discrete series representation of $G = N / \pker$ 
and $\Gamma \leq G$ will denote a lattice.

The Schwartz space $\mathcal{S} (\Gamma)$ on the discrete subgroup $\Gamma \leq G$ 
is defined by
\[
\mathcal{S}(\Gamma) = \bigg\{ c \in \mathbb{C}^{\Gamma} : \sum_{\gamma \in \Gamma} |c_{\gamma} | w_{\alpha} (\gamma) < \infty, \; \forall \alpha \in \mathbb{N}_0 \bigg\}, 
\]
and equipped with the semi-norms $\| c \|_{\alpha} := \| c \|_{\ell^1_{w_{\alpha}}} = \| w_{\alpha} \cdot c \|_{\ell^1}$ for $\alpha \in \mathbb{N}_0$; see \cite{ji1992smooth, jolissaint1990rapidly, schweitzer1993dense}.

The following result shows that the action of the analysis (resp.\ synthesis) operator associated to smooth vectors is well-defined into (resp.\ on) the space $\mathcal{S}(\Gamma)$.

\begin{proposition}\label{prop:admissible_nilpotent}
Let $(\pi, \Hpi)$ be a projective relative discrete series representation of a nilpotent Lie group $G$. 
Suppose that $\Gamma \leq G$ is a lattice. 
Then the following assertions hold:
\begin{enumerate}[(i)]
    \item For all $\veco, \vect \in \Hs$, the mapping 
    \[ \Gamma \ni \gamma \mapsto \langle \veco, \pi(\gamma) \vect \rangle \in \mathbb{C} \]
    defines an element of $\mathcal{S}(\Gamma)$.
    \item For $(c_{\gamma})_{\gamma \in \Gamma} \in \mathcal{S} (\Gamma)$ and $\vect \in \Hs$, 
    the series 
    \[ \sum_{\gamma \in \Gamma} c_{\gamma} \pi(\gamma) \vect \]
    defines an element of $\Hs$.
\end{enumerate}
Thus $(\mathcal{S}(\Gamma), \Hs)$ forms an admissible pair for $\pi|_{\Gamma}$ in the sense of \Cref{def:admissible}.
\end{proposition}
\begin{proof}
The space $L^1_{w_{\alpha}} (G)$ is invariant under translations $L_x F := F(x^{-1} \cdot)$ and $R_x F := F(\cdot x)$ for $x \in G$, with 
 $\|L_x \|_{\mathcal{B}(L^1_{w_{\alpha}})}, \| R_x \|_{\mathcal{B}(L^1_{w_{\alpha}})} \leq w_{\alpha} (x)$, see, e.g., \cite[Proposition 3.7.6]{reiter2000classical}. In particular, the weight $w_{\alpha} : G \to [1,\infty)$ is a control weight for $L^1_{w_{\alpha}}(G)$ in the sense of \cite[Section 3]{feichtinger1989banach}. Throughout, let $\Hik$ be as defined in  \Cref{prop:smooth_vectors}, see \cite{feichtinger1989banach, christensen1996atomic} for basic properties.

(i)  Let $\alpha \in \mathbb{N}_0$. If $\veco, \vect \in \Hs$, then $C_{\vect} \veco \in \mathcal{S}(G)$ by Lemma \ref{lem:smooth_matrix}. In particular, this implies, by using the semi-norms \eqref{eq:schwartz_seminorm} with $p = \infty$, that, for all $\alpha' \in \mathbb{N}_0$, there exists $C_{\alpha'} > 0$ such that $|C_{\vect} \veco (x)| \leq C_{\alpha'} (1 + \tau(x) )^{-\alpha'}$ for $x \in G$. Since $(1 + \tau(\cdot))^{-\alpha'} \in L^1_{w_{\alpha}} (G)$ for a sufficiently large $\alpha' \geq \alpha$ (cf.~\cite[Proposition 1.5.1]{schweitzer1993dense}), the submultiplicativity and local boundedness of  $w = (1+\tau(\cdot))$ yields that 
\begin{align*} 
\int_G \sup_{y \in V} |C_{\vect} \veco (x y)| w_{\alpha} (x) \; d\mu_G (x) &\leq C_{\alpha'}
\int_G \sup_{y \in V} (1+\tau(xy))^{-\alpha'} w_{\alpha} (x) \; d\mu_G (x) 
\\
&\leq C_{\alpha'} \sup_{y \in V} (1 + \tau(y^{-1}))^{\alpha'} \int_G (1+\tau(x))^{-\alpha'} w_{\alpha} (x) \; d\mu_G (x)  \numberthis \label{eq:amalgam} \\
&< \infty
\end{align*}
for any relatively compact unit neighborhood $V \subset G$.
The property \eqref{eq:amalgam} allows an application of \cite[Lemma 3.8]{feichtinger1989banach}, which yields that $ (\langle \veco, \pi (\gamma) \vect \rangle )_{\gamma \in \Gamma}\in \ell^1_{w_{\alpha}} (\Gamma)$. Since $\alpha \in \mathbb{N}_0$ was chosen arbitrary, it follows that $(\langle \veco, \pi (\gamma) \vect \rangle )_{\gamma \in \Gamma} \in \bigcap_{\alpha \in \mathbb{N}_0} \ell^1_{w_{\alpha}} (\Gamma) = \mathcal{S} (\Gamma)$.

(ii) If $(c_{\gamma} )_{\gamma \in \Gamma} \in \mathcal{S}(\Gamma)$ and $\vect \in \Hs$, 
then $(c_{\gamma} )_{\gamma \in \Gamma} \in \ell^1_{w_{\alpha}} (\Gamma)$ and $\vect$ satisfies \eqref{eq:amalgam} with the choice $\vect = \veco$ for all $\alpha \in \mathbb{N}_0$. Hence, by  \cite[Proposition 5.2]{feichtinger1989banach} or \cite[Theorem 6.1]{christensen1996atomic}, the mapping $(c_{\gamma})_{\gamma \in \Gamma} \mapsto \sum_{\gamma \in \Gamma} c_{\gamma} \pi(\gamma) \vect$ is bounded from $\ell^1_{w_{\alpha}} (\Gamma)$ into $\Hik$ for $\alpha \in \mathbb{N}_0$. This shows $\sum_{\gamma \in \Gamma} c_{\gamma} \pi(\gamma) \vect \in \bigcap_{\alpha \in \mathbb{N}_0} \Hik = \Hs$ by \Cref{prop:smooth_vectors}.

For 
the admissibility of the pair $(\mathcal{S}(\Gamma), \Hs)$ for $\pi|_{\Gamma}$,
it is obvious that $C_c(\Gamma, \sigma)$ is contained in $\mathcal{S}(\Gamma)$, so 
it remains only to show that  $\mathcal{S}(\Gamma)$ is a 
$*$-subalgebra of $\ell^1 (\Gamma, \sigma)$.
It is straightforward to see that $\mathcal{S}(\Gamma)$ is closed under twisted involution. For the algebra property,  note that $\mathcal{S}(\Gamma) = \bigcap_{\alpha \in \mathbb{N}_0} \ell^1_{w_{\alpha}} (\Gamma)$ and that each $\ell^1_{w_{\alpha}} (\Gamma)$, where $\alpha \in \mathbb{N}_0$, is an ordinary convolution algebra. Since $|c \ast_{\sigma} d| \leq |c| \ast |d|$ for $c, d \in \ell^1 (\Gamma)$, it follows readily that $\mathcal{S}(\Gamma)$ is also closed under twisted convolution.
\end{proof}

\subsection{Finitely generated modules associated to lattices}
This section is devoted to the construction of a Hilbert $C^*$-module from $\Hs$. The following observation will guarantee that this module is finitely generated.

\begin{proposition}\label{prop:multiwindow}
Let $(\pi, \Hpi)$ be a projective relative discrete series representation of a nilpotent Lie group $G$. Suppose that $\Gamma \leq G$ is a lattice. Then there exists a finite family $(\vect_j)_{j = 1}^n$ of vectors $\vect_j \in \Hs$ such that $(\pi(\gamma) \vect_j)_{\gamma \in \Gamma, 1 \leq j \leq n}$ is a frame for $\Hip$.
\end{proposition}
\begin{proof}
Let $\vect \in  \Hs \setminus \{0\}$, so that $C_g g$ satifies the property \eqref{eq:amalgam} with $g = f$. Then, by \cite[Theorem 6.4]{christensen1996atomic} or \cite[Section 4]{groechenig1991describing}, there exists a compact unit neighborhood $U \subset G$ such that for any discrete family $\Lambda$ in $G$ satisfying $G = \bigcup_{\lambda \in \Lambda} \lambda U$ and $\sup_{x \in G} | \Lambda \cap xU | < \infty$, 
\[ 
\| f \|_{\Hip}^2 \asymp \sum_{\lambda \in \Lambda} | \langle f, \pi (\lambda) \vect \rangle |^2, \quad f \in \Hip.
\]
Since $G$ is a nilpotent Lie group, $\Gamma \leq G$ is also co-compact, see, e.g., \cite[Corollary 5.4.6]{corwin1990representations}. Hence, there exists a relatively compact fundamental domain $\Sigma \subset G$ for $\Gamma$. Let $(x_j U)_{j = 1}^n$ be a finite cover of $\Sigma$. Then $\Lambda' := \{\gamma x_j : \gamma \in \Gamma, j = 1, \ldots ,n\}$ satisfies 
$G = \bigcup_{\lambda' \in \Lambda'} \lambda' U$ and $\sup_{x \in G} |\Lambda' \cap xU| < \infty$, 
so that
\[ 
\| f \|_{\Hip}^2 \asymp \sum_{\lambda \in \Lambda'} | \langle f, \pi (\lambda') \vect \rangle |^2 
= \sum_{j = 1}^n \sum_{\gamma \in \Gamma} | \langle f, \pi (\gamma) \pi(x_j) \vect \rangle |^2, \quad f \in \Hip.
\]
Therefore, defining $\vect_j := \pi(x_j) \vect \in \Hs$ for $j = 1, \ldots , n$, gives the desired result.
\end{proof}

The existence of localized multi-window Gabor frames was proven in \cite[Theorem 4.6]{Lu09} via a correspondence to projective modules over non-commutative tori. The proof of Proposition \ref{prop:multiwindow} shows that this is also a direct consequence of the classical sampling techniques \cite{feichtinger1989banach, christensen1996atomic}. 

\begin{theorem}\label{thm:trace}
Let $(\pi, \Hpi)$ be a $\sigma$-projective relative discrete series representation of a nilpotent Lie group $G$ of formal dimension $d_{\pi} > 0$. Suppose that $\Gamma \leq G$ is a lattice.
Then 
$(\mathcal{S}(\Gamma), \Hip^\infty)$ is  an admissible pair for  $\pi|_{\Gamma}$ in the sense of \Cref{def:admissible},
so that $\Hip^\infty$ can be completed into a Hilbert $C_r^*(\Gamma,\sigma)$-module $\E$. The module $\E$ is finitely generated, and if $\tau$ denotes the canonical tracial state on  $C_r^*(\Gamma,\sigma)$, then
\[ \tau(\E) =  \vol(G/\Gamma) d_{\pi}.\]
$($The constant $\vol(G/\Gamma) d_{\pi}$ is independent of the choice of Haar measure on $G$.$)$
\end{theorem}

\begin{proof}
Admissibility of the pair $(\mathcal{S}(\Gamma), \Hip^\infty)$ was proved in \Cref{prop:admissible_nilpotent}. Combining \Cref{prop:multiwindow} with \Cref{prop:generating_multiwindow}, we get that $\E$ is finitely generated.

By the discussion preceding Proposition \ref{prop:generating_multiwindow}, the localization space 
$\Hi_\E^\tau$ of $\E$ with respect to $\tau$ can be naturally identified with $\Hip$, 
and the representation $\pi_\E^\tau$
of $C_r^*(\Gamma,\sigma)$ on $\Hi_\E^\tau$ 
induces a representation $\pi_r$ of $C_r^*(\Gamma,\sigma)$ on $\Hip$.
By \Cref{proposition:extend_von_neumann}, this representation can be extended  to give $\Hip$ the structure of a Hilbert $\vN(\Gamma,\sigma)$-module, where the action is determined by 
$\lambda_\Gamma^\sigma(\gamma) \cdot \veco = \pi(\gamma) \veco$
for $\gamma \in \Gamma$ and $\veco \in \Hip$. The dimension $\dim_{(\vN(\Gamma,\sigma),\tau)} \Hip$ of this Hilbert $\vN(\Gamma,\sigma)$-module was computed in \cite{En21} to be $\vol(G/\Gamma) d_{\pi}$, see \cite[Theorem 4.3]{En21}. Therefore, by \Cref{prop:trace_dimension}, it follows that
$ \tau(\E) = \dim_{(\vN(\Gamma,\sigma),\tau)} \Hip =  \vol(G/\Gamma) d_{\pi}$,
as required.
\end{proof}

\begin{proof}[Proof of Theorem \ref{thm:module_intro}]
The statement of Theorem \ref{thm:module_intro} follows directly from \Cref{thm:trace} 
combined with the fact that $\pi \in \SIP$ and $\pker = Z$, cf.~Section  \ref{sec:relativeDS}.
\end{proof}

\begin{proof}[Proof of Theorem \ref{thm:intro_module_frame}]
Theorem \ref{thm:intro_module_frame} follows by applying \Cref{prop:generating_multiwindow} to the module of \Cref{thm:trace}.
\end{proof}

\subsection{Existence of smooth frames and Riesz sequences}
The following theorem is the main result of this paper.

\begin{theorem}\label{thm:existence}
Let $(\pi, \Hpi)$ be a $\sigma$-projective relative discrete series representation of a nilpotent Lie group $G$. 
Suppose $\Gamma \leq G$ is a lattice such that $(\Gamma,\sigma)$ satisfies Kleppner's condition. Then the following assertions hold:
\begin{enumerate}[(i)]
    \item If $ \vol(G/\Gamma) d_{\pi} < 1$, then there exists $\vect \in \Hip^\infty$ such that $\pi(\Gamma) \vect$ is a frame for $\Hip$.
    \item If $ \vol(G/\Gamma) d_{\pi} > 1$, then there exists $\vect \in \Hip^\infty$ such that $\pi(\Gamma) \vect$ is a Riesz sequence in $\Hip$.
\end{enumerate}
\end{theorem}

\begin{proof}
For the applicability of the results of Section \ref{sec:nilpotent}, we note that a discrete $\Gamma \leq G$ is finitely generated, see, e.g., \cite[Corollary 5.4.4]{corwin1990representations}.

(i) Suppose $\vol(G/\Gamma) d_{\pi} < 1$. By \Cref{thm:trace}, it follows that $\tau(\E) < 1$ for the canonical trace $\tau$ on $C_r^*(\Gamma,\sigma)$, which is the unique tracial state on $C_r^*(\Gamma,\sigma)$ by \cite{P}. Since $C_r^*(\Gamma,\sigma)$ has strict comparison of projections by \Cref{fg-nil4}, it follows from  \Cref{prop:trace_generators} that $\E$ admits a generating set with one element. By \Cref{prop:generating_isometry}, the generating element may be chosen to be $g \in \Hip^\infty$. Hence, $\pi(\Gamma) g$ is a frame for $\Hip$ by \Cref{prop:generating_multiwindow}.

(ii) Suppose $\vol(G/\Gamma) d_{\pi} > 1$. Just as in (i), we get $\tau(\E) > 1$ for the unique tracial state $\tau$ on $C_r^*(\Gamma,\sigma)$, so by strict comparison of projections and \Cref{prop:trace_generators}, $\E$ admits an $A$-linearly independent set $\{ \vect \}$ which is closed in $\E$. By \Cref{prop:generating_isometry}, $\vect$ can be chosen in $\Hs$, so $\pi(\Gamma) \vect$ is a Riesz sequence by \Cref{prop:generating_multiwindow}.
\end{proof}

\begin{proof}[Proof of Theorem \ref{thm:main_intro}]
The statement of Theorem \ref{thm:main_intro} follows directly from \Cref{thm:existence} 
combined with the fact that $\pi \in \SIP$ and $\pker = Z$, cf.~Section  \ref{sec:relativeDS}.
\end{proof}

\begin{remark}
\Cref{thm:existence} can be extended to multi-window and super systems, cf.~\cite{En21} for these notions. Under  Kleppner's condition, the inequality $\vol(G/\Gamma)d_{\pi} < n/d$ (resp.\ $\vol(G/\Gamma) d_{\pi} > n/d$) implies the existence of an $n$-multiwindow $d$-super frame (resp.\ Riesz sequence) in $\Hip^d$ with windows in $\Hip^\infty$.
\end{remark}

\subsection{Special classes of smooth vectors}
Theorem \ref{thm:existence} can also be used to prove the existence of frames and Riesz sequences generated by smooth vectors with additional qualities, such as G\aa rding vectors or analytic vectors for a representation $(\pi, \Hpi)$ of a Lie group $N$.

For $k \in C_c^{\infty} (N)$ and $g \in \Hpi$, a \emph{G\aa rding vector} is defined by
\begin{align} \label{eq:garding}
\pi(k) g = \int_N k(x) \pi(x) g \; d\mu_N (x).
\end{align}
The G\aa rding subspace $\Hpi^{\gamma} \subseteq \Hpi$ is the linear span of all vectors of the form \eqref{eq:garding}.  The space $\Hpi^{\gamma}$ is $\pi$-invariant and norm dense in $\Hpi$ and satisfies $\Hpi^{\gamma} \subseteq \Hs$, cf.~\cite[Appendix A]{corwin1990representations}. 

A vector $\vect \in \Hpi$ is called \emph{analytic} if the orbit map $x \mapsto \pi (x) \vect$ is real-analytic. The space of all analytic vectors is denoted by $\Hpi^{\omega}$ and is a $\pi$-invariant dense subspace of $\Hpi$, cf.~\cite{goodman1969analytic, nelson1959analytic}. 

The following modification result can be proved in a similar manner as \cite[Proposition 4.4]{grochenig2020balian} (cf.~also \cite[Proposition 1]{grochenig2013phase}). Its proof will be omitted here.

\begin{lemma} \label{lem:stability}
Let $\pi$ be an irreducible, square-integrable projective representation of a nilpotent Lie group $G$. For $g \in \Hs \setminus \{0\}$, let $\mathcal{H}_{\pi}^1 = \{ f \in \Hpi : C_{\vect} \veco \in L^1 (G) \}$ be equipped with the norm $\| \veco \|_{\Hpi^1} := \| C_{\vect} \veco \|_{L^1}$. 
Suppose that $\mathcal{V} \subset \Hpi^1$ is a norm dense subspace. Then the following assertions hold:
\begin{enumerate}[(i)]
    \item If $\pi (\Gamma) g$ is a frame, then there exists $\widetilde{g} \in \mathcal{V}$ such that $\pi (\Gamma) \widetilde{g}$ is a frame.
    \item If $\pi (\Gamma) g$ is a Riesz sequence, then there exists $\widetilde{g} \in \mathcal{V}$ such that $\pi (\Gamma) \widetilde{g}$ is a Riesz sequence.
\end{enumerate}
\end{lemma}

\begin{corollary} \label{cor:analytic}
Under the assumptions of \Cref{thm:existence}, the following hold:
\begin{enumerate}[(i)]
    \item If $\vol(G/\Gamma) d_{\pi} < 1$, there exists $\vect \in \Hip^{\omega}$ (resp.\ $ \vect \in \Hpi^{\gamma}$) such that $\pi(\Gamma) \vect$ is a frame.
    \item If $\vol(G/\Gamma) d_{\pi}  > 1$, there exists $\vect \in \Hip^{\omega}$ (resp.\ $\vect \in \Hpi^{\gamma}$) such that $\pi(\Gamma) \vect$ is a Riesz sequence.
\end{enumerate}
\end{corollary}
\begin{proof}
The result follows from Theorem \ref{thm:existence} and Lemma \ref{lem:stability} after showing that $\Hpi^{\omega}$ (resp.\ $\Hpi^{\gamma}$) is dense in $\Hpi^1$. To show the latter, let $h \in \Hpi^{\omega}$ (resp.\ $h \in \Hpi^{\gamma}$) be non-zero. By the atomic decomposition of $\Hpi^1$ (cf.~\cite{feichtinger1989banach,christensen1996atomic}), there exists a sequence $(x_i )_{i \in \mathbb{N}}$ in $G$ such that any $f \in \Hpi^1$ can be represented as a norm convergent series $f = \sum_{i \in I} c_i \pi(x_i) h$ for some $(c_i)_{i \in I} \in \ell^1 (I)$. If $f_n := \sum_{i = 1}^n c_i \pi(x_i) h$ for $n \in \mathbb{N}$, then $f_n \in \Hpi^{\omega}$ (resp.\ $f_n \in \Hpi^{\gamma}$), and $f_n \to f$ in $\Hpi^1$ as $n \to \infty$. This completes the proof.
\end{proof}

\begin{remark}
An alternative argument for the existence claims (i) and (ii) in Corollary \ref{cor:analytic} for the G\aa rding space $\Hpi^{\gamma}$ can be obtained 
via the Dixmier-Malliavin theorem \cite{dixmier1978factorisations}, which 
asserts that $\Hs = \Hpi^{\gamma}$ for a nilpotent Lie group. Then (i) and (ii) follow already from Theorem \ref{thm:existence}. 
\end{remark}

\section{Examples}\label{sec:examples}
In this section we discuss two examples that illustrate our main result.

\begin{example}[The Heisenberg group] \label{sec:heisenberg}

Let $N$ be the $2d+1$-dimensional Heisenberg group, i.e.,\ $N = \R^{d} \times \R^d \times \R$ with multiplication
\[ (x,\omega,s)(x',\omega',s') = (x+x',\omega+\omega',s+s' + x \cdot \omega') .\]
The center of $N$ is given by $Z = \{ 0 \} \times \{ 0 \} \times \R \cong \R$, hence the quotient $G = N/Z$ is isomorphic to the abelian group $\R^d \times \R^d$.

The Schrödinger representation of $N$ on $L^2(\R^d)$ is given by
\[ \pi(x,\omega,s) \veco(t) = e^{2\pi i s} e^{-2\pi i \omega t} \veco(t-x) .\]
The corresponding projective representation of $\R^{2d} \cong \R^d \times \R^d$ can be given by
\[ \pi(x,\omega) \veco(t) = e^{-2\pi i \omega t} \veco(t-x) \]
where the associated cocycle is given by $\sigma((x,\omega),(x',\omega')) = e^{-2\pi i x \cdot \omega'}$. A lattice orbit $\pi(\Gamma) \vect$ for $\Gamma$ a lattice in $\R^{2d}$ and $\vect \in L^2(\R^d)$ is in this context known as a \emph{Gabor system}.

A lattice $\Gamma$ in $\R^{2d}$ is of the form $\Gamma = M \Z^{2d}$ for some $M \in GL_{2n}(\R)$. Viewing instead $\Gamma$ as $\Z^{2d}$, the cocycle is given by $\sigma_{\Theta}(k,l) = e^{-2\pi i (\Theta k) \cdot l}$ for $k,l \in \Z^{2d}$, where $\Theta = M^t J M$ and $J$ denotes the standard symplectic $2n \times 2n$ matrix
\[ J = \begin{pmatrix} 0 & I_n \\ -I_n & 0 \end{pmatrix} . \]
With this notation, Kleppner's condition translates into the statement that whenever $k \in \Z^{2d}$ satisfies $e^{2\pi i (\Theta k) \cdot l} = 1$ for all $l \in \Z^{2d}$, then $k=0$. For a lattice of the form $\Gamma = \alpha \Z^{d} \times \beta \Z^d$ with $\alpha, \beta > 0$, this translates into the number $\alpha \beta$ being irrational. Kleppner's condition implies a weaker condition, namely that $\Theta$ contains at least one irrational entry. Let us call $\Gamma$ \emph{nonrational} when the latter condition holds. For nonrational lattices, Rieffel proved that the non-commutative tori $C^*(\Gamma,\sigma)$ have strict comparison of projections and cancellation \cite{Ri88} (see also \cite[Theorem 5.3.2]{Bla}). A consequence of this (cf.~\cite[Corollary 7.10]{Ri88}) was used in \cite[Theorem 5.4]{jakobsen2021duality}, combined with the link between Heisenberg modules over non-commutative tori and Gabor frames \cite{Lu09}, to prove the existence of Gabor frames $\pi(\Gamma)\vect$ with integrable vector $\vect \in \Hpi^1$ (hence, $\vect \in \mathcal{S} (\mathbb{R}^d)$ by Lemma \ref{lem:stability}) for nonrational lattices $\Gamma$ satisfying $\vol(\R^{2d}/\Gamma) < 1$. Therefore, in this setting, our main result is already covered by the result in \cite{jakobsen2021duality}.
\end{example}

The following example considers the group $G_{5,3}$ from Nielsen's catalogue \cite{Ni83}. This example is of interest to time-frequency analysis as it leads to so-called coorbit spaces \cite{feichtinger1989banach} that are different \cite[Example 3.3]{grochenig2021new} from the coorbit spaces associated to the Schrödinger representation defined in \Cref{sec:heisenberg}, so-called modulation spaces. In addition, we mention that group $C^*$-algebras associated with lattices in $G_{5,3}$ have been studied in \cite{MilWal2005}. 

\begin{example}[The group $G_{5,3}$]

Consider the group $G_{5,3}$ from \cite[p.\ 6]{Ni83}. This is a step 3 nilpotent Lie group with $\R^5$ as underlying manifold. The group operation is given by
\[ (x_1, \ldots ,x_5)(y_1, \ldots ,y_5) = (x_1 + y_1 + x_4 y_2 + x_5 y_3 + x_5^2 y_4 / 2, x_2 + y_2, x_3 + y_3 + x_5 y_4, x_4 + y_4 , x_5 + y_5) . \]
The center of $G_{5,3}$ is given by $\R \times \{ 0 \}^4$. An irreducible representation $(\pi,L^2(\R^2))$ of $G_{5,3}$ which is square-integrable modulo the center is given by
\[ \pi(x_1, \ldots, x_5) \veco(s,t) = e^{2\pi i( x_1 - x_2 x_4 + x_4 s - x_3 t + x_4 t^2/2)}\vect(s - x_2, t - x_5) .\]
The formal dimension of $\pi$ is $d_{\pi} = 1$. The quotient of $G_{5,3}$ by the center is isomorphic to $G \coloneqq \R \times N$, where $N$ denotes the $3$-dimensional Heisenberg group from the previous example, although the multiplication is in a different order:
\[ (x_1, x_2, x_3, x_4)(y_1, y_2, y_3, y_4) = (x_1 + y_1, x_2 + y_2 + x_4 y_3, x_3 + y_3, x_4 + y_4) .\]
Here we have relabeled the coordinates from $x_j$ to $x_{j-1}$ for $j=2,3,4,5$. The Haar measure $\mu_G$ on $G$ is just the 4-dimensional Lebesgue measure. The center $Z$ of $G$ is $\R^2 \times \{ 0 \}^2$. The corresponding projective representation of $\R \times N$ corresponding to $\pi$ (which we denote also by $\pi$) is given by
\[ \pi(x_1, \ldots, x_4) \veco(s,t) = e^{2\pi i (x_3 s - x_2 t + x_3 t^2/2)}\veco(s-x_1,t-x_4) .\]
The cocycle is given by
\[ \sigma((x_1,x_2,x_3,x_4),(y_1,y_2,y_3,y_4)) = \exp(2\pi i( - x_1 y_3 + x_4 y_2 + x_4^2 y_3/2)) .\]

Let $N' = N \cap \Z^3$ denote the discrete Heisenberg group, which is a cocompact lattice in the Heisenberg group. Hence the group $\Gamma = \Z \times N'$ is a lattice in $G$.
The conjugacy class of an element $(k_1,k_2,k_3,k_4) \in \Gamma$ is given by
\[ \{ (k_1,k_2 + p k_3 + q k_4,k_3,k_4 ) : p,q \in \Z \} .\]
From this we see that the elements with finite conjugacy class in $\Gamma$ are exactly elements of the center $Z \cap \Gamma = \Z^2 \times \{ 0 \}^2$, and these elements have singleton conjugacy classes.

Let us consider the dilation automorphisms $\delta_{\alpha,\beta}$ of $G$ ($\alpha,\beta > 0$) given by
\[ \delta_{\alpha,\beta}(x_1,x_2,x_3,x_4) = (\alpha x_1, \beta^2 x_2, \beta x_3, \beta x_4) .\]
 Applying these to $\Gamma \subseteq G$, we get a family of lattices $\Gamma_{\alpha, \beta}:=\delta_{\alpha,\beta}(\Gamma)$ in $G$
 which are isomorphic as discrete groups to $\Gamma$. 
We can compute the covolume of 
$\Gamma_{\alpha, \beta}$ as
\[
\vol(G/\Gamma_{\alpha, \beta})=
\mu_G(\delta_{\alpha,\beta}([0,1]^4)) = \mu_G([0,\alpha] \times [0,\beta^2] \times [0,\beta] \times [0,\beta]) = \alpha \beta^4 .\]
Let us check Kleppner's condition for 
$(\Gamma_{\alpha, \beta}, \sigma)$. We need only check the elements with finite conjugacy class, i.e.,\ those in the center of 
$\Gamma_{\alpha, \beta}$.
Thus, an element $(\alpha k_1, \beta^2 k_2, 0, 0)$ is $\sigma$-regular if and only if for all $(\alpha l_1, \beta^2 l_2, \beta l_3, \beta l_4) \in \Gamma_{\alpha, \beta}$ we have that
\begin{align*}
    1 &= \sigma((\alpha k_1, \beta^2 k_2, 0, 0),(\alpha l_1, \beta^2 l_2, \beta l_3, \beta l_4)) \overline{\sigma((\alpha l_1, \beta^2 l_2, \beta l_3, \beta l_4),(\alpha k_1, \beta^2 k_2, 0, 0))} \\
    &= \exp(-2\pi i (\alpha \beta k_1 l_3 + \beta^3 k_2 l_4)) .
\end{align*}
For this to happen, we need $\alpha \beta k_1 l_3 + \beta^3 k_2 l_4 \in \Z$ for all $l_3,l_4 \in \Z$. Hence, we see that if at least one of the numbers $\alpha \beta$ and $\beta^3$ is rational, then nontrivial $\sigma$-regular conjugacy classes exist. On the other hand, if both $\alpha \beta$ and $\beta^3$ are irrational, then Kleppner's condition is satisfied.

Our main result now states the following: Let $\alpha,\beta > 0$ such that both $\beta^3$ and $\alpha \beta$ are irrational numbers. If $\alpha \beta^4 < 1$ (resp.\ $\alpha \beta^4 > 1$), then there exists $\vect \in \Hip^\infty = \mathcal{S}(\R^2)$ such that 
$\pi(\Gamma_{\alpha, \beta}) \vect$
is a frame (resp.\ Riesz sequence) for $L^2(\R^2)$.
\end{example}

\section*{Acknowledgements}
U.E.\ gratefully acknowledges support from the The Research Council of Norway through project 314048. J.v.V. gratefully acknowledges support from the Research
Foundation - Flanders (FWO) Odysseus 1 grant G.0H94.18N and the Austrian Science Fund (FWF) project J-4445.

\bibliography{bibl}

\begin{thebibliography}{100}

\bibitem{aldroubi2008slanted}
A.~{Aldroubi}, A.~{Baskakov}, and I.~{Krishtal}.
\newblock {Slanted matrices, Banach frames, and sampling}.
\newblock {\em {J. Funct. Anal.}}, 255(7):1667--1691, 2008.

\bibitem{AnPo}
C.~Anantharaman and S.~Popa.
\newblock An introduction to {II$_1$} factors.
\newblock Draft, preliminary version, available at
  https://www.idpoisson.fr/anantharaman/publications/IIun.pdf, 2021.

\bibitem{Ani06}
P.~Aniello.
\newblock Square integrable projective representations and square integrable
  representations modulo a relatively central subgroup.
\newblock {\em Int. J. Geom. Methods Mod. Phys.}, 3(2):233--267, 2006.

\bibitem{APT}
P.~Ara, F.~Perera, and A.~S. Toms.
\newblock {$K$}-theory for operator algebras. {C}lassification of
  {$C^*$}-algebras.
\newblock In {\em Aspects of operator algebras and applications}, volume 534 of
  {\em Contemp. Math.}, pages 1--71. Amer. Math. Soc., Providence, RI, 2011.

\bibitem{Ar07}
L.~Aramba\v{s}i\'{c}.
\newblock On frames for countably generated {H}ilbert {$C^*$}-modules.
\newblock {\em Proc. Amer. Math. Soc.}, 135(2):469--478, 2007.

\bibitem{ascensi2014dilation}
G.~{Ascensi}, H.~G. {Feichtinger}, and N.~{Kaiblinger}.
\newblock {Dilation of the Weyl symbol and Balian-low theorem}.
\newblock {\em {Trans. Am. Math. Soc.}}, 366(7):3865--3880, 2014.

\bibitem{AuEn20}
A.~Austad and U.~Enstad.
\newblock Heisenberg modules as function spaces.
\newblock {\em J. Fourier Anal. Appl.}, 26(2):Paper No. 24, 28, 2020.

\bibitem{AuJaLu20}
A.~Austad, M.~S. Jakobsen, and F.~Luef.
\newblock Gabor duality theory for {M}orita equivalent {$C^*$}-algebras.
\newblock {\em Internat. J. Math.}, 31(10):2050073, 34, 2020.

\bibitem{balan2006density}
R.~{Balan}, P.~G. {Casazza}, C.~{Heil}, and Z.~{Landau}.
\newblock {Density, overcompleteness, and localization of frames. I: Theory}.
\newblock {\em {J. Fourier Anal. Appl.}}, 12(2):105--143, 2006.

\bibitem{BC}
E.~B\'{e}dos and R.~Conti.
\newblock Fourier theory and {$C^*$}-algebras.
\newblock {\em J. Geom. Phys.}, 105:2--24, 2016.

\bibitem{BO}
E.~B\'{e}dos and T.~Omland.
\newblock On reduced twisted group {$\rm C^*$}-algebras that are simple and/or
  have a unique trace.
\newblock {\em J. Noncommut. Geom.}, 12(3):947--996, 2018.

\bibitem{bekka2004square}
B.~{Bekka}.
\newblock {Square integrable representations, von Neumann algebras and an
  application to Gabor analysis}.
\newblock {\em {J. Fourier Anal. Appl.}}, 10(4):325--349, 2004.

\bibitem{beltita2011modulation}
I.~{Belti\c{t}\u{a}} and D.~{Belti\c{t}\u{a}}.
\newblock {Modulation spaces of symbols for representations of nilpotent Lie
  groups}.
\newblock {\em {J. Fourier Anal. Appl.}}, 17(2):290--319, 2011.

\bibitem{benedetto1998gabor}
J.~J. {Benedetto}, C.~{Heil}, and D.~F. {Walnut}.
\newblock {Gabor systems and the Balian-Low theorem}.
\newblock In {\em {Gabor analysis and algorithms. Theory and applications}},
  pages 85--122, 453--488. Boston, MA: Birkh\"auser, 1998.

\bibitem{BT}
F.~R. Beyl and J.~Tappe.
\newblock {\em Group extensions, representations, and the {S}chur
  multiplicator}, volume 958 of {\em Lecture Notes in Mathematics}.
\newblock Springer-Verlag, Berlin-New York, 1982.

\bibitem{Bla}
B.~Blackadar.
\newblock Comparison theory for simple {$C^*$}-algebras.
\newblock In {\em Operator algebras and applications, {V}ol. 1}, volume 135 of
  {\em London Math. Soc. Lecture Note Ser.}, pages 21--54. Cambridge Univ.
  Press, Cambridge, 1988.

\bibitem{Bla2006}
B.~Blackadar.
\newblock {\em Operator algebras}, volume 122 of {\em Encyclopaedia of
  Mathematical Sciences}.
\newblock Springer-Verlag, Berlin, 2006.
\newblock Theory of $C^*$-algebras and von Neumann algebras, Operator Algebras
  and Non-commutative Geometry, III.

\bibitem{BlHa82}
B.~Blackadar and D.~Handelman.
\newblock Dimension functions and traces on {$C^{\ast} $}-algebras.
\newblock {\em J. Functional Analysis}, 45(3):297--340, 1982.

\bibitem{BKR}
B.~Blackadar, A.~Kumjian, and M.~R{\o}rdam.
\newblock Approximately central matrix units and the structure of
  noncommutative tori.
\newblock {\em $K$-Theory}, 6(3):267--284, 1992.

\bibitem{BoChHeLi18}
C.~B\"{o}nicke, S.~Chakraborty, Z.~He, and H.-C. Liao.
\newblock Isomorphism and {M}orita equivalence classes for crossed products of
  irrational rotation algebras by cyclic subgroups of {$SL_2(\Bbb{Z})$}.
\newblock {\em J. Funct. Anal.}, 275(11):3208--3243, 2018.

\bibitem{christensen1996atomic}
O.~{Christensen}.
\newblock {Atomic decomposition via projective group representations}.
\newblock {\em {Rocky Mt. J. Math.}}, 26(4):1289--1312, 1996.

\bibitem{christensen2016introduction}
O.~{Christensen}.
\newblock {\em {An introduction to frames and Riesz bases}}.
\newblock Basel: Birkh\"auser/Springer, 2016.

\bibitem{Co80}
A.~Connes.
\newblock {$C^{\ast} $}-alg\`ebres et g\'{e}om\'{e}trie diff\'{e}rentielle.
\newblock {\em C. R. Acad. Sci. Paris S\'{e}r. A-B}, 290(13):A599--A604, 1980.

\bibitem{corwin1996lp}
L.~{Corwin} and C.~C. {Moore}.
\newblock {\(L^ p\) matrix coefficients for nilpotent Lie groups}.
\newblock {\em {Rocky Mt. J. Math.}}, 26(2):523--544, 1996.

\bibitem{corwin1990representations}
L.~J. {Corwin} and F.~P. {Greenleaf}.
\newblock {\em {Representations of nilpotent Lie groups and their applications.
  Part 1: Basic theory and examples}}, volume~18.
\newblock Cambridge etc.: Cambridge University Press, 1990.

\bibitem{Cun}
J.~Cuntz.
\newblock Dimension functions on simple {$C\sp*$}-algebras.
\newblock {\em Math. Ann.}, 233(2):145--153, 1978.

\bibitem{daubechies1990wavelet}
I.~{Daubechies}.
\newblock {The wavelet transform, time-frequency localization and signal
  analysis}.
\newblock {\em {IEEE Trans. Inf. Theory}}, 36(5):961--1005, 1990.

\bibitem{dixmier1978factorisations}
J.~{Dixmier} and P.~{Malliavin}.
\newblock {Factorisations de fonctions et de vecteurs indefiniment
  diff\'erentiables}.
\newblock {\em {Bull. Sci. Math., II. S\'er.}}, 102:305--330, 1978.

\bibitem{dutkay2013on}
D.~E. {Dutkay}, D.~{Han}, P.~E.~T. {Jorgensen}, and G.~{Picioroaga}.
\newblock {On common fundamental domains}.
\newblock {\em {Adv. Math.}}, 239:109--127, 2013.

\bibitem{EcLuPh10}
S.~Echterhoff, W.~L\"{u}ck, N.~C. Phillips, and S.~Walters.
\newblock The structure of crossed products of irrational rotation algebras by
  finite subgroups of {${\rm SL}_2(\Bbb Z)$}.
\newblock {\em J. Reine Angew. Math.}, 639:173--221, 2010.

\bibitem{EG}
C.~Eckhardt and E.~Gillaspy.
\newblock Irreducible representations of nilpotent groups generate classifiable
  {$C^*$}-algebras.
\newblock {\em M\"{u}nster J. Math.}, 9(1):253--261, 2016.

\bibitem{EGMK}
C.~Eckhardt, E.~Gillaspy, and P.~McKenney.
\newblock Finite decomposition rank for virtually nilpotent groups.
\newblock {\em Trans. Amer. Math. Soc.}, 371(6):3971--3994, 2019.

\bibitem{EMK}
C.~Eckhardt and P.~McKenney.
\newblock Finitely generated nilpotent group {$\rm C^*$}-algebras have finite
  nuclear dimension.
\newblock {\em J. Reine Angew. Math.}, 738:281--298, 2018.

\bibitem{En21}
U.~Enstad.
\newblock The density theorem for projective representations via twisted group
  von {N}eumann algebras.
\newblock {\em J. Math. Anal. Appl}, 511(2):126072, 2022.

\bibitem{jakobsen2020deformation}
U.~Enstad, M.~Jakobsen, F.~Luef, and T.~Omland.
\newblock Deformations of {G}abor frames on the adeles and other locally
  compact abelian groups.
\newblock Preprint, arXiv:2001.07080, 2020.

\bibitem{feichtinger1989banach}
H.~G. {Feichtinger} and K.~H. {Gr\"ochenig}.
\newblock {Banach spaces related to integrable group representations and their
  atomic decompositions. I}.
\newblock {\em {J. Funct. Anal.}}, 86(2):307--340, 1989.

\bibitem{feichtinger2004varying}
H.~G. {Feichtinger} and N.~{Kaiblinger}.
\newblock {Varying the time-frequency lattice of Gabor frames}.
\newblock {\em {Trans. Am. Math. Soc.}}, 356(5):2001--2023, 2004.

\bibitem{folland1989harmonic}
G.~B. {Folland}.
\newblock {\em {Harmonic analysis in phase space}}, volume 122.
\newblock Princeton, NJ: Princeton University Press, 1989.

\bibitem{folland2006abstruse}
G.~B. {Folland}.
\newblock {The abstruse meets the applicable: some aspects of time-frequency
  analysis}.
\newblock {\em {Proc. Indian Acad. Sci., Math. Sci.}}, 116(2):121--136, 2006.

\bibitem{folland2016course}
G.~B. Folland.
\newblock {\em A course in abstract harmonic analysis}.
\newblock Textb. Math. Boca Raton, FL: CRC Press, 2nd updated edition edition,
  2016.

\bibitem{fornasier2005intrinsic}
M.~{Fornasier} and K.~{Gr\"ochenig}.
\newblock {Intrinsic localization of frames}.
\newblock {\em {Constr. Approx.}}, 22(3):395--415, 2005.

\bibitem{FrLa02}
M.~Frank and D.~R. Larson.
\newblock Frames in {H}ilbert {$C^\ast$}-modules and {$C^\ast$}-algebras.
\newblock {\em J. Operator Theory}, 48(2):273--314, 2002.

\bibitem{fuehr2017density}
H.~{F\"uhr}, K.~{Gr\"ochenig}, A.~{Haimi}, A.~{Klotz}, and J.~L. {Romero}.
\newblock {Density of sampling and interpolation in reproducing kernel Hilbert
  spaces}.
\newblock {\em {J. Lond. Math. Soc., II. Ser.}}, 96(3):663--686, 2017.

\bibitem{goodman1969analytic}
R.~{Goodman}.
\newblock {Analytic and entire vectors for representations of Lie groups}.
\newblock {\em {Trans. Am. Math. Soc.}}, 143:55--76, 1969.

\bibitem{groechenig1991describing}
K.~{Gr\"ochenig}.
\newblock {Describing functions: Atomic decompositions versus frames}.
\newblock {\em {Monatsh. Math.}}, 112(1):1--42, 1991.

\bibitem{groechenig2001foundations}
K.~{Gr\"ochenig}.
\newblock {\em {Foundations of time-frequency analysis}}.
\newblock Boston, MA: Birkh\"auser, 2001.

\bibitem{grochenig2004localization}
K.~{Gr\"ochenig}.
\newblock {Localization of frames, Banach frames, and the invertibility of the
  frame operator}.
\newblock {\em {J. Fourier Anal. Appl.}}, 10(2):105--132, 2004.

\bibitem{grochenig2011multivariate}
K.~{Gr\"ochenig}.
\newblock {Multivariate Gabor frames and sampling of entire functions of
  several variables}.
\newblock {\em {Appl. Comput. Harmon. Anal.}}, 31(2):218--227, 2011.

\bibitem{grochenig2021new}
K.~Gr\"ochenig.
\newblock {New function spaces associated to representations of nilpotent Lie
  groups and generalized time-frequency analysis}.
\newblock {\em {J. Lie Theory}}, 31(3):659 -- 680, 2021.

\bibitem{grochenig2013phase}
K.~{Gr\"ochenig} and E.~{Malinnikova}.
\newblock {Phase space localization of Riesz bases for \(L^2(\mathbb{R}^d)\)}.
\newblock {\em {Rev. Mat. Iberoam.}}, 29(1):115--134, 2013.

\bibitem{grochenig2015deformation}
K.~{Gr\"ochenig}, J.~{Ortega-Cerd\`a}, and J.~L. {Romero}.
\newblock {Deformation of Gabor systems}.
\newblock {\em {Adv. Math.}}, 277:388--425, 2015.

\bibitem{grochenig2020balian}
K.~{Gr\"ochenig}, J.~L. {Romero}, D.~{Rottensteiner}, and J.~T. {Van
  Velthoven}.
\newblock {Balian-Low type theorems on homogeneous groups}.
\newblock {\em {Anal. Math.}}, 46(3):483--515, 2020.

\bibitem{grochenig2018sampling}
K.~{Gr\"ochenig}, J.~L. {Romero}, and J.~{St\"ockler}.
\newblock {Sampling theorems for shift-invariant spaces, Gabor frames, and
  totally positive functions}.
\newblock {\em {Invent. Math.}}, 211(3):1119--1148, 2018.

\bibitem{grochenig2018orthonormal}
K.~{Gr\"ochenig} and D.~{Rottensteiner}.
\newblock {Orthonormal bases in the orbit of square-integrable representations
  of nilpotent Lie groups}.
\newblock {\em {J. Funct. Anal.}}, 275(12):3338--3379, 2018.

\bibitem{grochenig2013gabor}
K.~{Gr\"ochenig} and J.~{St\"ockler}.
\newblock {Gabor frames and totally positive functions}.
\newblock {\em {Duke Math. J.}}, 162(6):1003--1031, 2013.

\bibitem{Ha14}
U.~Haagerup.
\newblock Quasitraces on exact {$C^*$}-algebras are traces.
\newblock {\em C. R. Math. Acad. Sci. Soc. R. Can.}, 36(2-3):67--92, 2014.

\bibitem{han2001lattice}
D.~{Han} and Y.~{Wang}.
\newblock {Lattice tiling and the Weyl-Heisenberg frames}.
\newblock {\em {Geom. Funct. Anal.}}, 11(4):742--758, 2001.

\bibitem{HNS}
S.~Hatui, E.~Narayanan, and P.~Singla.
\newblock On projective representations of finitely generated groups.
\newblock Preprint, arXiv:2006.02832, 2020.

\bibitem{heil2007history}
C.~{Heil}.
\newblock {History and evolution of the density theorem for Gabor frames}.
\newblock {\em {J. Fourier Anal. Appl.}}, 13(2):113--166, 2007.

\bibitem{HS}
P.~J. Hilton and U.~Stammbach.
\newblock {\em A course in homological algebra}, volume~4 of {\em Graduate
  Texts in Mathematics}.
\newblock Springer-Verlag, New York, second edition, 1997.

\bibitem{howe1977on}
R.~E. {Howe}.
\newblock {On a connection between nilpotent groups and oscillatory integrals
  associated to singularities}.
\newblock {\em {Pac. J. Math.}}, 73:329--363, 1977.

\bibitem{hulanicki1984functional}
A.~{Hulanicki}.
\newblock {A functional calculus for Rockland operators on nilpotent Lie
  groups}.
\newblock {\em {Stud. Math.}}, 78:253--266, 1984.

\bibitem{jakobsen2021duality}
M.~S. Jakobsen and F.~Luef.
\newblock {Duality of Gabor frames and Heisenberg modules}.
\newblock {\em {J. Noncommut. Geom.}}, 14(4):1445--1500, 2020.

\bibitem{janssen1994signal}
A.~J. E.~M. {Janssen}.
\newblock {Signal analytic proofs of two basic results on lattice expansions}.
\newblock {\em {Appl. Comput. Harmon. Anal.}}, 1(4):350--354, 1994.

\bibitem{ji1992smooth}
R.~{Ji}.
\newblock {Smooth dense subalgebras of reduced group \(C^*\)-algebras, Schwartz
  cohomology of groups, and cyclic cohomology}.
\newblock {\em {J. Funct. Anal.}}, 107(1):1--33, 1992.

\bibitem{JS}
X.~Jiang and H.~Su.
\newblock On a simple unital projectionless {$C^*$}-algebra.
\newblock {\em Amer. J. Math.}, 121(2):359--413, 1999.

\bibitem{jolissaint1990rapidly}
P.~{Jolissaint}.
\newblock {Rapidly decreasing functions in reduced \(C^*\)-algebras of groups}.
\newblock {\em {Trans. Am. Math. Soc.}}, 317(1):167--196, 1990.

\bibitem{JoSu97}
V.~Jones and V.~S. Sunder.
\newblock {\em Introduction to subfactors}, volume 234 of {\em London
  Mathematical Society Lecture Note Series}.
\newblock Cambridge University Press, Cambridge, 1997.

\bibitem{KW}
E.~Kirchberg and W.~Winter.
\newblock Covering dimension and quasidiagonality.
\newblock {\em Internat. J. Math.}, 15(1):63--85, 2004.

\bibitem{Kle62}
A.~Kleppner.
\newblock The structure of some induced representations.
\newblock {\em Duke Math. J.}, 29:555--572, 1962.

\bibitem{La95}
E.~C. Lance.
\newblock {\em Hilbert {$C^*$}-modules}, volume 210 of {\em London Mathematical
  Society Lecture Note Series}.
\newblock Cambridge University Press, Cambridge, 1995.
\newblock A toolkit for operator algebraists.

\bibitem{LeXi19}
H.~H. Lee and X.~Xiong.
\newblock Twisted {F}ourier(-{S}tieltjes) spaces and amenability.
\newblock Preprint, arXiv:1910.05888v1, 2019.

\bibitem{LeMo16}
M.~Lesch and H.~Moscovici.
\newblock Modular curvature and {M}orita equivalence.
\newblock {\em Geom. Funct. Anal.}, 26(3):818--873, 2016.

\bibitem{lueck1998dimension}
W.~{L\"uck}.
\newblock {Dimension theory of arbitrary modules over finite von Neumann
  algebras and \(L^2\)-Betti numbers. I: Foundations}.
\newblock {\em {J. Reine Angew. Math.}}, 495:135--162, 1998.

\bibitem{ludwig1988minimal}
J.~{Ludwig}.
\newblock {Minimal C-dense ideals and algebraically irreducible representations
  of the Schwartz-algebra of a nilpotent Lie group}.
\newblock {Harmonic analysis, Proc. Int. Symp., Luxembourg/Luxemb. 1987, Lect.
  Notes Math. 1359, 209-217 (1988).}, 1988.

\bibitem{ludwig1995algebre}
J.~{Ludwig} and C.~{Molitor-Braun}.
\newblock {Alg\`ebre de Schwartz d'un groupe de Lie nilpotent}.
\newblock In {\em {S\'eminaire de math\'ematique de Luxembourg}}, pages 25--67.
  Luxembourg: Centre Universitaire de Luxembourg, 1995.

\bibitem{Lu09}
F.~Luef.
\newblock Projective modules over noncommutative tori are multi-window {G}abor
  frames for modulation spaces.
\newblock {\em J. Funct. Anal.}, 257(6):1921--1946, 2009.

\bibitem{Lu11}
F.~Luef.
\newblock Projections in noncommutative tori and {G}abor frames.
\newblock {\em Proc. Amer. Math. Soc.}, 139(2):571--582, 2011.

\bibitem{lyubarskij1992frames}
Y.~I. {Lyubarskij}.
\newblock {Frames in the Bargmann space of entire functions}.
\newblock In {\em {Entire and subharmonic functions. Translation edited by A.
  B. Sossinsky}}, pages 167--180. Providence, RI: American Mathematical
  Society, 1992.

\bibitem{MS}
H.~Matui and Y.~Sato.
\newblock Strict comparison and {$\mathcal{Z}$}-absorption of nuclear
  {$C^*$}-algebras.
\newblock {\em Acta Math.}, 209(1):179--196, 2012.

\bibitem{MilWal2005}
P.~Milnes and S.~Walters.
\newblock Discrete cocompact-subgroups of {$G_{5,3}$} and related
  {$C^\ast$}-algebras.
\newblock {\em Rocky Mountain J. Math.}, 35(5):1765--1786, 2005.

\bibitem{mitkovski2020density}
M.~{Mitkovski} and A.~{Ramirez}.
\newblock {Density results for continuous frames}.
\newblock {\em {J. Fourier Anal. Appl.}}, 26(4):26, 2020.
\newblock Id/No 56.

\bibitem{Moore1964}
C.~C. Moore.
\newblock Extensions and low dimensional cohomology theory of locally compact
  groups. {II}.
\newblock {\em Trans. Amer. Math. Soc.}, 113:64--86, 1964.

\bibitem{moore1974square}
C.~C. {Moore} and J.~A. {Wolf}.
\newblock {Square integrable representations of nilpotent groups}.
\newblock {\em {Trans. Am. Math. Soc.}}, 185:445--462, 1974.

\bibitem{Mu90}
G.~J. Murphy.
\newblock {\em {$C^*$}-algebras and operator theory}.
\newblock Academic Press, Inc., Boston, MA, 1990.

\bibitem{nelson1959analytic}
E.~{Nelson}.
\newblock {Analytic vectors}.
\newblock {\em {Ann. Math. (2)}}, 70:572--615, 1959.

\bibitem{Ni83}
O.~A. Nielsen.
\newblock {\em Unitary representations and coadjoint orbits of low-dimensional
  nilpotent {L}ie groups}, volume~63 of {\em Queen's Papers in Pure and Applied
  Mathematics}.
\newblock Queen's University, Kingston, ON, 1983.

\bibitem{Om14}
T.~Omland.
\newblock Primeness and primitivity conditions for twisted group
  {$C^*$}-algebras.
\newblock {\em Math. Scand.}, 114(2):299--319, 2014.

\bibitem{Om15}
T.~Omland.
\newblock {$C^*$}-algebras generated by projective representations of free
  nilpotent groups.
\newblock {\em J. Operator Theory}, 73(1):3--25, 2015.

\bibitem{OsPhi2006}
H.~Osaka and N.~C. Phillips.
\newblock Furstenberg transformations on irrational rotation algebras.
\newblock {\em Ergodic Theory Dynam. Systems}, 26(5):1623--1651, 2006.

\bibitem{oussa2019compactly}
V.~{Oussa}.
\newblock {Compactly supported bounded frames on Lie groups}.
\newblock {\em {J. Funct. Anal.}}, 277(6):1718--1762, 2019.

\bibitem{P}
J.~A. Packer.
\newblock Twisted group {$C^*$}-algebras corresponding to nilpotent discrete
  groups.
\newblock {\em Math. Scand.}, 64(1):109--122, 1989.

\bibitem{Pac2}
J.~A. Packer.
\newblock Projective representations and the {M}ackey obstruction---a survey.
\newblock In {\em Group representations, ergodic theory, and mathematical
  physics: a tribute to {G}eorge {W}. {M}ackey}, volume 449 of {\em Contemp.
  Math.}, pages 345--378. Amer. Math. Soc., Providence, RI, 2008.

\bibitem{PR}
J.~A. Packer and I.~Raeburn.
\newblock Twisted crossed products of {$C^*$}-algebras.
\newblock {\em Math. Proc. Cambridge Philos. Soc.}, 106(2):293--311, 1989.

\bibitem{pedersen1994matrix}
N.~V. {Pedersen}.
\newblock {Matrix coefficients and a Weyl correspondence for nilpotent Lie
  groups}.
\newblock {\em {Invent. Math.}}, 118(1):1--36, 1994.

\bibitem{pfander2013remarks}
G.~E. {Pfander} and P.~{Rashkov}.
\newblock {Remarks on multivariate Gaussian Gabor frames}.
\newblock {\em {Monatsh. Math.}}, 172(2):179--187, 2013.

\bibitem{pfander2012geometric}
G.~E. {Pfander}, P.~{Rashkov}, and Y.~{Wang}.
\newblock {A geometric construction of tight multivariate Gabor frames with
  compactly supported smooth windows}.
\newblock {\em {J. Fourier Anal. Appl.}}, 18(2):223--239, 2012.

\bibitem{Pog}
D.~Poguntke.
\newblock Discrete nilpotent groups have a {$T_{1}$} primitive ideal space.
\newblock {\em Studia Math.}, 71(3):271--275, 1981/82.

\bibitem{RaWi98}
I.~Raeburn and D.~P. Williams.
\newblock {\em Morita equivalence and continuous-trace {$C^*$}-algebras},
  volume~60 of {\em Mathematical Surveys and Monographs}.
\newblock American Mathematical Society, Providence, RI, 1998.

\bibitem{raghunathan1972discrete}
M.~S. {Raghunathan}.
\newblock {\em {Discrete subgroups of Lie groups}}, volume~68.
\newblock Springer-Verlag, Berlin, 1972.

\bibitem{reiter2000classical}
H.~{Reiter} and J.~D. {Stegeman}.
\newblock {\em {Classical harmonic analysis and locally compact groups. 2nd
  ed}}, volume~22.
\newblock Oxford: Clarendon Press, 2nd ed. edition, 2000.

\bibitem{rieffel1981von}
M.~A. {Rieffel}.
\newblock {Von Neumann algebras associated with pairs of lattices in Lie
  groups}.
\newblock {\em {Math. Ann.}}, 257:403--418, 1981.

\bibitem{Ri85}
M.~A. Rieffel.
\newblock ``{V}ector bundles'' over higher-dimensional ``noncommutative tori''.
\newblock In {\em Operator algebras and their connections with topology and
  ergodic theory ({B}u\c{s}teni, 1983)}, volume 1132 of {\em Lecture Notes in
  Math.}, pages 456--467. Springer, Berlin, 1985.

\bibitem{Ri88}
M.~A. Rieffel.
\newblock Projective modules over higher-dimensional noncommutative tori.
\newblock {\em Canad. J. Math.}, 40(2):257--338, 1988.

\bibitem{rieffel2010vector}
M.~A. {Rieffel}.
\newblock {Vector bundles and Gromov-Hausdorff distance}.
\newblock {\em {J. \(K\)-Theory}}, 5(1):39--103, 2010.

\bibitem{romero2020density}
J.~L. Romero and J.~T. Van~Velthoven.
\newblock The density theorem for discrete series representations restricted to
  lattices.
\newblock {\em Expo. Math.}, To Appear.

\bibitem{romero2021dual}
J.~L. {Romero}, J.~T. {van Velthoven}, and F.~{Voigtlaender}.
\newblock {On dual molecules and convolution-dominated operators}.
\newblock {\em {J. Funct. Anal.}}, 280(10):56, 2021.
\newblock Id/No 108963.

\bibitem{Ror1}
M.~R{\o}rdam.
\newblock On the structure of simple {$C^*$}-algebras tensored with a
  {UHF}-algebra. {II}.
\newblock {\em J. Funct. Anal.}, 107(2):255--269, 1992.

\bibitem{Ror2}
M.~R{\o}rdam.
\newblock The stable and the real rank of {$\mathcal{Z}$}-absorbing
  {$C^*$}-algebras.
\newblock {\em Internat. J. Math.}, 15(10):1065--1084, 2004.

\bibitem{Ra98}
F.~R\u{a}dulescu.
\newblock The {$\Gamma$}-equivariant form of the {B}erezin quantization of the
  upper half plane.
\newblock {\em Mem. Amer. Math. Soc.}, 133(630):viii+70, 1998.

\bibitem{schweitzer1993dense}
L.~B. {Schweitzer}.
\newblock {Dense \(m\)-convex Fr\'echet subalgebras of operator algebra crossed
  products by Lie groups}.
\newblock {\em {Int. J. Math.}}, 4(4):601--673, 1993.

\bibitem{seip1992density}
K.~{Seip}.
\newblock {Density theorems for sampling and interpolation in the Bargmann-Fock
  space. I}.
\newblock {\em {J. Reine Angew. Math.}}, 429:91--106, 1992.

\bibitem{seip1992density2}
K.~{Seip} and R.~{Wallst\'en}.
\newblock {Density theorems for sampling and interpolation in the Bargmann-Fock
  space. II}.
\newblock {\em {J. Reine Angew. Math.}}, 429:107--113, 1992.

\bibitem{S}
U.~Stammbach.
\newblock {\em Homology in group theory}.
\newblock Lecture Notes in Mathematics, Vol. 359. Springer-Verlag, Berlin-New
  York, 1973.

\bibitem{Str}
K.~R. Strung.
\newblock {\em An introduction to {$C$}*-algebras and the classification
  program}.
\newblock Advanced Courses in Mathematics. CRM Barcelona.
  Birkh\"{a}user/Springer, Cham, [2021] \copyright 2021.
\newblock Edited by Francesc Perera.

\bibitem{Tak}
M.~Takesaki.
\newblock {\em Theory of operator algebras. {I}}, volume 124 of {\em
  Encyclopaedia of Mathematical Sciences}.
\newblock Springer-Verlag, Berlin, 2002.
\newblock Reprint of the first (1979) edition, Operator Algebras and
  Non-commutative Geometry, 5.

\bibitem{TWW}
A.~Tikuisis, S.~White, and W.~Winter.
\newblock Quasidiagonality of nuclear {$C^\ast$}-algebras.
\newblock {\em Ann. of Math. (2)}, 185(1):229--284, 2017.

\bibitem{varadarajan1986geometry}
V.~S. Varadarajan.
\newblock {\em Geometry of quantum theory}.
\newblock Springer-Verlag, New York, second edition, 1985.

\bibitem{Wa01}
S.~G. Walters.
\newblock {$K$}-theory of non-commutative spheres arising from the {F}ourier
  automorphism.
\newblock {\em Canad. J. Math.}, 53(3):631--672, 2001.

\bibitem{W}
W.~Winter.
\newblock Nuclear dimension and {$\mathcal{Z}$}-stability of pure {$\rm
  C^*$}-algebras.
\newblock {\em Invent. Math.}, 187(2):259--342, 2012.

\bibitem{young2001an}
R.~M. {Young}.
\newblock {\em {An introduction to nonharmonic Fourier series.}}
\newblock Orlando, FL: Academic Press, 2001.

\bibitem{Zel}
G.~Zeller-Meier.
\newblock Produits crois\'{e}s d'une {$C^{\ast} $}-alg\`ebre par un groupe
  d'automorphismes.
\newblock {\em J. Math. Pures Appl. (9)}, 47:101--239, 1968.

\end{thebibliography}
\bibliographystyle{abbrv}

\end{document}